\documentclass[11pt,final]{siamltex}
\usepackage[leqno]{amsmath}
\usepackage{graphicx}
\usepackage{subfigure}
 \usepackage{threeparttable}

\usepackage{bm}
\usepackage{color}
\usepackage{amssymb,version}
\usepackage{cases}
\newtheorem{rem}{Remark}[section]
\allowdisplaybreaks

\renewcommand\arraystretch{1.5}

    \title{On fully decoupled MSAV schemes  for the Cahn-Hilliard-Navier-Stokes model of Two-Phase Incompressible Flows
    \thanks{The work of X. Li is supported by the National Natural Science Foundation of China under grant number 11901489, 11971407 and China Postdoctoral Science Foundation under grant numbers BX20190187 and 2019M650152. The work of J. Shen is supported in part by NSF  DMS-2012585  and AFOSR FA9550-20-1-0309.}
}
    \author{ Xiaoli Li
        \thanks{School of Mathematical Sciences and Fujian Provincial Key Laboratory on Mathematical Modeling and High Performance Scientific Computing, Xiamen University, Xiamen, Fujian, 361005, China. Email: xiaolisdu@163.com}.
        \and Jie Shen 
         \thanks{Corresponding Author. Department of Mathematics, Purdue University, West Lafayette, IN 47907, USA. Email: shen7@purdue.edu}.
}

\begin{document}
\maketitle

\begin{abstract}
We construct first- and second-order time discretization schemes for the Cahn-Hilliard-Navier-Stokes system based on the multiple scalar auxiliary variables approach (MSAV) approach for gradient systems and (rotational) pressure-correction   for Navier-Stokes equations.  These schemes are linear, fully decoupled, unconditionally energy stable, and  only require solving a sequence of elliptic equations with constant coefficients at each time step. We carry out a rigorous error analysis for the first-order scheme, establishing  optimal convergence rate for all relevant functions in different norms. We also provide numerical experiments to verify our theoretical results.
\end{abstract}

 

\pagestyle{myheadings}
\thispagestyle{plain}
\markboth{XIAOLI LI AND JIE SHEN}{FULLY DECOUPLED MSAV SCHEMES}
 \section{Introduction} 
 We consider in this paper the construction and analysis of efficient time discretization  schemes for the following Cahn-Hilliard-Navier-Stokes system:
  \begin{subequations}\label{e_model}
    \begin{align}
    \frac{\partial \phi}{\partial t}+ (\textbf{u} \cdot \nabla ) \phi
    =M\Delta \mu  \quad &\ in\ \Omega\times J,
    \label{e_modelA}\\
    \mu=-\Delta \phi+   G^{\prime}(\phi) \quad &\ in\ \Omega\times J,
    \label{e_modelB}\\
     \frac{\partial \textbf{u}}{\partial t}+ \textbf{u}\cdot \nabla\textbf{u}
     -\nu\Delta\textbf{u}+\nabla p= \mu \nabla \phi
     \quad &\ in\ \Omega\times J,
      \label{e_modelC}\\
      \nabla\cdot\textbf{u}=0
      \quad &\ in\ \Omega\times J,
      \label{e_modelD}\\
       \frac{\partial \phi}{\partial \textbf{n}}= \frac{\partial \mu}{\partial \textbf{n}}=0,\
     \textbf{u}=\textbf{0}
      \quad &\ on\ \partial\Omega\times J,
      \label{e_modelE}
    \end{align}
  \end{subequations}
where $\displaystyle G(\phi)=\frac{1}{4\epsilon^2}(1-\phi^2)^2$ with $\epsilon$ representing the interfacial width, $M>0$ is the mobility constant, $\nu>0$ is the fluid viscosity. $\Omega$ is a bounded domain in $\mathbb{R}^2$ and  $J=(0,T]$.  The unknowns are the velocity $\textbf{u}$, the pressure $p$, the phase function $\phi$ and the chemical potential $\mu$. 
 We refer to \cite{hohenberg1977theory,gurtin1996two,MR1984386} for its physical interpretation and derivation as a phase-field model for the incompressible two phase flow with matching density (set to be $\rho_0=1$ for simplicity), and to \cite{MR2563636} for its mathematical analysis.
The above system satisfies the following energy dissipation law:
\begin{equation}\label{energy1}
 \frac{d E(\phi,\textbf{u})}{d t}=-M\|\nabla \mu\|^2-\nu\|\nabla\textbf{u}\|^2,
\end{equation}
 where $E(\phi,\textbf{u})=\int_\Omega\{\frac 12|\textbf{u}|^2 +\frac 12 |\nabla \phi|^2+G(\phi) \}d\textbf{x}$ is the total energy.

 For nonlinear dissipative systems such as the Navier-Stokes equation, Cahn-Hilliard equation and Cahn-Hilliard-Navier-Stokes system \eqref{e_model}, it is important that  numerical schemes  preserve a dissipative  energy law  at the discrete level. Various energy stable numerical methods have been proposed  in the last few decades for Navier-Stokes equations and for Cahn-Hilliard equations.
 A main difficulty in solving Navier-Stokes equation is the coupling of velocity and pressure by the incompressible condition $\nabla\cdot \textbf{u}=0$. A popular strategy is to use a  projection type method  pioneered by Chorin and Temam in late 60's \cite{Chor68,Tema69} which decouples the computation of pressure and velocity, we can refer to \cite{GMS06} for a review on various projection type methods for Navier-Stokes equations. The main issue in dealing with Cahn-Hilliard equation is how to treat the nonlinear term effectively so that the resulting discrete system can be efficiently solved while being energy stable. Popular approaches include the convex splitting \cite{eyre1998unconditionally}, stabilized semi-implicit \cite{shen2010numerical}, invariant energy quadratization (IEQ) \cite{yang2016linear}, and scalar auxiliary  variable (SAV) \cite{shen2019new}.
 We refer to \cite{Du.F19} (see also \cite{MR4132124}) for a up-to-date review on various methods for gradient flows which include in particular Cahn-Hilliard equation.

 On the other hand, it is much more challenging to develop efficient
 numerical schemes and to carry out corresponding error analysis for phase-field models such as \eqref{e_model} coupling   Navier-Stokes equations and Cahn-Hilliard equations. The system \eqref{e_model} is highly coupled nonlinear system whose dissipation law \eqref{energy1} relies on delicate cancellations of various nonlinear interactions. Usually,  energy stable  schemes for \eqref{e_model} are constructed using fully or weakly coupled fully implicit or partially implicit time discretization.
 Feng, He and Liu \cite{feng2007analysis} considered fully coupled first-order-in-time implicit  semi-discrete and fully discrete finite element schemes and established their  convergence results.
 Shen and Yang   constructed a sequence of weakly coupled \cite{shen2010phase} and full decoupled \cite{shen2015decoupled}, linear, first-order unconditionally energy stable schemes in time discretization for two-phase incompressible flows with same or different densities and viscosities with a modified double-well potential.  Gr\"un \cite{grun2013convergent} established an abstract convergence result of a fully discrete implicit scheme for a diffuse interface models of two-phase incompressible fluids with different densities.   Han and Wang \cite{MR3324579} constructed a coupled second-order energy stable scheme for the  Cahn-Hilliard-Navier-Stokes system based on convex splitting for the Cahn-Hilliard equation, a related fully discrete scheme is constructed in  \cite{MR3712284}  where second-order convergence in time is established. Han et al. \cite{MR3608328} developed a class of second-order energy stable schemes based on the IEQ approach. 
  Recently in \cite{Li2019SAV}, we constructed a second-order weakly-coupled, linear, energy stable SAV-MAC scheme  for the Cahn-Hilliard-Navier-Stokes equations, and established  second order convergence both in time and space for the simpler Cahn-Hilliard-Stokes equations. Note that in  all these works,   a coupled linear or nonlinear system with variable coefficients has to be solved at each time step.  We refer to the aforementioned papers for the references therein for other related work on this subject. 
  
We would like to point out that Yang and Dong \cite{MR3954060} developed    linear and  unconditionally energy-stable schemes for a more complicated phase-field model of two-phase incompressible flow with variable density, but the velocity and pressure are still coupled and it requires solving a nonlinear algebraic equation at each time step.
  To the best of our knowledge, despite a large number of work devoted to the construction and analysis for the Cahn-Hilliard-Navier-Stokes system \eqref{e_model}, there is still no fully decoupled, linear, second-order-in-time, unconditionally energy stable scheme, and there is no error analysis  for any fully decoupled  schemes for \eqref{e_model} as all previous analyses are for schemes which are either fully coupled or weakly coupled. In particular, it is highly nontrivial to 
establish  error estimates  for fully decoupled linear schemes due to additional difficulties arise from explicit treatment of nonlinear terms and the extra splitting error due to the decoupling of pressure from velocity.
   
The main purposes of this work are (i) to construct first- and second-order fully decoupled, linear and unconditionally energy stable schemes for \eqref{e_model}, and (ii) to carry out a rigorous error analysis. By using a combination of techniques in the multiple SAV approach \cite{MR3881243}, pressure-correction and rotational pressure-correction \cite{MR2059733} and a special SAV approach for Navier-Stokes equation \cite{Li2020new}, we are finally able to construct a
 fully decoupled, linear, second-order-in-time, unconditionally energy stable scheme for \eqref{e_model}. Furthermore, the schemes we constructed do not involve a nonlinear algebraic equation as in \cite{lin2019numerical,Li2020error}
 and lead to  bounds including the kinetic energy $\frac12\|u\|^2$ rather a positive SAV constant as an approximation to the kinetic energy as in \cite{lin2019numerical,Li2020error}. This turns out to be crucial in the error analysis. 
 More precisely, 
 the work presented in this paper for   \eqref{e_model} is unique in the following aspects: (i) we construct fully decoupled, unconditionally energy stable, first- and second-order linear schemes which  only require solving a sequence of elliptic equations with constant coefficients at each time; (ii) we establish rigorous first-order error estimates in time for all relevant functions in different norms without using an induction argument which often requires restriction on the time step. The key property is that our schemes lead to a uniform bounds on  the kinetic energy $\frac12\|u\|^2$. We believe that our second-order scheme is the first fully decoupled, linear, second-order-in-time, unconditionally energy stable scheme for \eqref{e_model}, and our error analysis is the 
  the first for any  linear and fully decoupled schemes  for \eqref{e_model} with explicit treatment of all nonlinear  terms.

The paper is organized as follows. In Section 2 we describe   some notations and useful inequalities. In Section 3 we construct the fully decoupled MSAV schemes, prove their unconditional energy stability, and describe an efficient   procedure for their implementation. 
In Section 4 we carry out error estimates for the first-order MSAV scheme for all functions except the pressure.   In Section 5, we present numerical experiments to verify the accuracy of the theoretical results. The error estimate for the pressure is derived in the appendix.

\section{Preliminaries} \label{sec:Notation}

We first introduce some standard notations. Let $L^m(\Omega)$ be the standard Banach space with norm
$$\| v\|_{L^m(\Omega)}=\left(\int_{\Omega}| v|^md\Omega\right)^{1/m}.$$
For simplicity, let
$$(f,g)=(f,g)_{L^2(\Omega)}=\int_{\Omega}fgd\Omega$$
denote the $L^2(\Omega)$ inner product,
 $\|v\|_{\infty}=\|v\|_{L^{\infty}(\Omega)}.$ And $W_p^k(\Omega)$ be the standard Sobolev space
$$W_p^k(\Omega)=\{g:~\| g\|_{W_p^k(\Omega)}<\infty\},$$
where
\begin{equation}\label{enorm1}
\| g\|_{W_p^k(\Omega)}=\left(\sum\limits_{|\alpha|\leq k}\| D^\alpha g\|_{L^p(\Omega)}^p \right)^{1/p}.
\end{equation}
By using  Poincar\'e inequality, we have
\begin{equation}\label{e_norm H1}
\aligned
\|\textbf{v}\|\leq c_1\|\nabla\textbf{v}\|, \ \forall \  \textbf{v} \in \textbf{H}^1_0(\Omega),
\endaligned
\end{equation}
where $c_1$ is a positive constant depending only on $\Omega$.

Define   
  $$\textbf{H}=\{ \textbf{v}\in \textbf{L}^2(\Omega): div\textbf{v}=0, \textbf{v}\cdot \textbf{n}|_{\Gamma}=0 \},\ \ \textbf{V}=\{\textbf{v}\in \textbf{H}^1_0(\Omega):  div\textbf{v}=0 \},$$
and the trilinear form $b(\cdot,\cdot,\cdot)$ by
\begin{equation*}
\aligned
b(\textbf{u},\textbf{v},\textbf{w})=\int_{\Omega}(\textbf{u}\cdot\nabla)\textbf{v}\cdot \textbf{w}d\textbf{x}.
\endaligned
\end{equation*}
We can easily obtain that the trilinear form $b(\cdot,\cdot,\cdot)$ is a skew-symmetric with respect to its last two arguments, i.e., 
\begin{equation}\label{e_skew-symmetric1}
\aligned
b(\textbf{u},\textbf{v},\textbf{w})=-b(\textbf{u},\textbf{w},\textbf{v}),\ \ \forall \ \textbf{u}\in \textbf{H}, \ \ \textbf{v}, \textbf{w}\in \textbf{H}^1_0(\Omega),
\endaligned
\end{equation}
and 
\begin{equation}\label{e_skew-symmetric2}
\aligned
b(\textbf{u},\textbf{v},\textbf{v})=0,\ \ \forall \ \textbf{u}\in \textbf{H}, \ \ \textbf{v}\in \textbf{H}^1_0(\Omega).
\endaligned
\end{equation}
By using a combination of integration by parts, Holder's inequality and Sobolev inequalities, we have \cite{Tema95}
\begin{flalign}\label{e_estimate for trilinear form}
b(\textbf{u},\textbf{v},\textbf{w})\leq \left\{
   \begin{array}{l}
   c_2\|\textbf{u}\|_1\|\textbf{v}\|_1\|\textbf{w}\|_1,\ \ \forall \ \textbf{u}, \textbf{v} \in \textbf{H}
   , \textbf{w}\in \textbf{H}^1_0(\Omega),\\
   c_2\|\textbf{u}\|_2\|\textbf{v}\|\|\textbf{w}\|_1, \ \ \forall \ \textbf{u}\in \textbf{H}^2(\Omega)\cap\textbf{H},\ \textbf{v} \in \textbf{H}, \textbf{w}\in \textbf{H}^1_0(\Omega),\\
   c_2\|\textbf{u}\|_2\|\textbf{v}\|_1\|\textbf{w}\|, \ \ \forall \ \textbf{u}\in \textbf{H}^2(\Omega)\cap\textbf{H},\ \textbf{v} \in \textbf{H}, \textbf{w}\in \textbf{H}^1_0(\Omega),\\
   c_2\|\textbf{u}\|_1\|\textbf{v}\|_2\|\textbf{w}\|, \ \ \forall \  \textbf{v}\in \textbf{H}^2(\Omega)\cap\textbf{H},\ \textbf{u}\in \textbf{H}, \textbf{w}\in \textbf{H}^1_0(\Omega),\\
   c_2\|\textbf{u}\|\|\textbf{v}\|_2\|\textbf{w}\|_1, \ \ \forall \ \textbf{v}\in \textbf{H}^2(\Omega)\cap\textbf{H},\ \textbf{u} \in \textbf{H}, \textbf{w} \in \textbf{H}^1_0(\Omega),\\
   c_2\|\textbf{u}\|_1^{1/2}\|\textbf{u}\|^{1/2}\|\textbf{v}\|_1^{1/2}\|\textbf{v}\|^{1/2}\|\textbf{w}\|_1,\ \ \forall \ \textbf{u}, \textbf{v} \in \textbf{H},  \textbf{w}\in \textbf{H}^1_0(\Omega),
   \end{array}
   \right.
\end{flalign}
where $c_2$ is a positive constant depending only on $\Omega$.

Let $P_{\textbf{H}}$ be the projection operator in $\textbf{L}^2(\Omega)$ onto $\textbf{H}$. We have (cf. (1.47) in \cite{temam2001navier})
\begin{equation}\label{PH}
 \|P_{\textbf{H}} u\|_1\le C(\Omega) \|u\|_1, \ \forall u\in  \textbf{H}^1(\Omega).
\end{equation}

We will frequently use the following discrete version of the Gronwall lemma \cite{shen1990long,HeSu07}:

\medskip
\begin{lemma} \label{lem: gronwall1}
Let $a_k$, $b_k$, $c_k$, $d_k$, $\gamma_k$, $\Delta t_k$ be nonnegative real numbers such that
\begin{equation}\label{e_Gronwall3}
\aligned
a_{k+1}-a_k+b_{k+1}\Delta t_{k+1}+c_{k+1}\Delta t_{k+1}-c_k\Delta t_k\leq a_kd_k\Delta t_k+\gamma_{k+1}\Delta t_{k+1}
\endaligned
\end{equation}
for all $0\leq k\leq m$. Then
 \begin{equation}\label{e_Gronwall4}
\aligned
a_{m+1}+\sum_{k=0}^{m+1}b_k\Delta t_k \leq \exp \left(\sum_{k=0}^md_k\Delta t_k \right)\{a_0+(b_0+c_0)\Delta t_0+\sum_{k=1}^{m+1}\gamma_k\Delta t_k \}.
\endaligned
\end{equation}
\end{lemma}

Throughout the paper we use $C$, with or without subscript, to denote a positive
constant, independent of discretization parameters, which could have different values at different places.

  \section{The MSAV schemes}
 In this section, we first reformulate the Cahn-Hilliad-Navier-Stokes system into an equivalent 
  system with multiple scalar auxiliary variables (MSAV). Then,  
   we construct first-order and second-order fully decoupled semi-discrete MSAV schemes, present a detail procedure to efficiently implement them, and prove that they are unconditionally energy stable.
 \subsection{MSAV reformulation} 
 Let $\gamma>0$ is a  positive constant  and set $F(\phi)=G(\phi)-\frac\gamma 2  \phi^2$ and  $E_1(\phi)=\int_\Omega  F(\phi) d\textbf{x} $. Here the term  $\frac\gamma 2  \phi^2$ is introduced to simplify the analysis 
  (cf. \cite{shen2018convergence}).
We introduce the following two scalar auxiliary variables 
  \begin{subequations}\label{e_definition_SAVs}
    \begin{align}
    &r(t)=\sqrt{E_1(\phi)+\delta},\ \ \forall \  \delta>\gamma, \label{e_definition_SAVsA} \\
    &q(t)=\rm{exp} (-\frac{t}{T}), \label{e_definition_SAVsB}
    \end{align}
  \end{subequations}
   and reformulate the  system (\ref{e_model})  as:
 \begin{subequations}\label{e_model_r}
    \begin{align}
    \frac{\partial \phi}{\partial t} + \frac{r}{\sqrt{E_1(\phi)+\delta}} (\textbf{u} \cdot \nabla ) \phi
    =M\Delta \mu  \quad &\ in\ \Omega\times J,
    \label{e_model_rA}\\
    \mu=- \Delta \phi + \gamma \phi + \frac{r}{\sqrt{E_1(\phi)+\delta}}F^{\prime}(\phi) \quad &\ in\ \Omega\times J,
    \label{e_model_rB}\\
        r_t=\frac{1 }{2\sqrt{E_1(\phi)+\delta}}\int_{\Omega}F^{\prime}(\phi)\phi_t d\textbf{x}\quad &\ in\ \Omega\times J,
    \label{e_model_rC}\\   
     \frac{\partial \textbf{u}}{\partial t}+ \exp( \frac{t}{T} ) q(t)\textbf{u}\cdot \nabla\textbf{u}-\nu\Delta\textbf{u}+\nabla p= \frac{r}{\sqrt{E_1(\phi)+\delta}} \mu \nabla \phi
     \quad &\ in\ \Omega\times J,
      \label{e_model_rD}\\
            \nabla\cdot\textbf{u}=0
      \quad &\ in\ \Omega\times J,
      \label{e_model_rF}\\
      \frac{\rm{d} q}{\rm{d} t}=-\frac{1}{T}q+ \exp( \frac{t}{T} ) \int_{\Omega}\textbf{u}\cdot \nabla\textbf{u}\cdot \textbf{u}d\textbf{x}
           \quad &\ in\ \Omega\times J.
      \label{e_model_rE}
 \end{align}
  \end{subequations}
 Since $\int_{\Omega}\textbf{u}\cdot \nabla\textbf{u}\cdot \textbf{u}d\textbf{x}=0$, 
 it is easy to see that, with $r(0)=\sqrt{E_1(\phi|_{t=0})+\delta}$ and $q(0)=1$, the above system is equivalent to the original system.
Taking the inner products of \eqref{e_model_rA} with $\mu$, \eqref{e_model_rB}  with $\frac{\partial \phi}{\partial t}$, \eqref{e_model_rD} with $\textbf{u}$, and multiplying \eqref{e_model_rC}  with $2 r$, summing up the results, we obtain the original energy law \eqref{energy1}. Furthermore, if we take the inner product of 
\eqref{e_model_rE} with $q$ and add it to the previous expression,   we obtain an equivalent dissipation law:
\begin{equation}\label{energy1_modify}
 \frac{d \tilde E(\phi,\textbf{u},r)}{d t}= -M\|\nabla \mu\|^2-\nu\|\nabla\textbf{u}\|^2-\frac{1}{T}q^2,
\end{equation}
where $\tilde E(\phi,\textbf{u},r,q)=\int_\Omega\frac 12\{|\textbf{u}|^2+\gamma \phi^2 + |\nabla \phi|^2\}d\textbf{x}+\frac{1}{2} q^2 + r^2$. We shall construct below efficient numerical schemes for the above system which are energy stable with respect to 
\eqref{energy1_modify}.
  
\subsection{A first-order scheme}
 We denote $$\Delta t=T/N,\ t^n=n\Delta t, \  d_t g^{n}=\frac{g^{n}-g^{n-1}}{\Delta t}, 
\ {\rm for} \ n\leq N. $$ 
Our   first-order  scheme for \eqref{e_model_r} is as follows: Find $(\phi^{n+1}, \mu^{n+1}, \tilde{\textbf{u}}^{n+1}, \textbf{u}^{n+1}, p^{n+1}, r^{n+1}, q^{n+1} )$ such that
\begin{align}
 &    \frac{\phi^{n+1}-\phi^n}{\Delta t} + \frac{r^{n+1}}{\sqrt{E_1( \phi^{n} )+\delta}} ( \textbf{u} ^{n} \cdot \nabla ) \phi^{n}
=M\Delta \mu^{n+1}, \label{e_model_semi_first_discrete1}\\
& \mu^{n+1}=- \Delta \phi^{n+1} + \gamma \phi^{n+1}+ 
 \frac{r^{n+1}}{\sqrt{E_1( \phi^{n})+\delta}}F^{\prime}( \phi^{n});\label{e_model_semi_first_discrete2}\\
& \frac{r^{n+1}-r^n}{\Delta t}= \frac{1}{2\sqrt{E_1( \phi^{n})+\delta}} \left( 
( F^{\prime}( \phi^{n}), \frac{\phi^{n+1}-\phi^n}{\Delta t} ) 
+ (\mu^{n+1}, \textbf{u}^n \cdot \nabla \phi^n) - ( \tilde{\textbf{u}}^{n+1}, \mu^n \nabla \phi^n ) \right);      \label{e_model_semi_first_discrete3} \\
& \frac{ \tilde{\textbf{u}}^{n+1}-\textbf{u}^{n}}{\Delta t}+\exp( \frac{ t^{n+1} } {T})q^{n+1} \textbf{u}^{n}\cdot \nabla\textbf{u}^{n}
     -\nu\Delta\tilde{\textbf{u}}^{n+1}     +\nabla p^{n}\label{e_model_semi_first_discrete4}\\
 &\hskip 1in   = \frac{r^{n+1}}{\sqrt{E_1( \phi^{n} )+\delta}}  \mu^n \nabla \phi^n, \ \ \tilde{\textbf{u}}^{n+1}|_{\partial \Omega}=0;   \notag \\
      & \nabla\cdot\textbf{u}^{n+1}=0, \ \ \textbf{u}^{n+1}\cdot \textbf{n}|_{\partial \Omega}=0,
 \label{e_model_semi_first_discrete7} \\
    & \frac{\textbf{u}^{n+1}-\tilde{\textbf{u}}^{n+1}}{\Delta t}+\nabla(p^{n+1}-p^n)=0; \label{e_model_semi_first_discrete6}\\
    & \frac{q^{n+1}-q^n}{\Delta t}=-\frac{1}{T}q^{n+1}+ \exp( \frac{t^{n+1}}{T} )
(\textbf{u}^n\cdot\nabla \textbf{u}^n,\tilde{\textbf{u}}^{n+1}).
 \label{e_model_semi_first_discrete5}
\end{align}
Note that we added the terms $(\mu^{n+1}, \textbf{u}^n \cdot \nabla \phi^n) - ( \tilde{\textbf{u}}^{n+1}, \mu^n \nabla \phi^n )$ in \eqref{e_model_semi_first_discrete3} which is a first-order approximation to
$(\mu,\textbf{u}\nabla \phi)-(\textbf{u},\mu\nabla \phi)=0$. On the other hand, \eqref{e_model_semi_first_discrete4}-\eqref{e_model_semi_first_discrete6} is a first-order  pressure-correction scheme \cite{GMS06} for \eqref{e_model_rD}-\eqref{e_model_rF}. Hence, the above scheme is first-order consistent to \eqref{e_model_r}.

\begin{rem}
 There are two main differences between the current scheme and the scheme in \cite{Li2019SAV} (and other schemes for \eqref{e_model}): 
 \begin{itemize}
 \item We employ a pressure-correction technique to decouple the computation of pressure and velocity.
 \item  We introduced two SAVs here instead of one in \cite{Li2019SAV}. The second SAV $q(t)$,  allows us to totally decouple the numerical scheme, as opposed to weakly coupled in  \cite{Li2019SAV}, as well as avoiding solving a nonlinear algebraic equation at each time step, which presents great challenge in establishing well-posedness and error estimates of the  scheme. 
 \end{itemize}
\end{rem}

\subsubsection{Efficient implementation}\label{Efficient implementation}
We  observe that the above scheme is linear but coupled. A  remarkable property  is that the scheme can be decoupled as we show below.  Denote 
\begin{equation}\label{e_efficient_implement_1}
\aligned
\xi^{n+1}_1= \frac{r^{n+1}}{\sqrt{E_1( \phi^{n} )+\delta}}, \ 
\xi^{n+1}_2= \exp ( \frac{t^{n+1}}{T}) q^{n+1} ,
\endaligned
\end{equation}
and set
  \begin{numcases} {} 
  \phi^{n+1}=\phi_0^{n+1}+\xi^{n+1}_1 \phi_1^{n+1}+ \xi^{n+1}_2 \phi_2^{n+1},   \notag \\
   \mu^{n+1}=\mu_0^{n+1}+\xi^{n+1}_1 \mu_1^{n+1}+ \xi^{n+1}_2 \mu_2^{n+1}, \notag \\  
  \tilde{\textbf{u}}^{n+1}=\tilde{\textbf{u}}_0^{n+1} + \xi^{n+1}_1 \tilde{\textbf{u}}_1^{n+1}
  + \xi^{n+1}_2 \tilde{\textbf{u}}_2^{n+1} , \label{e_efficient_implement_2} \\
   \textbf{u}^{n+1}=\textbf{u}_0^{n+1} + \xi^{n+1}_1 \textbf{u}_1^{n+1}
  + \xi^{n+1}_2 \textbf{u}_2^{n+1} , \notag \\ 
   p^{n+1}=p_0^{n+1} + \xi^{n+1}_1 p_1^{n+1} + \xi^{n+1}_2 p_2^{n+1}.  \notag 
\end{numcases}
Plugging \eqref{e_efficient_implement_2} in  \eqref{e_model_semi_first_discrete1}-\eqref{e_model_semi_first_discrete2} and \eqref{e_model_semi_first_discrete4}-\eqref{e_model_semi_first_discrete6}, and collecting  terms without $\xi^{n+1}_1,\xi^{n+1}_2$, with $\xi^{n+1}_1$ and with $\xi^{n+1}_2$, respectively,  we can obtain $\phi^{n+1}_i$, $\mu^{n+1}_i$, $ \textbf{u}_i^{n+1} $, $\tilde{\textbf{u}}_i^{n+1}$ and  $p^{n+1}_i$ $( i=0,1,2 )$ as follows.
 
\textbf{Step 1:} Find $(\phi^{n+1}_i, \mu^{n+1}_i)$ $(i=0,1,2)$ such that
\begin{align}
 &    \phi^{n+1}_0-\phi^n= M \Delta t  \Delta \mu^{n+1}_0, \ \  \mu^{n+1}_0 = - \Delta \phi^{n+1}_0 + \gamma \phi^{n+1}_0, \label{e_efficient_implement_3} \\
  &     \phi^{n+1}_1 + \Delta t( \textbf{u} ^{n} \cdot \nabla ) \phi^{n} = M \Delta t  \Delta \mu^{n+1}_1, \ \  \mu^{n+1}_1 = - \Delta \phi^{n+1}_1 + \gamma \phi^{n+1}_1 + F^{\prime}(\phi^n) , \label{e_efficient_implement_4} \\
   &     \phi^{n+1}_2 =  M \Delta t \Delta \mu^{n+1}_2, \ \  \mu^{n+1}_2 = - \Delta \phi^{n+1}_2 + \gamma \phi^{n+1}_2. \label{e_efficient_implement_5} 
\end{align}
We derive immediately from the last relation that $\phi^{n+1}_2 =0$, $\mu^{n+1}_2=0$. On the other hand, \eqref{e_efficient_implement_3} (resp. \eqref{e_efficient_implement_4}) is a coupled second-order system with constant coefficients in the same form as a simple semi-implicit scheme for the Cahn-Hilliard equation.

\textbf{Step 2:} Find $\tilde{\textbf{u}}^{n+1}_i$ $ (i=0,1,2) $ such that
\begin{align}
& \tilde{\textbf{u}}^{n+1}_0 - \textbf{u}^{n} - \nu \Delta t \Delta \tilde{\textbf{u}}^{n+1}_0 + \Delta t \nabla p^n=0,  \ \ \tilde{\textbf{u}}^{n+1}_0|_{\partial \Omega}=0, \label{e_efficient_implement_6}  \\
& \tilde{\textbf{u}}^{n+1}_1 - \nu \Delta t \Delta \tilde{\textbf{u}}^{n+1}_1 = \Delta t \mu^n \nabla \phi^n,  \ \ \tilde{\textbf{u}}^{n+1}_1 |_{\partial \Omega}=0, \label{e_efficient_implement_7}  \\
& \tilde{\textbf{u}}^{n+1}_2 +\Delta t  \textbf{u}^{n}\cdot \nabla\textbf{u}^{n} - \nu \Delta t \Delta \tilde{\textbf{u}}^{n+1}_2 = 0,  \ \ \tilde{\textbf{u}}^{n+1}_2 |_{\partial \Omega}=0. \label{e_efficient_implement_8}  
\end{align}
The above three systems are decoupled second-order equations with   same constant coefficients.

\textbf{Step 3:} Find $(\textbf{u}_i^{n+1}, p_i^{n+1})$  $ (i=0,1,2) $ such that
\begin{align}
    & \textbf{u}^{n+1}_0 - \tilde{\textbf{u}}^{n+1}_0+ \Delta t \nabla(p^{n+1}_0-p^n)=0, \ \ \nabla\cdot\textbf{u}^{n+1}_0 =0, \ \ \textbf{u}^{n+1}_0 \cdot \textbf{n}|_{\partial \Omega}=0 , \label{e_efficient_implement_9} \\
    & \textbf{u}^{n+1}_i - \tilde{\textbf{u}}^{n+1}_i+ \Delta t \nabla p^{n+1}_i =0, \ \ \nabla\cdot\textbf{u}^{n+1}_i =0, \ \ \textbf{u}^{n+1}_i \cdot \textbf{n}|_{\partial \Omega}=0 , \ \ i=1,2 . \label{e_efficient_implement_10} 
\end{align}
The above  systems correspond to the projection step in the pressure-correction scheme for Navier-Stokes equations. By taking the divergence operator on each of the above system, we find that $p_i^{n+1}$  $ (i=0,1,2) $ can be determined by solving a Poisson equation with homogeneous boundary conditions \cite{GMS06}, and then $\textbf{u}_i^{n+1}$  $ (i=0,1,2) $ can be obtained explicitly.

Once $\phi^{n+1}_i$, $\mu^{n+1}_i$, $ \textbf{u}_i^{n+1} $, $\tilde{\textbf{u}}_i^{n+1}$ and  $p^{n+1}_i$ $( i=0,1,2 )$ are known, we are now in position to determine $ \xi^{n+1}_1 $ and  $ \xi^{n+1}_2 $.  From \eqref{e_model_semi_first_discrete3} and \eqref{e_model_semi_first_discrete5}, we find that $ \xi^{n+1}_1 $ and  $ \xi^{n+1}_2 $ can be explicitly determined by solving the following $2\times 2$ linear algebraic system:
\begin{flalign} \label{e_efficient_implement_11}
  \begin{array}{l}
     A_1\xi_1+A_2\xi_2=A_0,  \\ 
     B_1\xi_1+B_2\xi_2=B_0, 
 \end{array}    
\end{flalign}
where
\begin{flalign}\label{e_efficient_implement_12}
\left\{
   \begin{array}{l}
    \displaystyle A_0= \frac{ r^n }{ \Delta t } + \frac{ 1 }{ 2 \sqrt{E_1( \phi^{n} )+\delta} } \left( ( F^{\prime}( \phi^{n}), \frac{\phi^{n+1}_0- \phi^n }{ \Delta t } ) 
+ (\mu^{n+1}_0 , \textbf{u}^n \cdot \nabla \phi^n ) - ( \tilde{\textbf{u}}^{n+1}_0 , \mu^n \nabla \phi^n ) \right), \\
   \displaystyle A_1= \frac{ \sqrt{E_1( \phi^{n} )+\delta} }{ \Delta t } - \frac{ 1 }{ 2 \sqrt{E_1( \phi^{n} )+\delta} } \left( ( F^{\prime}( \phi^{n}), \frac{\phi^{n+1}_1 }{ \Delta t } ) 
+ (\mu^{n+1}_1 , \textbf{u}^n \cdot \nabla \phi^n ) - ( \tilde{\textbf{u}}^{n+1}_1 , \mu^n \nabla \phi^n ) \right), \\
  \displaystyle  A_2= - \frac{ 1 }{ 2 \sqrt{E_1( \phi^{n} )+\delta} } \left(  F^{\prime}( \phi^{n}),
 - ( \tilde{\textbf{u}}^{n+1}_2 , \mu^n \nabla \phi^n ) \right), \\
 \displaystyle  B_0= \frac{ q^n }{ \Delta t } +  \exp( \frac{ t^{n+1} } {T} ) ( \textbf{u}^n\cdot\nabla \textbf{u}^n,\tilde{\textbf{u}}^{n+1}_0 ), \\
 \displaystyle  B_1= - \exp( \frac{ t^{n+1} } {T} ) ( \textbf{u}^n\cdot\nabla \textbf{u}^n,\tilde{\textbf{u}}^{n+1}_1 ), \\
 \displaystyle  B_2= \frac{ \exp( -\frac{ t^{n+1} } {T} ) }{ \Delta t } + \frac{  \exp( -\frac{ t^{n+1} } {T} ) }{ T }- \exp( \frac{ t^{n+1} } {T} )( \textbf{u}^n\cdot\nabla \textbf{u}^n,\tilde{\textbf{u}}^{n+1}_2 ) . 
   \end{array}
   \right.
\end{flalign}

In summary, at each time step, we only need to solve two coupled second-order systems with the same constant coefficients in \eqref{e_efficient_implement_3}-\eqref{e_efficient_implement_5}, and three Poisson-type equations in  \eqref{e_efficient_implement_6}-\eqref{e_efficient_implement_8}, and three Poisson equations \eqref{e_efficient_implement_9}-\eqref{e_efficient_implement_10}. 
Finally we can determine $\xi_i$ $(i=1,2)$  by solving a $2 \times 2$ linear algebraic system  \eqref{e_efficient_implement_11} with negligible computational cost. Hence, this scheme is very efficient and easy to implement.

\subsubsection{Energy stability}\label{Energy stability}
We show below that  the scheme  \eqref{e_model_semi_first_discrete1}-\eqref{e_model_semi_first_discrete5} is well posed, and despite the fact  that  nonlinear terms are treated explicitly, it is still unconditionally energy stable.

\medskip
\begin{theorem} \label{thm_energy_first order}
 The scheme \eqref{e_model_semi_first_discrete1}-\eqref{e_model_semi_first_discrete5}  is unconditionally energy stable in the sense that
 \begin{equation}\label{e_stable result_first}
 \aligned
 &\tilde E^{n+1}(\phi,\textbf{u},r,q)-\tilde E^n(\phi,\textbf{u},r,q)\leq -2M\Delta t\|\nabla \mu^{n+1}\|^2 \\
 & \ \ \ \ \ \ \ \ 
 -2  \nu\Delta t\|\nabla \tilde{\textbf{u}}^{n+1}\|^2-\frac{2 \Delta t} {T} |q^{n+1}|^2,
 \endaligned
\end{equation}
where 
\begin{equation*}
\aligned
& \tilde E^{n+1}(\phi,\textbf{u},r,q)=   \|\nabla \phi^{n+1} \|^2+\gamma  \|\phi^{n+1} \|^2+ 2  |r^{n+1}|^2 \\
 & \ \ \ \ \ \ 
 +  \| \textbf{u}^{n+1} \|^2 + (\Delta t)^2\|\nabla p^{n+1}\|^2 + |q^{n+1}|^2.
 \endaligned
 \end{equation*}
 Furthermore, it  admits a unique solution if the time step is sufficiently small. \end{theorem}

\begin{proof}
We first prove the unconditional energy stability \eqref{e_stable result_first}. Taking the inner products of \eqref{e_model_semi_first_discrete1} with $2\Delta t \mu^{n+1}$, \eqref{e_model_semi_first_discrete2} with $2( \phi^{n+1}-\phi^n )$ respectively and multiplying \eqref{e_model_semi_first_discrete3} with $4 \Delta t  r^{n+1}$ , we can obtain 
\begin{equation}\label{e_model_first_energy1}
\aligned
&  \|\nabla \phi^{n+1} \|^2- \|\nabla \phi^{n} \|^2 +\|\nabla \phi^{n+1}-\nabla \phi^n \|^2  +  \gamma ( \|\phi^{n+1} \|^2- \| \phi^{n} \|^2+  \| \phi^{n+1}- \phi^n \|^2 ) \\
&+2 (  | r^{n+1} |^2- | r^{n} |^2+  | r^{n+1}- r^n |^2 ) \\
=&-  \frac{2\Delta t r^{n+1} }{ \sqrt{E_1( \phi^{n})+\delta}}  ( \tilde{\textbf{u}}^{n+1}, \mu^n \nabla \phi^n ) -2M\Delta t\|\nabla \mu^{n+1}\|^2.
\endaligned
\end{equation}
Taking the inner product of \eqref{e_model_semi_first_discrete4} with $2\Delta t \tilde{\textbf{u}}^{n+1} $ leads to
\begin{equation}\label{e_model_first_energy4}
\aligned
&   \| \tilde{\textbf{u}}^{n+1} \|^2- \| \textbf{u}^{n} \|^2 + \| \tilde{\textbf{u}}^{n+1}- \textbf{u}^{n} \|^2 + 2 \Delta t  \exp( \frac{ t^{n+1} } {T}
)q^{n+1} (\textbf{u}^{n}\cdot \nabla\textbf{u}^{n}, \tilde{\textbf{u}}^{n+1} ) \\
& \ \ \ \ \ 
 +2 \nu \Delta t \| \nabla \tilde{\textbf{u}}^{n+1} \|^2+ 2 \Delta t ( \tilde{\textbf{u}}^{n+1}, \nabla p^n )= \frac{2 \Delta t r^{n+1} }{ \sqrt{E_1( \phi^{n})+\delta}}  ( \tilde{\textbf{u}}^{n+1}, \mu^n \nabla \phi^n ).
\endaligned
\end{equation}
Taking the inner product of \eqref{e_model_semi_first_discrete5} with $2 \Delta t  q^{n+1} $ gives
\begin{equation}\label{e_model_first_energy5}
\aligned
&   | q^{n+1} |^2- | q^{n} |^2 + | q^{n+1}- q^{n} |^2+ \frac{2 \Delta t} {T} |q^{n+1}|^2 \\
&\ \ \ \ \ \ \ \ = 2 \Delta t  \exp( \frac{ t^{n+1} } {T}
)q^{n+1} (\textbf{u}^{n}\cdot \nabla\textbf{u}^{n}, \tilde{\textbf{u}}^{n+1} ). 
\endaligned
\end{equation}
Recalling \eqref{e_model_semi_first_discrete6}, we have 
\begin{equation}\label{e_model_first_energy2}
\aligned
\textbf{u}^{n+1}+\Delta t\nabla p^{n+1}=\tilde{\textbf{u}}^{n+1}+\Delta t \nabla p^n.
\endaligned
\end{equation}
Taking the inner product of \eqref{e_model_first_energy2} with itself on both sides and noticing that $(\nabla p^{n+1},\textbf{u}^{n+1})=-(p^{n+1},\nabla\cdot \textbf{u}^{n+1})=0$, we have
\begin{equation}\label{e_model_first_energy3}
\aligned
\|\textbf{u}^{n+1}\|^2+(\Delta t)^2\|\nabla p^{n+1}\|^2=\|\tilde{\textbf{u}}^{n+1}\|^2+2\Delta t(\nabla p^n,\tilde{\textbf{u}}^{n+1})+(\Delta t)^2\|\nabla p^n\|^2.
\endaligned
\end{equation}
Combining \eqref{e_model_first_energy4} with \eqref{e_model_first_energy5} and \eqref{e_model_first_energy3} results in
\begin{equation}\label{e_model_first_energy6}
\aligned
&  \| \textbf{u}^{n+1} \|^2- \| \textbf{u}^{n} \|^2 + \| \tilde{\textbf{u}}^{n+1}- \textbf{u}^{n} \|^2+  (\Delta t)^2\|\nabla p^{n+1}\|^2- (\Delta t)^2\|\nabla p^{n}\|^2\\
& +  | q^{n+1} |^2- | q^{n} |^2 + | q^{n+1}- q^{n} |^2+ \frac{2 \Delta t} {T} |q^{n+1}|^2 \\
& \ \ \ \ \ 
+2 \nu \Delta t \| \nabla \tilde{\textbf{u}}^{n+1} \|^2 = \frac{2  \Delta t r^{n+1} }{ \sqrt{E_1( \phi^{n})+\delta}}  ( \tilde{\textbf{u}}^{n+1}, \mu^n \nabla \phi^n ).
\endaligned
\end{equation}
Thus we can obtain the desired result by combining \eqref{e_model_first_energy6} with \eqref{e_model_first_energy1}.

For the well posedness, we only have to prove that  $2 \times 2$ linear system \eqref{e_efficient_implement_11} has a unique solution, i.e. $A_1B_2-A_2B_1\neq 0$,  where $A_i,B_i$ $(i=1,2)$ are given in \eqref{e_efficient_implement_12}. 

We determine from  \eqref{e_efficient_implement_4} that
\begin{equation}\label{e_model_first_unique1}
\aligned
& \phi_1^{n+1}=\Delta t \mathbb{A}^{-1} ( M \Delta F^{\prime}(\phi^n) - \textbf{u} ^{n} \cdot \nabla  \phi^{n} ),
\endaligned
\end{equation}
with $\mathbb{A}=I+M\Delta t \Delta^2-M \gamma \Delta t \Delta$. Hence, we can rewrite the formula for $A_1$ in \eqref{e_efficient_implement_12} as
\begin{equation}\label{e_model_first_unique3}
\aligned
& A_1= \frac{ \sqrt{E_1( \phi^{n} )+\delta} }{ \Delta t } - \frac{ 1 }{ 2 \sqrt{E_1( \phi^{n} )+\delta} } \left(  F^{\prime}( \phi^{n}), \mathbb{A}^{-1} ( M \Delta F^{\prime}(\phi^n) - \textbf{u} ^{n} \cdot \nabla  \phi^{n} ) \right) \\
& \ \ \ \ \ \ 
+ \frac{ 1 }{ 2 \sqrt{E_1( \phi^{n} )+\delta} } \left( (\mu^{n+1}_1 , \textbf{u}^n \cdot \nabla \phi^n ) - ( \tilde{\textbf{u}}^{n+1}_1 , \mu^n \nabla \phi^n ) \right).
\endaligned
\end{equation} 
Therefore,  for $\Delta t$ sufficiently small, the first positive term in   $A_1B_2-A_2B_1$ is $ \frac{ \sqrt{E_1( \phi^{n} )+\delta}  }{ (\Delta t)^2 } \exp( -\frac{ t^{n+1} } {T} ) \sim O(\frac{1}{(\Delta t)^2})
$ is big enough to guarantee its positivity, which in turn ensures that
 \eqref{e_model_semi_first_discrete1}-\eqref{e_model_semi_first_discrete5} admits a unique solution.
\end{proof}

\medskip
 \subsection{A second-order scheme}
 By replacing first-order approximations in the scheme  \eqref{e_model_semi_first_discrete1}-\eqref{e_model_semi_first_discrete5} with second-order approximations, and using particularly  the second-order rotational pressure-correction scheme for Navier-Stokes equations, we can obtain a 
 second-order linear MSAV scheme  as follows: Find $(\phi^{n+1}, \mu^{n+1}, \tilde{\textbf{u}}^{n+1}, \textbf{u}^{n+1}, p^{n+1}, r^{n+1}, q^{n+1} )$ such that 
\begin{align}
 &    \frac{ 3\phi^{n+1}-4\phi^n+ \phi^{n-1} }{ 2\Delta t} + \frac{r^{n+1}}{\sqrt{E_1( \bar{\phi}^{n+1} )+\delta}} ( \bar{ \textbf{u} }^{n+1} \cdot \nabla ) \bar{\phi}^{n+1}
=M\Delta \mu^{n+1}, \label{e_model_semi_second_discrete1}\\
& \mu^{n+1}=-\Delta \phi^{n+1} +\gamma \phi^{n+1}+ 
\frac{r^{n+1}}{\sqrt{E_1(\bar{\phi}^{n+1})+\delta}}F^{\prime}(\bar{\phi}^{n+1});\label{e_model_semi_second_discrete2}\\
& \frac{3r^{n+1}-4r^n+r^{n-1}}{ 2\Delta t}=\frac{1}{2\sqrt{E_1( \bar{\phi}^{n+1})+\delta}} 
\left( (F^{\prime}(\bar{\phi}^{n+1}), \frac{3\phi^{n+1}-4\phi^n+\phi^{n-1}}{ 2\Delta t} )\right. \label{e_model_semi_second_discrete3}\\
& \hskip 1in
+\left. (\mu^{n+1}, \bar{\textbf{u}}^{n+1} \cdot \nabla \bar{\phi}^{n+1} ) - ( \tilde{\textbf{u}}^{n+1}, \bar{\mu}^{n+1} \nabla \bar{\phi}^{n+1} ) \right); \notag \\
&  \frac{3\tilde{\textbf{u}}^{n+1}-4\textbf{u}^{n}+\textbf{u}^{n-1} }{2\Delta t}+ \exp( \frac{ t^{n+1} } {T} )q^{n+1} \bar{\textbf{u}}^{n+1}\cdot \nabla \bar{ \textbf{u} }^{n+1} -\nu\Delta\tilde{\textbf{u}}^{n+1}     +\nabla p^{n}\label{e_model_semi_second_discrete4} \\
& \hskip 1in=
   \frac{r^{n+1}}{\sqrt{E_1( \bar{\phi}^{n+1} )+\delta}}  \bar{\mu}^{n+1} \nabla \bar{\phi}^{n+1} , \ \ \tilde{\textbf{u}}^{n+1}|_{\partial \Omega}=0;  \notag    \\
    & \frac{3\textbf{u}^{n+1}-3\tilde{\textbf{u}}^{n+1}}{2\Delta t}+\nabla(p^{n+1}-p^n+\nu \nabla \cdot \tilde{\textbf{u}}^{n+1})=0, \label{e_model_semi_second_discrete5}\\
    & \nabla\cdot\textbf{u}^{n+1}=0, \ \ \textbf{u}^{n+1}\cdot \textbf{n}|_{\partial \Omega}=0;  \label{e_model_semi_second_discrete6} \\
  & \frac{ 3q^{n+1}-4q^n+q^{n-1} }{ 2\Delta t}=-\frac{1}{T}q^{n+1}+ \exp( \frac{t^{n+1}}{T} )
( \bar{\textbf{u}}^{n+1}\cdot\nabla \bar{\textbf{u}}^{n+1},\tilde{\textbf{u}}^{n+1}); \label{e_model_semi_second_discrete7} 
\end{align}
where $\bar{g}^{n+1}=2g^n-g^{n-1} $ for any sequence $\{g^n\}$.

\subsubsection{Efficient implementation}
The above scheme can  be efficiently implemented as  the first-order scheme by solving a sequence of linear systems with constant coefficients. In fact, plugging  \eqref{e_efficient_implement_2} in  \eqref{e_model_semi_second_discrete1}-\eqref{e_model_semi_second_discrete2} and \eqref{e_model_semi_second_discrete4}-\eqref{e_model_semi_second_discrete6}, and collecting  terms without $\xi^{n+1}_1,\xi^{n+1}_2$, with $\xi^{n+1}_1$ and with $\xi^{n+1}_2$, respectively,  we can obtain, for each $i=0,1,2$, linear systems for $(\phi^{n+1}_i,\mu^{n+1}_i)$ similar to \eqref{e_efficient_implement_3}-\eqref{e_efficient_implement_5}, for $\tilde{\textbf{u}}_i^{n+1}$ similar to \eqref{e_efficient_implement_6}-\eqref{e_efficient_implement_8}, and for $(\textbf{u}_i^{n+1},p^{n+1}_i)$, the corresponding linear systems are
\begin{align}
    & \textbf{u}^{n+1}_0 - \tilde{\textbf{u}}^{n+1}_0+ \Delta t \nabla(p^{n+1}_0-p^n+\nu \nabla\cdot \tilde{\textbf{u}}^{n+1}_0)=0, \ \ \nabla\cdot\textbf{u}^{n+1}_0 =0, \ \ \textbf{u}^{n+1}_0 \cdot \textbf{n}|_{\partial \Omega}=0 , \label{e_efficient_implement_9_second} \\
    & \textbf{u}^{n+1}_i - \tilde{\textbf{u}}^{n+1}_i+ \Delta t (\nabla p^{n+1}_i +\nu \nabla\cdot \tilde{\textbf{u}}^{n+1}_i)=0, \ \ \nabla\cdot\textbf{u}^{n+1}_i =0, \ \ \textbf{u}^{n+1}_i \cdot \textbf{n}|_{\partial \Omega}=0 , \ \ i=1,2 . \label{e_efficient_implement_10_second} 
\end{align}
The above  systems correspond to the projection step in the rotational pressure-correction scheme \cite{GMS06} for Navier-Stokes equations, and can be  solved again by solving a Poisson equation with homogeneous boundary conditions \cite{GMS06}.

Once $\phi^{n+1}_i$, $\mu^{n+1}_i$, $ \textbf{u}_i^{n+1} $, $\tilde{\textbf{u}}_i^{n+1}$ and  $p^{n+1}_i$ $( i=0,1,2 )$ are known,  we can plug \eqref{e_efficient_implement_2} in \eqref{e_model_semi_second_discrete3} and \eqref{e_model_semi_second_discrete7} to form a $2\times 2$ linear algebraic system for $ \xi^{n+1}_1 $ and  $ \xi^{n+1}_2$.
We leave the detail to the interested readers.

\subsubsection{Energy stability}
The second-order scheme is also unconditionally energy stable as we show below.
\begin{theorem}
The scheme \eqref{e_model_semi_second_discrete1}-\eqref{e_model_semi_second_discrete7}  is unconditionally energy stable in the sense that
 \begin{equation}\label{e_model_second_energy1}
\aligned 
& \tilde E^{n+1}(\phi,\textbf{u},r,q)-\tilde E^n(\phi,\textbf{u},r,q)\leq -\nu\Delta t \| \nabla\tilde{\textbf{u}}^{n+1} \|^2-\nu\Delta t\| \nabla \times \textbf{u}^{n+1}\|^2 \\
&\ \ \ \ \ \ \ \ \ \ 
 -\frac{2 \Delta t}{T} |q^{n+1}|^2- 2M \Delta t \| \nabla \mu^{n+1} \|^2,
 \endaligned
\end{equation}
where 
\begin{equation*}
\aligned
 \tilde E^{n+1}(\phi,\textbf{u},r,q)= &\frac{1}{2} \| \textbf{u}^{n+1}\|^2+ \frac{1}{2} \| 2\textbf{u}^{n+1}-\textbf{u}^{n} \|^2+\frac{2}{3}(\Delta t)^2 \| \nabla H^{n+1}\|^2+\nu^{-1}\Delta t\|g^{n+1}\|^2\\
&+\frac{1}{2}  \| \nabla \phi^{n+1} \|^2+ \frac{1}{2} \| 2\nabla \phi^{n+1}-\nabla \phi^n \|^2 +\frac{\gamma}{2} \| \phi^{n+1} \|^2 \\
&+ \frac{\gamma}{2} \| 2 \phi^{n+1}- \phi^n \|^2 + |r^{n+1}|^2+ |2r^{n+1}-r^{n} |^2 \\
& + \frac{1}{2} |q^{n+1}|^2+\frac{1}{2} |2q^{n+1}-q^n|^2,
 \endaligned
 \end{equation*}
 where $\{g^k, H^k\}$ are defined by
\begin{equation}\label{e_model_second_energy6}
\aligned
g^0=0,\ \ g^{n+1}=\nu\nabla \cdot \tilde{\textbf{u}}^{n+1}+g^n,\; H^{n+1}=p^{n+1}+g^{n+1},\; \ n\geq 0.
\endaligned
\end{equation} 
 Furthermore, for $\Delta t$ sufficiently small, it admits a unique solution.
 \end{theorem}

\begin{proof}
Taking the inner products of \eqref{e_model_semi_second_discrete1} with $2\Delta t \mu^{n+1}$, \eqref{e_model_semi_second_discrete2} with $ (3\phi^{n+1}-4\phi^n+ \phi^{n-1})$ respectively and multiplying \eqref{e_model_semi_second_discrete3} with $4 \Delta t  r^{n+1}$, and using the identity
\begin{equation}\label{e_model_second_energy2}
\aligned
2(3a-4b+c,a)=|a|^2+|2a-b|^2-|b|^2-|2b-c|^2+|a-2b+c|^2,
\endaligned
\end{equation} 
we can obtain 
\begin{equation}\label{e_model_second_energy3}
\aligned
& \frac{1}{2} ( \| \nabla \phi^{n+1} \|^2+ \| 2\nabla \phi^{n+1}-\nabla \phi^n \|^2 - \| \nabla \phi^{n} \|^2- \| 2\nabla \phi^{n}-\nabla \phi^{n-1} \|^2 ) \\
&+\frac{\gamma}{2} ( \| \phi^{n+1} \|^2+ \| 2 \phi^{n+1}- \phi^n \|^2-  \| \phi^{n} \|^2- \| 2 \phi^{n}- \phi^{n-1} \|^2 ) \\
& + |r^{n+1}|^2+ |2r^{n+1}-r^{n} |^2- |r^{n}|^2- |2r^{n}-r^{n-1} |^2 \\
&+  \frac{1}{2} \| \nabla \phi^{n+1}-2\nabla \phi^n+ \nabla \phi^{n-1} \|^2 \\
&+\frac{\gamma}{2} \| \phi^{n+1}-2 \phi^n+  \phi^{n-1} \|^2 + | r^{n+1} -2 r^n+r^{n-1} |^2\\
=& -2M \Delta t \| \nabla \mu^{n+1} \|^2- \frac{2 \Delta t r^{n+1} }{ \sqrt{E_1( \bar{\phi}^{n+1})+\delta} } ( \tilde{\textbf{u}}^{n+1}, \bar{\mu}^{n+1} \nabla \bar{\phi}^{n+1} ). 
\endaligned
\end{equation}
Taking the inner product \eqref{e_model_semi_second_discrete4} with $2\Delta t\tilde{\textbf{u}}^{n+1}$ leads to 
\begin{equation}\label{e_model_second_energy4}
\aligned
&( 3\tilde{\textbf{u}}^{n+1}-4\textbf{u}^{n}+\textbf{u}^{n-1},\tilde{\textbf{u}}^{n+1} )+2\nu\Delta t \| \nabla\tilde{\textbf{u}}^{n+1} \|^2 \\
=&-2\Delta t \exp( \frac{t^{n+1}}{T} )q^{n+1} (\bar{\textbf{u}}^{n+1}\cdot \nabla \bar{ \textbf{u} }^{n+1}, \tilde{\textbf{u}}^{n+1} )-2\Delta t(\nabla p^{n}, \tilde{\textbf{u}}^{n+1})  \\
&+\frac{2 \Delta t r^{n+1}}{\sqrt{E_1( \bar{\phi}^{n+1} )+\delta}}  (\bar{\mu}^{n+1} \nabla \bar{\phi}^{n+1}, \tilde{\textbf{u}}^{n+1} ).
\endaligned
\end{equation} 
Recalling \eqref{e_model_semi_second_discrete5} and \eqref{e_model_second_energy2}, the first term on the left hand side of \eqref{e_model_second_energy4} can be transformed into
\begin{equation}\label{e_model_second_energy5}
\aligned
&( 3\tilde{\textbf{u}}^{n+1}-4\textbf{u}^{n}+\textbf{u}^{n-1},\tilde{\textbf{u}}^{n+1} )=
\left( 3( \tilde{\textbf{u}}^{n+1}-\textbf{u}^{n+1} )+3\textbf{u}^{n+1}-4\textbf{u}^{n}+\textbf{u}^{n-1},\tilde{\textbf{u}}^{n+1} \right) \\
&\ \ \ \ \
=3(\tilde{\textbf{u}}^{n+1}-\textbf{u}^{n+1}, \tilde{\textbf{u}}^{n+1})+( 3\textbf{u}^{n+1}-4\textbf{u}^{n}+\textbf{u}^{n-1},\textbf{u}^{n+1} )\\
&\ \ \ \ \ \ \ \ 
+( 3\textbf{u}^{n+1}-4\textbf{u}^{n}+\textbf{u}^{n-1},\tilde{\textbf{u}}^{n+1}-\textbf{u}^{n+1} ) \\
&\ \ \ \ \
=\frac{3}{2}( \|\tilde{\textbf{u}}^{n+1}\|^2- \| \textbf{u}^{n+1}\|^2+  \|\tilde{\textbf{u}}^{n+1}-\textbf{u}^{n+1}\|^2)+  \frac{1}{2} \| \textbf{u}^{n+1}\|^2+  \frac{1}{2} \| 2\textbf{u}^{n+1}-\textbf{u}^{n} \|^2\\
&\ \ \ \ \ \ \ \ 
- \frac{1}{2} \| \textbf{u}^{n}\|^2- \frac{1}{2} \| 2\textbf{u}^{n}-\textbf{u}^{n-1} \|^2+ \frac{1}{2} \| \textbf{u}^{n+1}-2\textbf{u}^{n}+\textbf{u}^{n-1} \|^2.
\endaligned
\end{equation} 
Thanks to \eqref{e_model_second_energy6}, we can recast \eqref{e_model_semi_second_discrete5} as
\begin{equation}\label{e_model_second_energy7}
\aligned
\sqrt{3}\textbf{u}^{n+1}+\frac{2}{\sqrt{3}}\Delta t \nabla H^{n+1}=\sqrt{3} \tilde{\textbf{u}}^{n+1}+\frac{2}{\sqrt{3}}\Delta t \nabla H^{n}.
\endaligned
\end{equation} 
Taking the inner product of \eqref{e_model_second_energy7} with itself on both sides, we have
\begin{equation*}\label{e_model_second_energy8}
\aligned
&3 \|\textbf{u}^{n+1} \|^2+\frac{4}{3}(\Delta t)^2 \| \nabla H^{n+1}\|^2\\
=&3 \| \tilde{\textbf{u}}^{n+1} \|^2+\frac{4}{3}(\Delta t)^2 \| \nabla H^{n}\|^2+4\Delta t(\tilde{\textbf{u}}^{n+1}, \nabla p^n )
+4\Delta t(\tilde{\textbf{u}}^{n+1}, \nabla g^n ).
\endaligned
\end{equation*} 
The last term on the right hand side  can be controlled by
\begin{equation}\label{e_model_second_energy9}
\aligned
4\Delta t(\tilde{\textbf{u}}^{n+1}, \nabla g^n )=&-4\nu^{-1}\Delta t(g^{n+1}-g^n,  g^n)\\
=&2\nu^{-1}\Delta t( \|g^n\|^2- \|g^{n+1}\|^2+ \|g^{n+1}-g^n\|^2)\\
=&2\nu^{-1}\Delta t\|g^n\|^2- 2\nu^{-1}\Delta t\|g^{n+1}\|^2+ 2\nu\Delta t\| \nabla\cdot \tilde{\textbf{u}}^{n+1}\|^2.
\endaligned
\end{equation} 
Thanks to the identity 
\begin{equation}\label{e_model_second_energy10}
\aligned
\| \nabla \times \textbf{v}\|^2+\| \nabla \cdot \textbf{v}\|^2=\| \nabla \textbf{v}\|^2,\ \ \forall \  \textbf{v} \in \textbf{H}^1_0(\Omega),
\endaligned
\end{equation} 
we have
\begin{equation}\label{e_model_second_energy11}
\aligned
&4\Delta t(\tilde{\textbf{u}}^{n+1}, \nabla g^n )
=2\nu^{-1}\Delta t\|g^n\|^2- 2\nu^{-1}\Delta t\|g^{n+1}\|^2\\
&\ \ \ \ \ \ 
+ 2\nu\Delta t\| \nabla \tilde{\textbf{u}}^{n+1}\|^2-2\nu\Delta t\| \nabla \times \textbf{u}^{n+1}\|^2.
\endaligned
\end{equation}
Then combining \eqref{e_model_second_energy4} with \eqref{e_model_second_energy5}-\eqref{e_model_second_energy11} results in
\begin{equation}\label{e_model_second_energy12}
\aligned
& \frac{1}{2} \| \textbf{u}^{n+1}\|^2+ \frac{1}{2} \| 2\textbf{u}^{n+1}-\textbf{u}^{n} \|^2+\frac{2}{3}(\Delta t)^2 \| \nabla H^{n+1}\|^2+\nu^{-1}\Delta t\|g^{n+1}\|^2\\
&+\frac{3}{2} \|\tilde{\textbf{u}}^{n+1}-\textbf{u}^{n+1}\|^2+\nu\Delta t \| \nabla\tilde{\textbf{u}}^{n+1} \|^2+\nu\Delta t\| \nabla \times \textbf{u}^{n+1}\|^2\\
\leq & \frac{1}{2} \| \textbf{u}^{n}\|^2+  \frac{1}{2}\| 2\textbf{u}^{n}-\textbf{u}^{n-1} \|^2+\frac{2}{3}(\Delta t)^2 \| \nabla H^{n}\|^2+\nu^{-1}\Delta t\|g^{n}\|^2\\
&-2\Delta t \exp( \frac{t^{n+1}}{T} )q^{n+1} (\bar{\textbf{u}}^{n+1}\cdot \nabla \bar{ \textbf{u} }^{n+1}, \tilde{\textbf{u}}^{n+1} ) \\
&+\frac{2 \Delta t r^{n+1}}{\sqrt{E_1( \bar{\phi}^{n+1} )+\delta}}  (\bar{\mu}^{n+1} \nabla \bar{\phi}^{n+1}, \tilde{\textbf{u}}^{n+1} ).
\endaligned
\end{equation} 
Multiplying \eqref{e_model_semi_second_discrete7} by $2\Delta t q^{n+1}$ and using \eqref{e_model_second_energy2}, we have
\begin{equation}\label{e_model_second_energy13}
\aligned
&\frac{1}{2} |q^{n+1}|^2+\frac{1}{2} |2q^{n+1}-q^n|^2-\frac{1}{2} |q^n|^2 -\frac{1}{2} |2q^{n}-q^{n-1}|^2+\frac{1}{2} |q^{n+1}-2q^n+q^{n-1}|^2\\
=&-\frac{2 \Delta t}{T} |q^{n+1}|^2+ 2\Delta t \exp( \frac{t^{n+1}}{T} )q^{n+1} (\bar{\textbf{u}}^{n+1}\cdot \nabla \bar{ \textbf{u} }^{n+1}, \tilde{\textbf{u}}^{n+1} ).
\endaligned
\end{equation}
Then combining \eqref{e_model_second_energy3} with \eqref{e_model_second_energy12} and \eqref{e_model_second_energy13} leads to 
\begin{equation*}\label{e_model_second_energy14}
\aligned
&  \frac{1}{2} \| \textbf{u}^{n+1}\|^2+ \frac{1}{2} \| 2\textbf{u}^{n+1}-\textbf{u}^{n} \|^2+\frac{2}{3}(\Delta t)^2 \| \nabla H^{n+1}\|^2+\nu^{-1}\Delta t\|g^{n+1}\|^2\\
&+\frac{1}{2}  \| \nabla \phi^{n+1} \|^2+ \frac{1}{2} \| 2\nabla \phi^{n+1}-\nabla \phi^n \|^2 +\frac{\gamma}{2} \| \phi^{n+1} \|^2 \\
&+ \frac{\gamma}{2} \| 2 \phi^{n+1}- \phi^n \|^2 + |r^{n+1}|^2+ |2r^{n+1}-r^{n} |^2+ \frac{1}{2} |q^{n+1}|^2+\frac{1}{2} |2q^{n+1}-q^n|^2 \\
&+  \frac{1}{2} \| \nabla \phi^{n+1}-2\nabla \phi^n+ \nabla \phi^{n-1} \|^2 +\frac{\gamma}{2} \| \phi^{n+1}-2 \phi^n+  \phi^{n-1} \|^2 \\
&+\frac{3}{2} \|\tilde{\textbf{u}}^{n+1}-\textbf{u}^{n+1}\|^2+\nu\Delta t \| \nabla\tilde{\textbf{u}}^{n+1} \|^2+\nu\Delta t\| \nabla \times \textbf{u}^{n+1}\|^2  + 2M \Delta t \| \nabla \mu^{n+1} \|^2 \\
& + | r^{n+1} -2 r^n+r^{n-1} |^2  +\frac{1}{2} |q^{n+1}-2q^n+q^{n-1}|^2 + \frac{2 \Delta t}{T} |q^{n+1}|^2 \\
= & \frac{1}{2} \| \textbf{u}^{n}\|^2+  \frac{1}{2}\| 2\textbf{u}^{n}-\textbf{u}^{n-1} \|^2+\frac{2}{3}(\Delta t)^2 \| \nabla H^{n}\|^2+\nu^{-1}\Delta t\|g^{n}\|^2\\
&+\frac{1}{2} \| \nabla \phi^{n} \|^2+ \frac{1}{2} \| 2\nabla \phi^{n}-\nabla \phi^{n-1} \|^2  +\frac{\gamma}{2}  \| \phi^{n} \|^2+ \frac{\gamma}{2}  \| 2 \phi^{n}- \phi^{n-1} \|^2  \\
& +|r^{n}|^2+ |2r^{n}-r^{n-1} |^2+\frac{1}{2} |q^n|^2 +\frac{1}{2} |2q^{n}-q^{n-1}|^2,
\endaligned
\end{equation*}
which  leads to the desired result \eqref{e_model_second_energy1}. 

By using the exactly the same procedure as for the first-order scheme \eqref{e_model_semi_first_discrete1}-\eqref{e_model_semi_first_discrete5}, we can show that, for $\Delta t$ sufficiently small,  the second-order scheme \eqref{e_model_semi_second_discrete1}-\eqref{e_model_semi_second_discrete7} admits a unique solution. 
\end{proof}

 \section{Error estimates} 

 In this section, we carry out  error analysis for the first-order semi-discrete scheme \eqref{e_model_semi_first_discrete1}-\eqref{e_model_semi_first_discrete5}. While in principle the error analysis for the second-order scheme \eqref{e_model_semi_second_discrete1}-\eqref{e_model_semi_second_discrete7} can be carried out  by combing the procedures below and those in \cite{Gue.S04} for the rotational pressure-correction scheme, but it will be much more involved and beyond the scope of this paper.
 
Since the scheme \eqref{e_model_semi_first_discrete1}-\eqref{e_model_semi_first_discrete5} is totally decoupled, it is much more difficult to carry out an error analysis as we have to deal with additional splitting errors due to the decoupling of pressure from the velocity as well as additional errors due to the explicit treatment of all nonlinear terms. However, the scheme avoids an essential difficulty associated with the nonlinear algebraic equation for the SAV in \cite{MR3954060,Li2020error}.

 The error estimates will be established through the help of  a series of intermediate results. 
 We shall first derive an $H^2$ bound for $\phi^n$ without assuming the Lipschitz condition on $F(\phi)$. A key ingredient is the following  
 stability result    \begin{equation}\label{e_error1}
  \aligned
  \| \textbf{u}^{n+1} \|^2&+\nu \Delta t \sum\limits_{k=0}^n \| \tilde{ \textbf{u} }^{k+1} \|_1^2 + \Delta t \sum\limits_{k=0}^n \| \nabla \mu^{k+1} \|^2 \\ 
 & +\| \phi^{n+1} \|_{H^1}^2+ |r^{n+1}|^2+ |q^{n+1}|^2 \leq K_1,
 \endaligned
\end{equation} 
 where the positive constant $K_1$ is dependent on $\textbf{u}^{0}$ and $\phi^0$, which can be derived from the unconditionally energy stability \eqref{e_stable result_first}.

  \medskip
\begin{lemma} \label{lem_phi_H2_boundness}
There exists a positive constant $K_2$ independent of $\Delta t$  such  that 
$$ \| \Delta \phi^{n+1} \|^2+ \| \mu^{n+1} \|^2 \leq K_2,  \  \forall \ 0\leq n \leq N-1. $$  
\end{lemma}
\begin{proof}
Combining \eqref{e_model_semi_first_discrete1} with \eqref{e_model_semi_first_discrete2} and taking inner product with $\Delta^2 \phi^{n+1}$ lead to
\begin{equation}\label{e_error2}
\aligned
\frac{1}{2 \Delta t}& ( \| \Delta \phi^{n+1} \|^2-\| \Delta \phi^{n} \|^2 +\| \Delta \phi^{n+1}- \Delta \phi^{n}\|^2 )+ M \| \Delta^2 \phi^{n+1} \|^2 +M \gamma \| \nabla \Delta \phi^{n+1} \|^2 \\
=&M \frac{r^{n+1}}{\sqrt{E_1(\phi^{n})+\delta}} ( \Delta F^{\prime}(\phi^{n}), \Delta^2 \phi^{n+1} )- \frac{r^{n+1}}{\sqrt{E_1( \phi^{n} )+\delta}}  (  \textbf{u} ^{n} \cdot \nabla  \phi^{n} , \Delta^2 \phi^{n+1} ).
\endaligned
\end{equation} 
Similarly to the estimate in \cite[Lemma 2.4]{shen2018convergence}, the first term on the right hand side of \eqref{e_error2} can be controlled by the following equation with the aid of \eqref{e_error1}:
\begin{equation}\label{e_error3}
\aligned
M \frac{r^{n+1}}{\sqrt{E_1(\phi^{n})+\delta}} ( \Delta F^{\prime}(\phi^{n}), \Delta^2 \phi^{n+1} ) 
\leq  &\frac{M }{4} \| \Delta^2 \phi^{n+1} \|^2 +C(K_1) \| \Delta F^{\prime}(\phi^{n}) \|^2 \\
\leq & \frac{M }{4} \| \Delta^2 \phi^{n+1} \|^2+  \frac{M }{2} \| \Delta^2 \phi^{n} \|^2+ C(K_1).
\endaligned
\end{equation} 
Using  \eqref{e_error1} and the following Sobolev inequality 
\begin{equation}\label{e_Soblev}
\aligned
& \| f\|_{L^4} \leq C \|f \|^{1/2}_{L^2} \|f\|_{H^1}^{1/2}, 
\endaligned
\end{equation}  
the last term on the right hand side of \eqref{e_error2} can be bounded by
\begin{equation}\label{e_error5}
\aligned
-\frac{r^{n+1}}{\sqrt{E_1( \phi^{n} )+\delta}} & (\textbf{u}^{n} \cdot \nabla \phi^n,  \Delta^2 \phi^{n+1}) \leq C \| \textbf{u}^{n}\|_{L^4} \| \nabla \phi^n \|_{L^4} \| \Delta^2 \phi^{n+1} \| \\
\leq &C  \| \textbf{u}^{n}\|^{1/2}  \| \textbf{u}^{n}\|_{H^1}^{1/2}  \| \nabla \phi^n \|^{1/2}
  \| \nabla \phi^n \|_{H^1}^{1/2} \| \Delta^2 \phi^{n+1} \| \\
 \leq & C \| \textbf{u}^{n}\|_{H^1}  \| \nabla \phi^n \|_{H^1}+ \frac{M}{16}\| \Delta^2 \phi^{n+1}\|^2 \\
 \leq &  \frac{M}{4}\| \Delta^2 \phi^{n+1}\|^2+  C \| \textbf{u}^{n}\|_{H^1}^2  (\| \Delta \phi^n \|^2+C(K_1) ).
  \endaligned
\end{equation} 
Combining \eqref{e_error2} with \eqref{e_error3}-\eqref{e_error5} leads to
\begin{equation}\label{e_error6}
\aligned
\frac{1}{2 \Delta t}& ( \| \Delta \phi^{n+1} \|^2-\| \Delta \phi^{n} \|^2 +\| \Delta \phi^{n+1}- \Delta \phi^{n}\|^2 )+ \frac M 2 \| \Delta^2 \phi^{n+1} \|^2 +M \gamma \| \nabla \Delta \phi^{n+1} \|^2 \\
&\leq  \frac{M }{2} \| \Delta^2 \phi^{n} \|^2 +  C \| \textbf{u}^{n}\|_{H^1}^2  (\| \Delta \phi^n \|^2+C(K_1) )+ C(K_1). 
\endaligned
\end{equation} 
Then multiplying \eqref{e_error6} by $2\Delta t$ and summing over  $n$, $n=0,1,2,\ldots,m$, $m\leq N-1$, we have 
\begin{equation}\label{e_error7}
\aligned
 \| \Delta \phi^{m+1} \|^2&+ M \Delta t  \| \Delta^2 \phi^{m+1} \|^2 
+ M \gamma \Delta t \sum\limits_{n=0}^{m} \| \nabla \Delta \phi^{n+1} \|^2 \\
\leq & \| \Delta \phi^{0} \|^2 + M \Delta t  \| \Delta^2 \phi^{0} \|^2 + C \Delta t \sum\limits_{n=0}^{m} \| \textbf{u}^{n}\|_{H^1}^2 \| \Delta \phi^n \|^2 + C(K_1),
\endaligned
\end{equation} 
which, together with   Lemma \ref{lem: gronwall1} and equations \eqref{e_model_semi_first_discrete2} and \eqref{e_error1}, lead to the desired result. 
\end{proof}
 
We now proceed with the error analysis.  For notational simplicity, we shall drop the dependence on $x$ for all functions when there is no confusion. Let $(\phi,\mu, \textbf{u}, p, r, q)$ be the exact solution of \eqref{e_definition_SAVs}, and $(\phi^{n+1},\mu^{n+1},\tilde{\textbf{u}}^{n+1}, \textbf{u}^{n+1},p^{n+1}, r^{n+1}, q^{n+1})$ be the solution of the scheme \eqref{e_model_semi_first_discrete1}-\eqref{e_model_semi_first_discrete5}, we
denote
   \begin{numcases}{}
\displaystyle \tilde{e}_{\textbf{u}}^{n+1}=\tilde{\textbf{u}}^{n+1}-\textbf{u}(t^{n+1}),\ \ 
\displaystyle e_{\textbf{u}}^{n+1}=\textbf{u}^{n+1}-\textbf{u}(t^{n+1}), \notag\\
\displaystyle e_{p}^{n+1}=p^{n+1}-p(t^{n+1}),\ \ \ 
\displaystyle e_{q}^{n+1}=q^{n+1}-q(t^{n+1}), \notag \\
\displaystyle e_{\phi}^{n+1}=\phi^{n+1}-\phi(t^{n+1}), \ \ \ 
\displaystyle e_{\mu}^{n+1}=\mu^{n+1}-\mu(t^{n+1}), \notag \\
\displaystyle e_{r}^{n+1}=r^{n+1}-r(t^{n+1}).
\end{numcases}

\medskip
\begin{lemma}\label{lem: error_estimate_phi}
Assuming $ \phi \in W^{3,\infty}(0,T; L^2(\Omega)) \bigcap W^{2,\infty}(0,T; H^2(\Omega))  $, $ \mu \in W^{1,\infty}(0,T; H^1(\Omega)) $ $ \bigcap L^{\infty}(0,T; H^2(\Omega)) $, $\textbf{u}\in W^{3,\infty}(0,T;\textbf{L}^2(\Omega))\bigcap W^{2,\infty}(0,T;\textbf{H}^1(\Omega))\bigcap L^{\infty}(0,T;\textbf{H}^2(\Omega))$, and $p \in $ \\
$W^{2,\infty}(0,T;H^1(\Omega))$, then for the first-order  scheme \eqref{e_model_semi_first_discrete1}-\eqref{e_model_semi_first_discrete5}, we have
\begin{equation*}
\aligned
\frac{1}{2\Delta t}&( \| \nabla e_{\phi}^{n+1} \|^2-  \| \nabla e_{\phi}^{n} \|^2 +  \| \nabla e_{\phi}^{n+1}- \nabla e_{\phi}^{n} \|^2 )  + \frac{M}{4} \| \nabla e_{\mu}^{n+1} \|^2 \\
& + \frac{ \gamma+1 }{2 \Delta t} ( \| e_{\phi}^{n+1}\|^2- \| e_{\phi}^{n}\|^2 +\| e_{\phi}^{n+1}-e_{\phi}^{n} \|^2 ) +\frac{M}{4} \| e_{\mu}^{n+1} \|^2 \\
&+ \frac{1}{\Delta t} ( |e_r^{n+1}|^2- |e_r^n|^2 +|e_r^{n+1}-e_r^{n}|^2 )  \\
\leq & C \| e_{\phi}^{n} \|^2 +C \| \nabla e_{\phi}^{n} \|^2 +C \| \nabla e_{\phi}^{n+1} \|^2 +C\| e_{ \textbf{u} }^n \|^2+C \| e_{\textbf{u}}^{n+1} \|^2 \\
&+ C (\Delta t)^2 \| \nabla(e_p^{n+1}-e_p^n) \|^2 + (C+ \frac{1}{4K_1}\| \nabla \mu^{n} \|^2) | e_r^{n+1} |^2  \\
&+ \frac{M}{8}  \| e_{\mu}^n \|^2 +\frac{M}{8} \| \nabla e_{\mu}^n \|^2 +  C (\Delta t)^2,  \ \ \ \forall \ 0\leq n\leq N-1,
\endaligned
\end{equation*}
where $C$ is a positive constant  independent of $\Delta t$.
\end{lemma}

\begin{proof}
Let $R_{\phi}^{n+1}$ be the truncation error defined by 
\begin{equation}\label{e_phi1}
\aligned
R_{\phi}^{n+1}=\frac{\partial \phi(t^{n+1})}{\partial t}- \frac{\phi(t^{n+1})-\phi(t^{n})}{\Delta t} =\frac{1}{\Delta t}\int_{t^n}^{t^{n+1}}(t^n-t)\frac{\partial^2 \phi }{\partial t^2}dt,
\endaligned
\end{equation}
and $E_{N}^{n+1}$ is defined by
\begin{equation}\label{e_EN}
\aligned
E_{N}^{n+1}=\frac{r(t^{n+1}) }{\sqrt{E_1(\phi(t^{n+1}))+\delta}} (\textbf{u}(t^{n+1}) \cdot \nabla ) \phi(t^{n+1})- \frac{r^{n+1}}{\sqrt{E_1( \phi^{n} )+\delta}} ( \textbf{u} ^{n} \cdot \nabla ) \phi^{n} .
\endaligned
\end{equation}
Subtracting \eqref{e_model_rA} at $t^{n+1}$ from \eqref{e_model_semi_first_discrete1}, we obtain
\begin{equation}\label{e_phi2}
\aligned
\frac{e_{\phi}^{n+1}-e_{\phi}^n}{ \Delta t } - & M \Delta e_{\mu}^{n+1} =  R_{\phi}^{n+1} +
E_{N}^{n+1} . 
\endaligned
\end{equation}
Taking the inner product of \eqref{e_phi2} with $e_{\mu}^{n+1}$ and $ e_{\phi}^{n+1} $, respectively, we obtain  
\begin{equation}\label{e_phi3}
\aligned
&( \frac{e_{\phi}^{n+1}-e_{\phi}^n}{ \Delta t },  e_{\mu}^{n+1}) + M \| \nabla e_{\mu}^{n+1} \|^2 = ( R_{\phi}^{n+1},  e_{\mu}^{n+1}) +  ( E_{N}^{n+1},  e_{\mu}^{n+1}) ,
\endaligned
\end{equation}
and 
\begin{equation}\label{e_phi4}
\aligned
 \frac{1}{2 \Delta t} ( \| e_{\phi}^{n+1} \|^2- &\| e_{\phi}^{n} \|^2+  \| e_{\phi}^{n+1}-e_{\phi}^{n} \|^2 ) \\
=& ( R_{\phi}^{n+1},  e_{\phi}^{n+1}) +  ( E_{N}^{n+1},  e_{\phi}^{n+1})
-M ( \nabla e_{\mu}^{n+1}, \nabla e_{\phi}^{n+1} ) .
\endaligned
\end{equation}
Let $E_{F}^{n+1}$ be defined by
\begin{equation}\label{e_EF}
\aligned
E_{F}^{n+1}=& r(t^{n+1}) ( \frac{ F^{\prime}( \phi^{n}) }{\sqrt{E_1( \phi^{n})+\delta}} -
 \frac{ F^{\prime}(\phi(t^{n+1})) }{ \sqrt{E_1(\phi(t^{n+1}) )+\delta} } ).
\endaligned
\end{equation}
Subtracting \eqref{e_model_rB} at $t^{n+1}$ from \eqref{e_model_semi_first_discrete2}, we obtain
\begin{equation}\label{e_phi5}
\aligned
 e_{\mu}^{n+1}= & - \Delta e_{\phi}^{n+1} +   \gamma e_{\phi}^{n+1}+ \frac{e_r^{n+1}}{ \sqrt{E_1( \phi^{n})+\delta} }F^{\prime}( \phi^{n}) + E_{F}^{n+1}.
\endaligned
\end{equation}
Taking the inner product of \eqref{e_phi5} with $ M e_{\mu}^{n+1}$ and $ \frac{ e_{\phi}^{n+1}- e_{\phi}^{n} } {\Delta t} $, respectively, we obtain  
\begin{equation}\label{e_phi6}
\aligned
 M \| e_{\mu}^{n+1} \|^2= &M ( \nabla e_{\mu}^{n+1}, \nabla e_{\phi}^{n+1} ) + M \gamma (e_{\phi}^{n+1}, e_{\mu}^{n+1} ) \\
&  + M \frac{e_r^{n+1}}{ \sqrt{E_1( \phi^{n})+\delta} } ( F^{\prime}( \phi^{n}), e_{\mu}^{n+1} )
+M \left(  E_{F}^{n+1} , e_{\mu}^{n+1} \right),
\endaligned
\end{equation}
and 
\begin{equation}\label{e_phi7}
\aligned
 (\frac{ e_{\phi}^{n+1}- e_{\phi}^{n} } {\Delta t}, e_{\mu}^{n+1} )= &\frac{1}{2\Delta t} (
 \| \nabla e_{\phi}^{n+1} \|^2-  \| \nabla e_{\phi}^{n} \|^2 +  \| \nabla e_{\phi}^{n+1}- \nabla e_{\phi}^{n} \|^2 ) \\
 &+ \frac{e_r^{n+1}}{ \sqrt{E_1( \phi^{n})+\delta} } ( F^{\prime}( \phi^{n}), \frac{ e_{\phi}^{n+1}- e_{\phi}^{n} } {\Delta t} )+ \left(  E_{F}^{n+1} , \frac{ e_{\phi}^{n+1}- e_{\phi}^{n} } {\Delta t} \right)  \\
 &  + \frac{\gamma}{2 \Delta t} ( \| e_{\phi}^{n+1}\|^2- \| e_{\phi}^{n}\|^2 +\| e_{\phi}^{n+1}-e_{\phi}^{n} \|^2 ).
\endaligned
\end{equation}
Combining \eqref{e_phi3} with \eqref{e_phi4}-\eqref{e_phi7}, we have
\begin{equation}\label{e_phi8}
\aligned
\frac{1}{2\Delta t}& ( \| \nabla e_{\phi}^{n+1} \|^2-  \| \nabla e_{\phi}^{n} \|^2 +  \| \nabla e_{\phi}^{n+1}- \nabla e_{\phi}^{n} \|^2 )  + M \| \nabla e_{\mu}^{n+1} \|^2 \\
& \ \ \ + \frac{ \gamma+1 }{2 \Delta t} ( \| e_{\phi}^{n+1}\|^2- \| e_{\phi}^{n}\|^2 +\| e_{\phi}^{n+1}-e_{\phi}^{n} \|^2 ) +M \| e_{\mu}^{n+1} \|^2 \\
=& -\frac{e_r^{n+1}}{ \sqrt{E_1( \phi^{n})+\delta} } ( F^{\prime}( \phi^{n}), \frac{ e_{\phi}^{n+1}- e_{\phi}^{n} } {\Delta t} ) - M \frac{e_r^{n+1}}{ \sqrt{E_1( \phi^{n})+\delta} } ( F^{\prime}( \phi^{n}), e_{\mu}^{n+1} ) \\
 & + \left( E_{F}^{n+1} , M e_{\mu}^{n+1}- \frac{ e_{\phi}^{n+1}- e_{\phi}^{n} } {\Delta t} \right)  +  ( E_{N}^{n+1},  e_{\mu}^{n+1}+ e_{\phi}^{n+1} ) \\
& +( R_{\phi}^{n+1},  e_{\mu}^{n+1}) + ( R_{\phi}^{n+1},  e_{\phi}^{n+1}) + M \gamma (e_{\phi}^{n+1}, e_{\mu}^{n+1} ) .
\endaligned
\end{equation}
Using the Cauchy-Schwarz inequality, the second term on the right hand side of \eqref{e_phi8} can be recast as 
\begin{equation}\label{e_phi8_second}
\aligned
- M \frac{e_r^{n+1}}{ \sqrt{E_1( \phi^{n})+\delta} } ( F^{\prime}( \phi^{n}), e_{\mu}^{n+1} ) 
\leq & C | e_r^{n+1} |^2+\frac{M}{4} \| e_{\mu}^{n+1} \|^2.
\endaligned
\end{equation}
Recalling \eqref{e_phi2}, the third term on the right hand side of \eqref{e_phi8} can be written as
\begin{equation}\label{e_phi9}
\aligned
 \left( E_{F}^{n+1} , M e_{\mu}^{n+1}- \frac{ e_{\phi}^{n+1}- e_{\phi}^{n} } {\Delta t} \right)  
=& ( E_{F}^{n+1} , M e_{\mu}^{n+1}- M \Delta e_{\mu}^{n+1}-R_{\phi}^{n+1} -
E_{N}^{n+1} ) \\
=& ( E_{F}^{n+1} , M e_{\mu}^{n+1}-R_{\phi}^{n+1} -
E_{N}^{n+1} ) + M ( \nabla E_{F}^{n+1} ,  \nabla e_{\mu}^{n+1} ) .
\endaligned
\end{equation}
We now estimate the terms on the right hand side as follows:  Since
\begin{equation*}
\aligned
 E_{F}^{n+1}= &r(t^{n+1}) ( \frac{ F^{\prime}( \phi^{n}) }{\sqrt{E_1( \phi^{n})+\delta}} -
 \frac{ F^{\prime}(\phi(t^{n+1})) }{ \sqrt{E_1(\phi(t^{n+1}) )+\delta} } ) \\
 =&    \frac{ r(t^{n+1}) F^{\prime}(\phi(t^{n+1}) ) (E_1(\phi(t^{n+1}) )- E_1(\phi^n ) ) }{ \sqrt{E_1(\phi^n )+\delta} 
 \sqrt{E_1(\phi(t^{n+1}) )+\delta} (\sqrt{E_1(\phi^n )+\delta}+ \sqrt{E_1(\phi(t^{n+1}) )+\delta}) }  \\
 &+ \frac{ r(t^{n+1}) }{ \sqrt{E_1(\phi^n )+\delta} } \left( F^{\prime}( \phi^{n})-F^{\prime}(\phi(t^{n+1}) ) \right),
 \endaligned
\end{equation*}
we obtain
\begin{equation}\label{e_phi10}
\|E_{F}^{n+1}\|\leq  C \| e_{\phi}^{n} \|
+ C \| \phi \|_{W^{1,\infty}(0,T; L^2(\Omega))} \Delta t. 
\end{equation}
Similarly we have
\begin{equation}\label{e_phi12}
\aligned
& \|\nabla E_{F}^{n+1}\| \leq   C \| e_{\phi}^{n} \| +C \| \nabla e_{\phi}^{n} \|
+ C \| \phi \|_{W^{1,\infty}(0,T; H^1(\Omega))} \Delta t . 
\endaligned
\end{equation}
On the other hand,
\begin{equation}\label{e_phi13}
\aligned
 E_{N}^{n+1}= &\frac{r(t^{n+1}) }{\sqrt{E_1(\phi(t^{n+1}))+\delta}} (\textbf{u}(t^{n+1}) \cdot \nabla ) \phi(t^{n+1})- \frac{r^{n+1}}{\sqrt{E_1( \phi^{n} )+\delta}} ( \textbf{u} ^{n} \cdot \nabla ) \phi^{n} \\
=& \frac{r(t^{n+1}) }{ \sqrt{E_1(\phi(t^{n+1}))+\delta} } \nabla \cdot ( \textbf{u}(t^{n+1})\phi(t^{n+1}) )- \frac{r^{n+1}}{\sqrt{E_1( \phi^{n} )+\delta}} \nabla \cdot  ( \textbf{u} ^{n}  \phi^{n} ) \\
=&  \frac{r(t^{n+1}) }{ \sqrt{E_1(\phi(t^{n+1}))+\delta} } \nabla \cdot ( \textbf{u}(t^{n+1})\phi(t^{n+1})- \textbf{u}(t^{n})\phi(t^{n}) )  \\
& + ( \frac{ r(t^{n+1}) }{ \sqrt{E_1(\phi(t^{n+1}))+\delta} }-\frac{ r^{n+1} }{ \sqrt{E_1(\phi^n)+\delta} } )  \nabla \cdot ( \textbf{u}(t^{n})\phi(t^{n}) ) \\ 
 &- \frac{r^{n+1}}{\sqrt{E_1( \phi^{n} )+\delta}} \nabla  \cdot  ( \textbf{u}(t^{n}) e_{\phi}^n )-
\frac{r^{n+1}}{\sqrt{E_1( \phi^{n} )+\delta}} \nabla  \cdot  ( e_{ \textbf{u} }^n \phi^n ).
\endaligned
\end{equation}
Recalling Lemma \ref{lem_phi_H2_boundness} and \eqref{e_error1}, \eqref{e_Soblev}, and 
\begin{equation}\label{e_Soblev_infty}
\aligned
\| f\|_{L^{\infty}} \leq C \|f \|^{1/2}_{H^1} \|f\|_{H^2}^{1/2},
\endaligned
\end{equation}
the first term on the right hand side of \eqref{e_phi9} can be bounded by
\begin{equation}\label{e_phi11}
\aligned 
 ( E_{F}^{n+1} ,&\, M e_{\mu}^{n+1}-R_{\phi}^{n+1} - E_{N}^{n+1} ) \\
\leq & C \|  E_{F}^{n+1} \| ( \| e_{\mu}^{n+1} \|+ \| R_{\phi}^{n+1} \|) + 
C \|  \nabla E_{F}^{n+1}\|   |e_r^{n+1}|  \\
&+C \|  \nabla E_{F}^{n+1}\| ( \|  \textbf{u} (t^n)\|_{L^4} \|  e_{\phi}^{n} \|_{L^4} + \|e_{ \textbf{u} }^n \| \| \phi^n \|_{L^{\infty} } ) \\
&+ C \|  \nabla E_{F}^{n+1} \|  (\| \phi \|_{W^{1,\infty}(0,T; L^2(\Omega))} +  \|  \textbf{u} \|_{W^{1,\infty}(0,T; L^2(\Omega))} ) \Delta t \\
&+ C \|  \nabla E_{F}^{n+1} \|  (\| \phi \|_{L^{\infty}(0,T; H^1(\Omega))} +  \|  \textbf{u} \|_{L^{\infty}(0,T; H^1(\Omega))} ) \Delta t 
 \\
\leq & C |e_r^{n+1}|^2+C \| e_{\phi}^{n} \|^2 +C \| \nabla e_{\phi}^{n} \|^2 
+\frac{M}{4} \| e_{\mu}^{n+1} \|^2 + C \| e_{ \textbf{u} }^n \|^2 \\
&+  C ( \| \textbf{u} \|_{W^{1,\infty}(0,T; L^2(\Omega))} ^2 + \| \textbf{u} \|_{L^{\infty}(0,T; H^1(\Omega))} ^2 ) (\Delta t)^2 \\
&+  C ( \| \phi \|_{W^{2,\infty}(0,T; L^2(\Omega))}^2 + \| \phi \|_{W^{1,\infty}(0,T; H^1(\Omega))}^2  ) (\Delta t)^2. 
\endaligned
\end{equation}
The second term on the right hand side of \eqref{e_phi9} can be estimated by
\begin{equation}\label{e_phi14}
\aligned
  M ( \nabla E_{F}^{n+1} ,  \nabla e_{\mu}^{n+1} ) \leq & C \| e_{\phi}^{n} \|^2 +C \| \nabla e_{\phi}^{n} \|^2 + \frac{M}{4} \| \nabla e_{\mu}^{n+1} \|^2 \\
&  + C \| \phi \|_{W^{1,\infty}(0,T; H^1(\Omega))} ^2 (\Delta t )^2.
\endaligned
\end{equation}
Using \eqref{e_phi13}, the fourth term on the right hand side of \eqref{e_phi8} can be bounded by
\begin{equation}\label{e_phi15}
\aligned
 & ( E_{N}^{n+1},  e_{\mu}^{n+1}+ e_{\phi}^{n+1} )  \\
& \ \ \ 
 \leq C \|  \nabla e_{\mu}^{n+1}+ \nabla e_{\phi}^{n+1} \| (  |e_r^{n+1}|  +
 \|  \textbf{u} (t^n)\|_{L^4} \|  e_{\phi}^{n} \|_{L^4} + \|e_{ \textbf{u} }^n \| \| \phi^n \|_{L^{\infty} } ) \\
&\ \ \ \ \ \ 
+ C \|  \nabla e_{\mu}^{n+1}+ \nabla e_{\phi}^{n+1} \| (\| \phi \|_{W^{1,\infty}(0,T; L^2(\Omega))} +  \|  \textbf{u} \|_{W^{1,\infty}(0,T; L^2(\Omega))} ) \Delta t \\
&\ \ \ \ \ \ 
+ C \|  \nabla e_{\mu}^{n+1}+ \nabla e_{\phi}^{n+1} \| (\| \phi \|_{L^{\infty}(0,T; H^1(\Omega))} +  \|  \textbf{u} \|_{L^{\infty}(0,T; H^1(\Omega))} ) \Delta t \\
& \ \ \ 
\leq  C |e_r^{n+1}|^2+C \| e_{\phi}^{n} \|^2+C \| \nabla e_{\phi}^{n} \|^2 +C \| \nabla e_{\phi}^{n+1} \|^2 +
C \| e_{ \textbf{u} }^n \|^2 +\frac{M}{4} \| \nabla e_{\mu}^{n+1} \|^2 \\
& \ \ \ \ \ \ 
 +  C ( \| \textbf{u} \|_{W^{1,\infty}(0,T; L^2(\Omega))} ^2 + \| \textbf{u} \|_{L^{\infty}(0,T; H^1(\Omega))} ^2 ) (\Delta t)^2 \\
& \ \ \ \ \ \  
 + C ( \| \phi \|_{W^{1,\infty}(0,T; L^2(\Omega))} ^2 + \| \phi \|_{L^{\infty}(0,T; H^1(\Omega))} ^2 ) (\Delta t)^2.
\endaligned
\end{equation}
Combining \eqref{e_phi8} with \eqref{e_phi8_second}-\eqref{e_phi15}, we obtain
\begin{equation}\label{e_phi16}
\aligned
\frac{1}{2\Delta t}& ( \| \nabla e_{\phi}^{n+1} \|^2-  \| \nabla e_{\phi}^{n} \|^2 +  \| \nabla e_{\phi}^{n+1}- \nabla e_{\phi}^{n} \|^2 )  + \frac{M}{2} \| \nabla e_{\mu}^{n+1} \|^2 \\
& \ \ \ + \frac{ \gamma+1 }{2 \Delta t} ( \| e_{\phi}^{n+1}\|^2- \| e_{\phi}^{n}\|^2 +\| e_{\phi}^{n+1}-e_{\phi}^{n} \|^2 ) +\frac{M}{2} \| e_{\mu}^{n+1} \|^2 \\
\leq & - \frac{e_r^{n+1}}{ \sqrt{E_1( \phi^{n})+\delta} } ( F^{\prime}( \phi^{n}), \frac{ e_{\phi}^{n+1}- e_{\phi}^{n} } {\Delta t} )+ C |e_r^{n+1}|^2 \\
& +C \| e_{\phi}^{n} \|^2 +C \| \nabla e_{\phi}^{n} \|^2+C \| \nabla e_{\phi}^{n+1} \|^2 +C \| e_{ \textbf{u} }^n \|^2  \\
&+  C ( \| \textbf{u} \|_{W^{1,\infty}(0,T; L^2(\Omega))} ^2 + \| \textbf{u} \|_{L^{\infty}(0,T; H^1(\Omega))} ^2 ) (\Delta t)^2 \\
&+  C ( \| \phi \|_{W^{2,\infty}(0,T; L^2(\Omega))}^2 + \| \phi \|_{W^{1,\infty}(0,T; H^1(\Omega))}^2  ) (\Delta t)^2. 
\endaligned
\end{equation}
Next we continue the estimate by establishing an error equation corresponding to the auxiliary variable $r$. Let $R_{r}^{n+1}$ be the truncation error defined by 
\begin{equation}\label{e_phi17}
\aligned
R_{r}^{n+1}=\frac{\partial r(t^{n+1})}{\partial t}- \frac{r(t^{n+1})-r(t^{n})}{\Delta t} =\frac{1}{\Delta t}\int_{t^n}^{t^{n+1}}(t^n-t)\frac{\partial^2 r }{\partial t^2}dt.
\endaligned
\end{equation}
Subtracting \eqref{e_model_rC} at $t^{n+1}$ from \eqref{e_model_semi_first_discrete3} and multiplying the equation by $2 e_r^{n+1}$ lead to
\begin{equation}\label{e_phi18}
\aligned
\frac{1}{\Delta t}& ( |e_r^{n+1}|^2- |e_r^n|^2 +|e_r^{n+1}-e_r^{n}|^2 ) \\
=&  \frac{ e_r^{n+1} }{ \sqrt{E_1( \phi^{n})+\delta}} ( F^{\prime}( \phi^{n}), \frac{e_{\phi}^{n+1}-e_{\phi}^n}{\Delta t} ) - \frac{ e_r^{n+1} }{ \sqrt{E_1( \phi^{n})+\delta}} ( F^{\prime}( \phi^{n}), R_{\phi}^{n+1} ) \\
&+\frac{ e_r^{n+1} }{ \sqrt{E_1( \phi^{n})+\delta}} ( F^{\prime}( \phi^{n})-F^{\prime}(\phi(t^{n+1})) , 
\frac{\partial \phi(t^{n+1}) }{\partial t} ) \\
&+ (\frac{ e_r^{n+1} }{ \sqrt{E_1( \phi^{n})+\delta}} - \frac{ e_r^{n+1}  }{ \sqrt{E_1(\phi(t^{n+1}))+\delta} } ) 
( F^{\prime}(\phi(t^{n+1})) , \frac{\partial \phi(t^{n+1}) }{\partial t} ) \\
 &  +\frac{ e_r^{n+1} }{ \sqrt{E_1( \phi^{n})+\delta}} \left(  (\mu^{n+1}, \textbf{u}^n \cdot \nabla \phi^n) - ( \tilde{\textbf{u}}^{n+1}, \mu^n \nabla \phi^n ) \right) + R_{r}^{n+1}  e_r^{n+1} .
\endaligned
\end{equation}
The second term on the right hand side of \eqref{e_phi18} can be estimated by
\begin{equation}\label{e_phi19}
\aligned
&- \frac{ e_r^{n+1} }{ \sqrt{E_1( \phi^{n})+\delta}} ( F^{\prime}( \phi^{n}), R_{\phi}^{n+1} ) \leq  C |e_r^{n+1}|^2 +  C  \| \phi \|_{W^{2,\infty}(0,T; L^2(\Omega))}  (\Delta t)^2.
\endaligned
\end{equation}
The third and fourth terms on the right hand side of \eqref{e_phi18} can be bounded by
\begin{equation}\label{e_phi20}
\aligned
\frac{ e_r^{n+1} }{ \sqrt{E_1( \phi^{n})+\delta}} &( F^{\prime}( \phi^{n})-F^{\prime}(\phi(t^{n+1})) , \frac{\partial \phi(t^{n+1}) }{\partial t} ) \\
&+(\frac{ e_r^{n+1} }{ \sqrt{E_1( \phi^{n})+\delta}} - \frac{ e_r^{n+1}  }{ \sqrt{E_1(\phi(t^{n+1}))+\delta} } ) 
( F^{\prime}(\phi(t^{n+1})) , \frac{\partial \phi(t^{n+1}) }{\partial t} ) \\
\leq &  C |e_r^{n+1}|^2 + C |e_{\phi}^{n}|^2 + C  \| \phi \|_{W^{1,\infty}(0,T; L^2(\Omega))}  ^2 (\Delta t)^2.
\endaligned
\end{equation}
By using \eqref{e_model_semi_first_discrete6}, we have
\begin{equation}\label{e_phi21_add}
\aligned
 \tilde{e}_{\textbf{u}}^{n+1} =& e_{\textbf{u}}^{n+1} + \Delta t \nabla(p^{n+1}-p^n) \\
=&  e_{\textbf{u}}^{n+1} + \Delta t \nabla(e_p^{n+1}-e_p^n) +  \Delta t \nabla( p( t^{n+1} )-p(t^n) ).
\endaligned
\end{equation}
Using the above,  \eqref{e_error1} and Lemma \ref{lem_phi_H2_boundness}, 
the fifth term on the right hand side of \eqref{e_phi18} can be estimated by
\begin{equation}\label{e_phi21}
\aligned
\frac{ e_r^{n+1} }{ \sqrt{E_1( \phi^{n})+\delta}}& \left(  (\mu^{n+1}, \textbf{u}^n \cdot \nabla \phi^n) - ( \tilde{\textbf{u}}^{n+1}, \mu^n \nabla \phi^n ) \right) \\
\leq & \frac{ e_r^{n+1} }{ \sqrt{E_1( \phi^{n})+\delta}} \left(  ( \mu^{n+1}, \textbf{u}^n \cdot \nabla \phi^n) - (\mu^{n}, \textbf{u}^n \cdot \nabla \phi^n)  \right) \\
&+ \frac{ e_r^{n+1} }{ \sqrt{E_1( \phi^{n})+\delta}} \left( (\textbf{u}^n, \mu^n \nabla \phi^n ) - ( \tilde{\textbf{u}}^{n+1}, \mu^n \nabla \phi^n ) \right) \\
\leq & C | e_r^{n+1} | \| \mu^{n+1}-\mu^n \|_{L^4} \| \textbf{u}^n \| \| \nabla \phi^n \|_{L^4}
+ C | e_r^{n+1} | \| \textbf{u}^n - \tilde{\textbf{u}}^{n+1} \| \| \mu^n \|_{L^4} \| \nabla \phi^n \|_{L^4} \\
\leq & \frac{M}{8} ( \| e_{\mu}^n \|^2+  \| e_{\mu}^{n+1} \|^2 +\| \nabla e_{\mu}^n \|^2+  \| \nabla e_{\mu}^{n+1} \|^2 ) +C\| e_{ \textbf{u} }^n \|^2 +C \| e_{\textbf{u}}^{n+1} \|^2 \\
&+ C (\Delta t)^2 \| \nabla(e_p^{n+1}-e_p^n) \|^2 + C |e_r^{n+1} |^2 + \frac{1}{4K_1}\| \nabla \mu^{n} \|^2 | e_r^{n+1} |^2 \\
& + C \| p \|_{W^{1,\infty}(0,T; H^1(\Omega))}^2 (\Delta t )^4 + C \| \mu \|_{W^{1,\infty}(0,T; H^1(\Omega))} ^2 (\Delta t)^2 \\
&+ C \| \textbf{u} \|_{W^{1,\infty}(0,T; L^2(\Omega))} ^2 (\Delta t)^2.
\endaligned
\end{equation}
Combining \eqref{e_phi18} with \eqref{e_phi19}-\eqref{e_phi21} results in
\begin{equation}\label{e_phi22}
\aligned
&\frac{1}{\Delta t} ( |e_r^{n+1}|^2- |e_r^n|^2 +|e_r^{n+1}-e_r^{n}|^2 ) \\
 \leq  & \frac{ e_r^{n+1} }{ \sqrt{E_1( \phi^{n})+\delta}} ( F^{\prime}( \phi^{n}), \frac{e_{\phi}^{n+1}-e_{\phi}^n}{\Delta t} )   + \frac{M}{8} ( \| e_{\mu}^n \|^2+  \| e_{\mu}^{n+1} \|^2 +\| \nabla e_{\mu}^n \|^2+  \| \nabla e_{\mu}^{n+1} \|^2 ) \\
& +C\| e_{ \textbf{u} }^n \|^2 +C \| e_{\textbf{u}}^{n+1} \|^2 + C (\Delta t)^2 \| \nabla(e_p^{n+1}-e_p^n) \|^2 + (C+
 \frac{1}{4K_1} \| \nabla \mu^{n} \|^2) | e_r^{n+1} |^2   \\
& + C |e_{\phi}^{n}|^2 + C \| p \|_{W^{1,\infty}(0,T; H^1(\Omega))}^2 (\Delta t )^4 + C \| \mu \|_{W^{1,\infty}(0,T; H^1(\Omega))} ^2
 (\Delta t)^2  \\
 &+ C \| \textbf{u}  \|_{W^{1,\infty}(0,T; L^2(\Omega))} ^2 (\Delta t)^2 +  C \| r \|_{W^{2,\infty}(0,T)} ^2(\Delta t)^2 \\
 & +C   \| \phi \|_{W^{2,\infty}(0,T; L^2(\Omega))} ^2  (\Delta t)^2 .
\endaligned
\end{equation}
Combining the above equation with \eqref{e_phi16} gives
\begin{equation}\label{e_phi23}
\aligned
&\frac{1}{2\Delta t} ( \| \nabla e_{\phi}^{n+1} \|^2-  \| \nabla e_{\phi}^{n} \|^2 +  \| \nabla e_{\phi}^{n+1}- \nabla e_{\phi}^{n} \|^2 )  + \frac{M}{4} \| \nabla e_{\mu}^{n+1} \|^2 \\
& + \frac{ \gamma+1 }{2 \Delta t} ( \| e_{\phi}^{n+1}\|^2- \| e_{\phi}^{n}\|^2 +\| e_{\phi}^{n+1}-e_{\phi}^{n} \|^2 ) +\frac{M}{4} \| e_{\mu}^{n+1} \|^2 \\
&+ \frac{1}{\Delta t} ( |e_r^{n+1}|^2- |e_r^n|^2 +|e_r^{n+1}-e_r^{n}|^2 )  \\
\leq & C \| e_{\phi}^{n} \|^2 +C \| \nabla e_{\phi}^{n} \|^2 +C \| \nabla e_{\phi}^{n+1} \|^2 +C\| e_{ \textbf{u} }^n \|^2+C \| e_{\textbf{u}}^{n+1} \|^2 \\
&+ C (\Delta t)^2 \| \nabla(e_p^{n+1}-e_p^n) \|^2 + (C+ \frac{1}{4K_1}\| \nabla \mu^{n} \|^2) | e_r^{n+1} |^2  \\
&+ \frac{M}{8}  \| e_{\mu}^n \|^2 +\frac{M}{8} \| \nabla e_{\mu}^n \|^2 +  C \| r \|_{W^{2,\infty}(0,T)} ^2(\Delta t)^2 \\
&+  C ( \| \textbf{u} \|_{W^{1,\infty}(0,T; L^2(\Omega))} ^2 + \| \textbf{u} \|_{L^{\infty}(0,T; H^1(\Omega))} ^2 ) (\Delta t)^2 \\
&+  C ( \| \phi \|_{W^{2,\infty}(0,T; L^2(\Omega))}^2 + \| \phi \|_{W^{1,\infty}(0,T; H^1(\Omega))}^2  ) (\Delta t)^2 \\
&+ C \| p \|_{W^{1,\infty}(0,T; H^1(\Omega))}^2 (\Delta t )^4 + C \| \mu \|_{W^{1,\infty}(0,T; H^1(\Omega))} ^2
 (\Delta t)^2,
\endaligned
\end{equation}
which implies the desired result.
\end{proof}

\medskip
\begin{lemma}\label{lem: error_estimate_u}
Assuming $ \phi \in W^{3,\infty}(0,T; L^2(\Omega)) \bigcap W^{2,\infty}(0,T; H^2(\Omega))  $, $ \mu \in W^{1,\infty}(0,T; H^1(\Omega)) $ $ \bigcap L^{\infty}(0,T; H^2(\Omega)) $, $\textbf{u}\in W^{3,\infty}(0,T;\textbf{L}^2(\Omega))\bigcap W^{2,\infty}(0,T;\textbf{H}^1(\Omega))\bigcap L^{\infty}(0,T;\textbf{H}^2(\Omega))$, and $p \in $ \\
$ W^{2,\infty}(0,T;H^1(\Omega))$ , then for the first-order  scheme \eqref{e_model_semi_first_discrete1}-\eqref{e_model_semi_first_discrete5}, we have
\begin{equation*}
\aligned
& \frac{\|e_{\textbf{u}}^{n+1}\|^2-\|e_{\textbf{u}}^n\|^2}{2\Delta t} +\frac{\|\tilde{e}_{\textbf{u}}^{n+1}-e_{\textbf{u}}^n\|^2}{2\Delta t}+\frac{\nu} {2} \|\nabla\tilde{e}_{\textbf{u}}^{n+1}\|^2+\frac{\Delta t}{2}(\|\nabla e_p^{n+1}\|^2-\|\nabla e_p^n\|^2)\\
\leq & -\exp( \frac{t^{n+1}}{T} )e_q^{n+1} \left((\textbf{u}^n \cdot \nabla)\textbf{u}^n,\tilde{e}_{\textbf{u}}^{n+1}\right) +C (  \| \textbf{u}\|_{L^{\infty}(0,T; H^2(\Omega))}^2 +\|e_{\textbf{u}}^n\|^2_1)  \|e_{\textbf{u}}^n\|^2 \\
& +\frac{M}{16} \|e_{\mu}^{n} \|^2+\frac{M}{16} \| \nabla e_{\mu}^{n} \|^2 + C \| \nabla e_{\phi}^n\|^2  +C \| e_{\phi}^{n} \|^2   + C |e_r^{n+1}|^2 +C \| e_{\textbf{u}}^{n+1} \|^2 \\
&+ C (\Delta t)^2 ( \|\nabla e_p^n\|^2+ \|\nabla e_p^{n+1}\|^2 ) +C (\Delta t )^2,  \ \ \ \forall \ 0\leq n\leq N-1,
\endaligned
\end{equation*}
where the positive constant $C$ is independent of $\Delta t$. 
\end{lemma}

\begin{proof}
Let $\textbf{R}_{\textbf{u}}^{n+1}$ be the truncation error defined by  
\begin{equation}\label{e_u1}
\aligned
\textbf{R}_{\textbf{u}}^{n+1}=\frac{\partial \textbf{u}(t^{n+1})}{\partial t}- \frac{\textbf{u}(t^{n+1})-\textbf{u}(t^{n})}{\Delta t}=\frac{1}{\Delta t}\int_{t^n}^{t^{n+1}}(t^n-t)\frac{\partial^2 \textbf{u}}{\partial t^2}dt.
\endaligned
\end{equation}
Subtracting \eqref{e_model_rD} at $t^{n+1}$ from \eqref{e_model_semi_first_discrete4}, we obtain
\begin{equation}\label{e_u2}
\aligned
&\frac{\tilde{e}_{\textbf{u}}^{n+1}-e_{\textbf{u}}^n}{\Delta t}-\nu\Delta\tilde{e}_{\textbf{u}}^{n+1}
=\exp( \frac{t^{n+1}}{T} )q(t^{n+1}) (\textbf{u}(t^{n+1})\cdot \nabla)\textbf{u}(t^{n+1})\\
&-\exp( \frac{t^{n+1}}{T} )q^{n+1} \textbf{u}^{n}\cdot \nabla\textbf{u}^{n}
-\nabla (p^n-p(t^{n+1})) \\
& +\frac{r^{n+1}}{\sqrt{E_1( \phi^{n} )+\delta}}  \mu^n \nabla \phi^n-
 \frac{r(t^{n+1}) }{\sqrt{E_1(\phi(t^{n+1}) )+\delta}} \mu(t^{n+1}) \nabla \phi(t^{n+1} ) 
+\textbf{R}_{\textbf{u}}^{n+1}.
\endaligned
\end{equation}
Taking the inner product of \eqref{e_u2} with $\tilde{e}_{\textbf{u}}^{n+1}$, we obtain
\begin{equation}\label{e_u3}
\aligned
&\frac{\|\tilde{e}_{\textbf{u}}^{n+1}\|^2-\|e_{\textbf{u}}^n\|^2}{2\Delta t}+\frac{\|\tilde{e}_{\textbf{u}}^{n+1}-e_{\textbf{u}}^n\|^2}{2\Delta t}+\nu\|\nabla\tilde{e}_{\textbf{u}}^{n+1}\|^2\\
=& \left( \exp( \frac{t^{n+1}}{T} )q(t^{n+1}) (\textbf{u}(t^{n+1})\cdot \nabla)\textbf{u}(t^{n+1})- \exp( \frac{t^{n+1}}{T} )q^{n+1} \textbf{u}^{n}\cdot \nabla\textbf{u}^{n}, \tilde{e}_{\textbf{u}}^{n+1}\right) \\
&+ \left( \frac{r^{n+1}}{\sqrt{E_1( \phi^{n} )+\delta}}  \mu^n \nabla \phi^n-
 \frac{r(t^{n+1}) }{\sqrt{E_1(\phi(t^{n+1}) )+\delta}} \mu(t^{n+1}) \nabla \phi(t^{n+1} ) , \tilde{e}_{\textbf{u}}^{n+1} \right)  \\
&-(\nabla(p^n-p(t^{n+1})),\tilde{e}_{\textbf{u}}^{n+1})+(\textbf{R}_{\textbf{u}}^{n+1},\tilde{e}_{\textbf{u}}^{n+1}).
\endaligned
\end{equation}
By \eqref{e_model_semi_first_discrete6}, we can obtain that 
\begin{equation}\label{e_u4}
\aligned
&\frac{e_{\textbf{u}}^{n+1}-\tilde{e}_{\textbf{u}}^{n+1}}{\Delta t}+\nabla(p^{n+1}-p^n)=0.
\endaligned
\end{equation}
Taking the inner product of \eqref{e_u4} with $\frac{e_{\textbf{u}}^{n+1}+\tilde{e}_{\textbf{u}}^{n+1}}{2}$, we derive
\begin{equation}\label{e_u5}
\aligned
&\frac{\|e_{\textbf{u}}^{n+1}\|^2-\|\tilde{e}_{\textbf{u}}^{n+1}\|^2}{2\Delta t}+\frac{1}{2}\left(\nabla(p^{n+1}-p^n),\tilde{e}_{\textbf{u}}^{n+1}\right)=0.
\endaligned
\end{equation}
Adding \eqref{e_u3} and \eqref{e_u5}, we have
\begin{equation}\label{e_u6}
\aligned
&\frac{\|e_{\textbf{u}}^{n+1}\|^2-\|e_{\textbf{u}}^n\|^2}{2\Delta t}+\frac{\|\tilde{e}_{\textbf{u}}^{n+1}-e_{\textbf{u}}^n\|^2}{2\Delta t}+\nu\|\nabla\tilde{e}_{\textbf{u}}^{n+1}\|^2\\
=& \left( \exp( \frac{t^{n+1}}{T} )q(t^{n+1}) (\textbf{u}(t^{n+1})\cdot \nabla)\textbf{u}(t^{n+1})- \exp( \frac{t^{n+1}}{T} )q^{n+1} \textbf{u}^{n}\cdot \nabla\textbf{u}^{n}, \tilde{e}_{\textbf{u}}^{n+1}\right) \\
&+ \left( \frac{r^{n+1}}{\sqrt{E_1( \phi^{n} )+\delta}}  \mu^n \nabla \phi^n-
 \frac{r(t^{n+1}) }{\sqrt{E_1(\phi(t^{n+1}) )+\delta}} \mu(t^{n+1}) \nabla \phi(t^{n+1} ) , \tilde{e}_{\textbf{u}}^{n+1} \right)  \\
&-\frac{1}{2}\left(\nabla(p^{n+1}+p^n-2p(t^{n+1})),\tilde{e}_{\textbf{u}}^{n+1}\right)+(\textbf{R}_{\textbf{u}}^{n+1},\tilde{e}_{\textbf{u}}^{n+1}).
\endaligned
\end{equation}
For the first term on the right hand side of \eqref{e_u6}, we have
\begin{equation}\label{e_u7}
\aligned
&\left( \exp( \frac{t^{n+1}}{T} )q(t^{n+1}) (\textbf{u}(t^{n+1})\cdot \nabla)\textbf{u}(t^{n+1})- \exp( \frac{t^{n+1}}{T} )q^{n+1} \textbf{u}^{n}\cdot \nabla\textbf{u}^{n}, \tilde{e}_{\textbf{u}}^{n+1}\right) \\
=& \exp( \frac{t^{n+1}}{T} )q(t^{n+1}) \left( ( \textbf{u}(t^{n+1})-\textbf{u}^{n} )\cdot \nabla \textbf{u}(t^{n+1}),\tilde{e}_{\textbf{u}}^{n+1}\right)\\
&+  \exp( \frac{t^{n+1}}{T} ) q(t^{n+1}) \left(  \textbf{u}^{n}\cdot \nabla (\textbf{u}(t^{n+1})-\textbf{u}^{n} ),\tilde{e}_{\textbf{u}}^{n+1}\right)\\
&-\exp( \frac{t^{n+1}}{T} )e_q^{n+1} \left((\textbf{u}^n \cdot \nabla)\textbf{u}^n,\tilde{e}_{\textbf{u}}^{n+1}\right).
\endaligned
\end{equation}
Recalling \eqref{e_norm H1} and  \eqref{e_estimate for trilinear form}, the first term on the right hand side of \eqref{e_u7} can be estimated by
\begin{equation}\label{e_u8}
\aligned
\exp( \frac{t^{n+1}}{T} ) & q(t^{n+1}) \left( ( \textbf{u}(t^{n+1})-\textbf{u}^{n} )\cdot \nabla \textbf{u}(t^{n+1}),\tilde{e}_{\textbf{u}}^{n+1}\right)\\
\leq &c_2(1+c_1) \| \textbf{u}(t^{n+1})-\textbf{u}^{n} \| \| \textbf{u}(t^{n+1})\|_2
\|\nabla \tilde{e}_{\textbf{u}}^{n+1} \| \\
\leq & \frac{\nu}{8} \|\nabla \tilde{e}_{\textbf{u}}^{n+1} \|^2+C \| \textbf{u}\|_{L^{\infty}(0,T; H^2(\Omega))}^2 \|e_{\textbf{u}}^n\|^2 \\
&+C \| \textbf{u}\|_{L^{\infty}(0,T; H^2(\Omega))}^2
 \| \textbf{u}\|_{W^{1,\infty}(0,T; L^2(\Omega))}^2 
 (\Delta t)^2.
\endaligned
\end{equation}
Noticing \eqref{e_error1} and using Cauchy-Schwarz inequality, the second term on the right hand side of \eqref{e_u7} can be estimated by
\begin{equation}\label{e_u9}
\aligned
 \exp( \frac{t^{n+1}}{T} )&  q(t^{n+1}) \left(  \textbf{u}^{n}\cdot \nabla (\textbf{u}(t^{n+1})-\textbf{u}^{n} ),\tilde{e}_{\textbf{u}}^{n+1}\right)  \\
= & \exp( \frac{t^{n+1}}{T} )  q(t^{n+1}) \left(  \textbf{u}^{n}\cdot \nabla (\textbf{u}(t^{n+1})-\textbf{u}(t^{n}) ),\tilde{e}_{\textbf{u}}^{n+1}\right) \\
& -  \exp( \frac{t^{n+1}}{T} )  q(t^{n+1}) \left( e_{\textbf{u}}^{n}\cdot \nabla e_{\textbf{u}}^n,\tilde{e}_{\textbf{u}}^{n+1}\right)
-\exp( \frac{t^{n+1}}{T} )  q(t^{n+1})  \left(  \textbf{u}(t^n)\cdot \nabla e_{\textbf{u}}^n,\tilde{e}_{\textbf{u}}^{n+1}\right) \\
\leq & \frac{\nu}{8} \|\nabla \tilde{e}_{\textbf{u}}^{n+1} \|^2+C (  \| \textbf{u}\|_{L^{\infty}(0,T; H^2(\Omega))}^2 +\|e_{\textbf{u}}^n\|^2_1)  \|e_{\textbf{u}}^n\|^2\\
&+C \| \textbf{u} \|_{W^{1,\infty}(0,T; H^2(\Omega) )}^2 (\Delta t )^2.
\endaligned
\end{equation}
Then using \eqref{e_phi21_add} and \eqref{e_error1},  \eqref{e_Soblev} and Lemma \ref{lem_phi_H2_boundness}, the second term on the right hand side of \eqref{e_u6} can be bounded by
\begin{equation}\label{e_u11}
\aligned
& \left( \frac{r^{n+1}}{\sqrt{E_1( \phi^{n} )+\delta}}  \mu^n \nabla \phi^n-
 \frac{r(t^{n+1}) }{\sqrt{E_1(\phi(t^{n+1}) )+\delta}} \mu(t^{n+1}) \nabla \phi(t^{n+1} ) , \tilde{e}_{\textbf{u}}^{n+1} \right)  \\
= & \frac{r^{n+1}}{\sqrt{E_1( \phi^{n} )+\delta}}  \left(  \mu^n \nabla \phi^n- \mu(t^{n+1}) \nabla \phi(t^{n+1} ) , \tilde{e}_{\textbf{u}}^{n+1} \right) \\
&+ (  \frac{r^{n+1}}{\sqrt{E_1( \phi^{n} )+\delta}}-  \frac{r(t^{n+1}) }{\sqrt{E_1(\phi(t^{n+1}) )+\delta}} ) ( \mu(t^{n+1}) \nabla \phi(t^{n+1} ), \tilde{e}_{\textbf{u}}^{n+1}  ) \\
\leq & C \| \mu^n- \mu(t^{n+1})\|_{L^4} \| \nabla \phi^n \|_{L^4} \| \tilde{e}_{\textbf{u}}^{n+1} \| +C \| \mu(t^{n+1}) \|_{L^{\infty}(\Omega)} \| \nabla \phi^n- \nabla \phi(t^{n+1}) \| \| \tilde{e}_{\textbf{u}}^{n+1} \| \\
&+ r(t^{n+1}) (  \frac{1}{\sqrt{E_1( \phi^{n} )+\delta}}-  \frac{1 }{\sqrt{E_1(\phi(t^{n+1}) )+\delta}} ) ( \mu(t^{n+1}) \nabla \phi(t^{n+1} ), \tilde{e}_{\textbf{u}}^{n+1}  ) \\
&+ \frac{e_r^{n+1}}{\sqrt{E_1( \phi^{n} )+\delta}}  ( \mu(t^{n+1}) \nabla \phi(t^{n+1} ), \tilde{e}_{\textbf{u}}^{n+1}  )   \\
\leq & \frac{\nu}{8} \|\nabla \tilde{e}_{\textbf{u}}^{n+1} \|^2
+\frac{M}{16} \|e_{\mu}^{n} \|^2+\frac{M}{16} \| \nabla e_{\mu}^{n} \|^2 + C \| \nabla e_{\phi}^n\|^2 
+C \| e_{\phi}^{n} \|^2 + C |e_r^{n+1}|^2 \\
&+C \| e_{\textbf{u}}^{n+1} \|^2 + C (\Delta t)^2 \| \nabla(e_p^{n+1}-e_p^n) \|^2 + C \| p \|_{W^{1,\infty}(0,T; H^1(\Omega))}^2 (\Delta t )^4 \\
&   +C(  \| \mu \|_{W^{1,\infty}(0,T; H^{1}(\Omega))}^2 +  \| \mu \|_{L^{\infty}(0,T; H^{2}(\Omega))}^2 )(\Delta t )^2 \\
& +C \| \phi \|_{W^{1,\infty}(0,T; H^1(\Omega))}^2 (\Delta t )^2 .
\endaligned
\end{equation}
Next we estimate the third term on the right hand side of \eqref{e_u6}. Using \eqref{e_phi21_add}, we have
\begin{equation}\label{e_u12}
\aligned
-\frac{1}{2}&\left(\nabla(p^{n+1}+p^n-2p(t^{n+1})),\tilde{e}_{\textbf{u}}^{n+1}\right)=
-\frac{1}{2}\left(\nabla(e_p^{n+1}+e_p^n-p(t^{n+1})+p(t^n)),
\tilde{e}_{\textbf{u}}^{n+1}\right)\\
=&-\frac{1}{2}\left(\nabla(e_p^{n+1}+e_p^n-p(t^{n+1})+p(t^n)),e_{\textbf{u}}^{n+1}+\Delta t(\nabla(e_p^{n+1}-e_p^n)+\nabla(p(t^{n+1})-p(t^n))\right)\\
=&-\frac{\Delta t}{2}(\|\nabla e_p^{n+1}\|^2-\|\nabla e_p^n\|^2)
-\Delta t (\nabla(p(t^{n+1})-p(t^n)),\nabla e_p^n)\\
&+\frac{\Delta t}{2}\|\nabla(p(t^{n+1})-p(t^n))\|^2\\
\leq & -\frac{\Delta t}{2}(\|\nabla e_p^{n+1}\|^2-\|\nabla e_p^n\|^2)+(\Delta t)^2\|\nabla e_p^n\|^2
+C\| p \|_{W^{1,\infty}(0,T; H^1(\Omega))}^2 \left( (\Delta t)^2+(\Delta t)^3\right).
\endaligned
\end{equation}
For the last term on the right hand side of \eqref{e_u6}, we have 
\begin{equation}\label{e_u13}
\aligned
&(\textbf{R}_{\textbf{u}}^{n+1},\tilde{e}_{\textbf{u}}^{n+1})\leq \frac{\nu}{8} \|\nabla \tilde{e}_{\textbf{u}}^{n+1} \|^2+C\| \textbf{u} \|_{W^{2,\infty}(0,T;H^{-1} (\Omega) )}^2 ( \Delta t )^2.
\endaligned
\end{equation}
Combining \eqref{e_u6} with \eqref{e_u7}-\eqref{e_u13}, we obtain
\begin{equation*}\label{e_u14}
\aligned
& \frac{\|e_{\textbf{u}}^{n+1}\|^2-\|e_{\textbf{u}}^n\|^2}{2\Delta t} +\frac{\|\tilde{e}_{\textbf{u}}^{n+1}-e_{\textbf{u}}^n\|^2}{2\Delta t}+\frac{\nu} {2} \|\nabla\tilde{e}_{\textbf{u}}^{n+1}\|^2+\frac{\Delta t}{2}(\|\nabla e_p^{n+1}\|^2-\|\nabla e_p^n\|^2)\\
\leq & -\exp( \frac{t^{n+1}}{T} )e_q^{n+1} \left((\textbf{u}^n \cdot \nabla)\textbf{u}^n,\tilde{e}_{\textbf{u}}^{n+1}\right) +C (  \| \textbf{u}\|_{L^{\infty}(0,T; H^2(\Omega))}^2 +\|e_{\textbf{u}}^n\|^2_1)  \|e_{\textbf{u}}^n\|^2 \\
& +\frac{M}{16} \|e_{\mu}^{n} \|^2+\frac{M}{16} \| \nabla e_{\mu}^{n} \|^2 + C \| \nabla e_{\phi}^n\|^2 
+C \| e_{\phi}^{n} \|^2  \\
& + C |e_r^{n+1}|^2 +C \| e_{\textbf{u}}^{n+1} \|^2+ C (\Delta t)^2 ( \|\nabla e_p^n\|^2+ \|\nabla e_p^{n+1}\|^2 )  \\
&+C\| p \|_{W^{1,\infty}(0,T; H^1(\Omega))}^2 \left( (\Delta t)^2+(\Delta t)^3 + (\Delta t )^4\right)   +C \| \phi \|_{W^{1,\infty}(0,T; H^1(\Omega))}^2 (\Delta t )^2  \\
 &+C(  \| \mu \|_{W^{1,\infty}(0,T; H^{1}(\Omega))}^2 +\| \mu \|_{L^{\infty}(0,T; H^{2}(\Omega))}^2 ) (\Delta t )^2  \\
& +C ( \| \textbf{u} \|_{W^{2,\infty}(0,T;H^{-1} (\Omega) )}^2 + \| \textbf{u} \|_{W^{1,\infty}(0,T; H^2(\Omega) )}^2 ) (\Delta t )^2,
\endaligned
\end{equation*}
which implies the desired result.
\end{proof}

\medskip
\begin{lemma}\label{lem: error_estimate_q}
Assuming $ \phi \in W^{3,\infty}(0,T; L^2(\Omega)) \bigcap W^{2,\infty}(0,T; H^2(\Omega))  $, $ \mu \in W^{1,\infty}(0,T; H^1(\Omega)) $ $ \bigcap L^{\infty}(0,T; H^2(\Omega)) $, $\textbf{u}\in W^{3,\infty}(0,T;\textbf{L}^2(\Omega))\bigcap W^{2,\infty}(0,T;\textbf{H}^1(\Omega))\bigcap L^{\infty}(0,T;\textbf{H}^2(\Omega))$, and $p \in $ \\
$ W^{2,\infty}(0,T;H^1(\Omega))$ , then for the first-order  scheme \eqref{e_model_semi_first_discrete1}-\eqref{e_model_semi_first_discrete5}, we have
\begin{equation*}
\aligned
& \frac{|e_q^{n+1}|^2-|e_q^n|^2}{2\Delta t} +\frac{|e_q^{n+1}-e_q^n|^2}{2\Delta t}+\frac{1}{2T}|e_q^{n+1}|^2\\
 \leq &   \exp( \frac{t^{n+1}}{T} ) e_q^{n+1}  (\textbf{u}^n\cdot\nabla \textbf{u}^n, \tilde{e}_{\textbf{u}}^{n+1} ) + C \|\textbf{u}^n \|_1^2 \| e_{\textbf{u}}^n \|^2 + C (\Delta t)^2 , \ \ \ \forall \ 0\leq n\leq N-1,
\endaligned
\end{equation*}
where the positive constant $C$ is independent of $\Delta t$. 
\end{lemma}

\begin{proof}
Subtracting \eqref{e_model_rE} from \eqref{e_model_semi_first_discrete5} leads to
\begin{equation}\label{e_q1}
\aligned
 &\frac{e_q^{n+1}-e_q^n}{\Delta t}+\frac{1}{T}e_q^{n+1}\\
 =& \exp( \frac{t^{n+1}}{T} ) \left( (\textbf{u}^n\cdot\nabla \textbf{u}^n,\tilde{\textbf{u}}^{n+1})-  (\textbf{u}(t^{n+1})\cdot \nabla \textbf{u}(t^{n+1}),\textbf{u}(t^{n+1}) ) \right)+\textbf{R}_{q}^{n+1},
\endaligned
\end{equation}
where 
\begin{equation}\label{e_q2}
\aligned
\textbf{R}_{q}^{n+1}=\frac{ \rm{d} q(t^{n+1})}{ \rm{d}  t}- \frac{q(t^{n+1})-q(t^{n})}{\Delta t}=\frac{1}{\Delta t}\int_{t^n}^{t^{n+1}}(t^n-t)\frac{\partial^2 q}{\partial t^2}dt.
\endaligned
\end{equation}
Multiplying both sides of \eqref{e_q1} by $e_q^{n+1}$ yields
\begin{equation}\label{e_q3}
\aligned
 &\frac{|e_q^{n+1}|^2-|e_q^n|^2}{2\Delta t}+\frac{|e_q^{n+1}-e_q^n|^2}{2\Delta t}+\frac{1}{T}|e_q^{n+1}|^2\\
 =&  \exp( \frac{t^{n+1}}{T} ) e_q^{n+1} (\textbf{u}^n\cdot\nabla \textbf{u}^n, \tilde{e}_{\textbf{u}}^{n+1} )-   \exp( \frac{t^{n+1}}{T} ) e_q^{n+1}   \left( \textbf{u}^n\cdot\nabla ( 
 \textbf{u}(t^{n+1})-\textbf{u}^n), \textbf{u}(t^{n+1}) \right)\\
 &-  \exp( \frac{t^{n+1}}{T} ) e_q^{n+1}  \left( ( 
 \textbf{u}(t^{n+1})-\textbf{u}^n)\cdot \nabla \textbf{u}(t^{n+1}), \textbf{u}(t^{n+1}) \right)+\textbf{R}_{q}^{n+1}e_q^{n+1},
\endaligned
\end{equation}
Taking notice of \eqref{e_estimate for trilinear form} and \eqref{e_error1}, the second term on the right hand side of \eqref{e_q3} can be recast into
\begin{equation}\label{e_q4}
\aligned
  - \exp( \frac{t^{n+1}}{T} ) & e_q^{n+1}   \left( \textbf{u}^n\cdot\nabla ( 
 \textbf{u}(t^{n+1})-\textbf{u}^n), \textbf{u}(t^{n+1}) \right) \\
 \leq & c_2 \exp(1)\|\textbf{u}^n \|_1 \| \textbf{u}(t^{n+1})-\textbf{u}^n \|_0 \| \textbf{u}(t^{n+1}) \|_{2} |e_q^{n+1}| \\
 \leq & \frac{1}{6 T} |e_q^{n+1}|^2 + C \|\textbf{u}^n \|_1^2 \| e_{\textbf{u}}^n \|^2
+C \| \textbf{u} \|_{W^{1,\infty}(0,T; L^2(\Omega) )} ^2  \| \textbf{u} \|_{L^{\infty}(0,T; H^2(\Omega) )} ^2 (\Delta t)^2. 
 \endaligned
\end{equation}
The third term on the right hand side of \eqref{e_q3} can be estimated by
\begin{equation}\label{e_q5}
\aligned
 - \exp( \frac{t^{n+1}}{T} ) & e_q^{n+1}  \left( ( 
 \textbf{u}(t^{n+1})-\textbf{u}^n)\cdot \nabla \textbf{u}(t^{n+1}), \textbf{u}(t^{n+1}) \right)\\
 \leq &c_2\exp(1) \| \textbf{u}(t^{n+1})-\textbf{u}^n \| \|\textbf{u}(t^{n+1}) \|_1 \|\textbf{u}(t^{n+1}) \|_2 |e_q^{n+1}| \\
 \leq & C\|e_{\textbf{u}}^{n}\|^2+\frac{1}{6 T}  |e_q^{n+1}|^2 + 
 C \| \textbf{u} \|_{W^{1,\infty}(0,T; L^2(\Omega) )} ^2  \| \textbf{u} \|_{L^{\infty}(0,T; H^2(\Omega) )} ^2 (\Delta t)^2.
\endaligned
\end{equation}
For the last term on the right hand side of \eqref{e_q3}, we obtain
\begin{equation}\label{e_q6}
\aligned
\textbf{R}_{q}^{n+1}e_q^{n+1} \leq \frac{1}{6 T}  |e_q^{n+1}|^2+C |q|_{W^{2,\infty}(0,T)} ^2 (\Delta t)^2.
\endaligned
\end{equation}
Combining \eqref{e_q3} with \eqref{e_q4}-\eqref{e_q6} results in 
\begin{equation*}\label{e_q7}
\aligned
& \frac{|e_q^{n+1}|^2-|e_q^n|^2}{2\Delta t} +\frac{|e_q^{n+1}-e_q^n|^2}{2\Delta t}+\frac{1}{2T}|e_q^{n+1}|^2\\
 \leq &   \exp( \frac{t^{n+1}}{T} ) e_q^{n+1}  (\textbf{u}^n\cdot\nabla \textbf{u}^n, \tilde{e}_{\textbf{u}}^{n+1} ) + C \|\textbf{u}^n \|_1^2 \| e_{\textbf{u}}^n \|^2  \\
 & + C \| \textbf{u} \|_{W^{1,\infty}(0,T; L^2(\Omega) )} ^2  \| \textbf{u} \|_{L^{\infty}(0,T; H^2(\Omega) )} ^2 (\Delta t)^2 +
C |q|_{W^{2,\infty}(0,T)} ^2 (\Delta t)^2,
\endaligned
\end{equation*}
which leads to the desired result.
\end{proof}

We are now in position to derive our  main results.
\medskip
\begin{theorem} \label{thm: error_estimate_u_phi}
Assuming $ \phi \in W^{3,\infty}(0,T; L^2(\Omega)) \bigcap W^{2,\infty}(0,T; H^2(\Omega))  $, $ \mu \in W^{1,\infty}(0,T; H^1(\Omega)) $ $ \bigcap L^{\infty}(0,T; H^2(\Omega)) $, $\textbf{u}\in W^{3,\infty}(0,T;\textbf{L}^2(\Omega))\bigcap W^{2,\infty}(0,T;\textbf{H}^1(\Omega))\bigcap L^{\infty}(0,T;\textbf{H}^2(\Omega))$, and $p \in $ \\
$ W^{2,\infty}(0,T;H^1(\Omega))$ , then for the first-order  scheme \eqref{e_model_semi_first_discrete1}-\eqref{e_model_semi_first_discrete5}, we have
\begin{equation*} 
\aligned
& \| \nabla e_{\phi}^{m+1} \|^2 +\| e_{\phi}^{m+1} \|^2 +\Delta t \sum\limits_{n=0}^{m} \| \nabla e_{\mu}^{n+1} \|^2 + \Delta t \sum\limits_{n=0}^{m} \| e_{\mu}^{n+1} \|^2 + |e_r^{m+1}|^2 \\
&+ \|e_{\textbf{u}}^{m+1}\|^2 + \nu \Delta t \sum\limits_{n=0}^{m} \|\nabla\tilde{e}_{\textbf{u}}^{n+1}\|^2 + \Delta t \|\nabla e_p^{m+1}\|^2+  |e_q^{m+1}|^2 \\
&+ \sum\limits_{n=0}^{m}  \| \nabla e_{\phi}^{n+1}- \nabla e_{\phi}^{n} \|^2
+ \sum\limits_{n=0}^{m}  \| e_{\phi}^{n+1}- e_{\phi}^{n} \|^2 +
 \sum\limits_{n=0}^{m}  \|\tilde{e}_{\textbf{u}}^{n+1}-e_{\textbf{u}}^n\|^2 \\
 &+  \sum\limits_{n=0}^{m}  | e_{r}^{n+1}- e_{r}^{n} |^2  +  \sum\limits_{n=0}^{m}  | e_{q}^{n+1}- e_{q}^{n} |^2 \leq C (\Delta t)^2,
 \ \ \ \forall \ 0\leq n\leq N-1,
\endaligned
\end{equation*}
where the positive constant $C$ is independent of $\Delta t$. 
\end{theorem}

\begin{proof}
We observe from \eqref{e_model_semi_first_discrete6} that $\textbf{u} ^{n+1}=P_{\textbf{H}} \tilde{ \textbf{u} }^{n+1}$. Hence, \eqref{PH} implies that
$ \| \textbf{u} ^{n+1} \|_1\le C(\Omega)  \| \tilde{ \textbf{u} }^{n+1} \|_1$.
Using the above, Lemmas \ref{lem: error_estimate_phi}, \ref{lem: error_estimate_u} and \ref{lem: error_estimate_q} leads to
\begin{equation} \label{e_final_u1}
\aligned
\frac{1}{2\Delta t}& ( \| \nabla e_{\phi}^{n+1} \|^2-  \| \nabla e_{\phi}^{n} \|^2 +  \| \nabla e_{\phi}^{n+1}- \nabla e_{\phi}^{n} \|^2 )  + \frac{M}{4} \| \nabla e_{\mu}^{n+1} \|^2 \\
& + \frac{ \gamma+1 }{2 \Delta t} ( \| e_{\phi}^{n+1}\|^2- \| e_{\phi}^{n}\|^2 +\| e_{\phi}^{n+1}-e_{\phi}^{n} \|^2 ) +\frac{M}{4} \| e_{\mu}^{n+1} \|^2 \\
&+ \frac{1}{\Delta t} ( |e_r^{n+1}|^2- |e_r^n|^2 +|e_r^{n+1}-e_r^{n}|^2 ) + \frac{\|e_{\textbf{u}}^{n+1}\|^2-\|e_{\textbf{u}}^n\|^2}{2\Delta t} \\
&  +\frac{\|\tilde{e}_{\textbf{u}}^{n+1}-e_{\textbf{u}}^n\|^2}{2\Delta t}+\frac{\nu} {2} \|\nabla\tilde{e}_{\textbf{u}}^{n+1}\|^2+\frac{\Delta t}{2}( \|\nabla e_p^{n+1}\|^2-\|\nabla e_p^n\|^2) \\
&+ \frac{|e_q^{n+1}|^2-|e_q^n|^2}{2\Delta t} +\frac{|e_q^{n+1}-e_q^n|^2}{2\Delta t}+\frac{1}{2T}|e_q^{n+1}|^2 \\
\leq & (C+ \frac{1}{4K_1} \| \nabla \mu^{n} \|^2) | e_r^{n+1} |^2  + C ( \|e_{\textbf{u}}^n\|^2_1+ \|\textbf{u}^n \|_1^2 )  \|e_{\textbf{u}}^n\|^2   \\
& +C \| \nabla e_{\phi}^{n+1} \|^2   + C \| \nabla e_{\phi}^n\|^2 
+C \| e_{\phi}^{n} \|^2 \\
&+C \| e_{\textbf{u}}^{n+1} \|^2 +\frac{3}{16} M \|e_{\mu}^{n} \|^2+\frac{3}{16}M \| \nabla e_{\mu}^{n} \|^2  +C |e_q^{n+1}|^2\\ 
&+ C (\Delta t)^2 ( \|\nabla e_p^n\|^2+ \|\nabla e_p^{n+1}\|^2 ) + C  (\Delta t )^2.
\endaligned
\end{equation}
 Multiplying \eqref{e_final_u1} by $2\Delta t$ and summing over $n$, $n=0,2,\ldots,m^*$, where $m^*$ is the time step at which $|e_r^{m^*+1}|$ achieves its maximum value, and applying the discrete Gronwall lemma \ref{lem: gronwall1}, we can obtain
\begin{equation} \label{e_final_u2}
\aligned
& \| \nabla e_{\phi}^{m^*+1} \|^2 +\| e_{\phi}^{m^*+1} \|^2 +\Delta t \sum\limits_{n=0}^{m^*} \| \nabla e_{\mu}^{n+1} \|^2 + \Delta t \sum\limits_{n=0}^{m^*} \| e_{\mu}^{n+1} \|^2 + |e_r^{m^*+1}|^2 \\
&+ \|e_{\textbf{u}}^{m^*+1}\|^2 + \nu \Delta t \sum\limits_{n=0}^{m^*} \|\nabla\tilde{e}_{\textbf{u}}^{n+1}\|^2 + \Delta t \|\nabla e_p^{m^*+1}\|^2+  |e_q^{m^*+1}|^2 \\
&+ \sum\limits_{n=0}^{m^*}  \| \nabla e_{\phi}^{n+1}- \nabla e_{\phi}^{n} \|^2
+ \sum\limits_{n=0}^{m^*}  \| e_{\phi}^{n+1}- e_{\phi}^{n} \|^2 +
 \sum\limits_{n=0}^{m^*}  \|\tilde{e}_{\textbf{u}}^{n+1}-e_{\textbf{u}}^n\|^2 \\
 &+  \sum\limits_{n=0}^{m^*}  | e_{r}^{n+1}- e_{r}^{n} |^2  +  \sum\limits_{n=0}^{m^*}  | e_{q}^{n+1}- e_{q}^{n} |^2 \leq  C (\Delta t)^2 ,
\endaligned
\end{equation}
where we use the fact that
  \begin{equation*}\label{e_final_u3}
  \aligned
 2 \Delta t \sum\limits_{n=0}^{m^*}  \frac{1}{4K_1} \| \nabla \mu^{n} \|^2 | e_r^{n+1} |^2  \leq  \frac{1}{2K_1}  | e_r^{m^*+1} |^2\Delta t \sum\limits_{n=0}^{m^*} \| \nabla \mu^{n} \|^2 \leq \frac{1}{2}  | e_r^{m^*+1} |^2,
 \endaligned
\end{equation*} 
which is a direct consequence of   \eqref{e_error1}.

Since $|e_r^{m^*+1}|=\max_{0\le m\le N-1}|e_r^{m+1}|$, \eqref{e_final_u2} also implies
\begin{equation} \label{e_final_u4}
\aligned
& \| \nabla e_{\phi}^{m+1} \|^2 +\| e_{\phi}^{m+1} \|^2 +\Delta t \sum\limits_{n=0}^{m} \| \nabla e_{\mu}^{n+1} \|^2 + \Delta t \sum\limits_{n=0}^{m} \| e_{\mu}^{n+1} \|^2 + |e_r^{m+1}|^2 \\
&+ \|e_{\textbf{u}}^{m+1}\|^2 + \nu \Delta t \sum\limits_{n=0}^{m} \|\nabla\tilde{e}_{\textbf{u}}^{n+1}\|^2 + \Delta t \|\nabla e_p^{m+1}\|^2+  |e_q^{m+1}|^2 \\
&+ \sum\limits_{n=0}^{m}  \| \nabla e_{\phi}^{n+1}- \nabla e_{\phi}^{n} \|^2
+ \sum\limits_{n=0}^{m}  \| e_{\phi}^{n+1}- e_{\phi}^{n} \|^2 +
 \sum\limits_{n=0}^{m}  \|\tilde{e}_{\textbf{u}}^{n+1}-e_{\textbf{u}}^n\|^2 \\
 &+  \sum\limits_{n=0}^{m}  | e_{r}^{n+1}- e_{r}^{n} |^2  +  \sum\limits_{n=0}^{m}  | e_{q}^{n+1}- e_{q}^{n} |^2 \leq C (\Delta t)^2,\quad\forall \ 0\le m\le N-1,
\endaligned
\end{equation}
which implies the desired result.
\end{proof}

The above result indicates that the errors for $(\phi,\mu,\textbf{u},r,q)$ are first-order accurate in various norms. However, it does not provide any error estimate for the pressure. The error estimate for the pressure is very technical so we shall carry it out in the appendix.

\section{Numerical results}
We now provide some numerical experiments to verify our theoretical results.
First we rewrite the total energy in (\ref{e_model_r}) as
\begin{equation}\label{definition of energy_transformed}
\aligned
E(\phi)=\int_{\Omega}\{\frac 1 2|\textbf{u}|^2+\frac{1}{2}|\nabla \phi |^2+( \frac{\gamma}{2}+\frac{\beta}{2\epsilon^2} )\phi^2
+\frac{1}{4\epsilon^2}(\phi^2-1-\beta)^2 -\frac{\beta^2+2\beta}{4\epsilon^2} \}d\textbf{x},
\endaligned
\end{equation}
where $\beta$ is a positive stabilization cosntant to be specified. To apply our first-order scheme \eqref{e_model_semi_first_discrete1}-\eqref{e_model_semi_first_discrete5} and second-order scheme \eqref{e_model_semi_second_discrete1}-\eqref{e_model_semi_second_discrete7} to the system (\ref{e_model_r}), we drop the constant in the free energy and specify $\displaystyle E_1(\phi)=\frac{1}{4\epsilon^2}\int_{\Omega}(\phi^2-1-\beta)^2d\textbf{x}$, and modify (\ref{e_model_semi_first_discrete2}) and \eqref{e_model_semi_second_discrete2} into
\begin{equation}\label{e_model_semi_first_discrete2_modify}
\aligned
 \mu^{n+1}=- \Delta \phi^{n+1} + ( \gamma+ \frac{ \beta }{ \epsilon^2 } ) \phi^{n+1}+ 
 \frac{r^{n+1}}{\sqrt{E_1( \phi^{n})+\delta}}F^{\prime}( \phi^{n}),
\endaligned
\end{equation}

\begin{equation}\label{e_model_semi_second_discrete2_modify}
\aligned
 \mu^{n+1}=- \Delta \phi^{n+1} + ( \gamma+ \frac{ \beta }{ \epsilon^2 } ) \phi^{n+1}+ 
 \frac{r^{n+1}}{\sqrt{E_1( \bar{\phi}^{n+1} )+\delta}}F^{\prime}( \bar{\phi}^{n+1} ).
\endaligned
\end{equation}

Then we can obtain
\begin{equation}\label{e_numerical_1}
\aligned
F^{\prime}(\phi)=\frac{\delta E_1}{\delta \phi}=\frac{1}{\epsilon^2}\phi(\phi^2-1-\beta).
\endaligned
\end{equation}
For simplicity, we define
\begin{flalign*}
\renewcommand{\arraystretch}{1.5}
  \left\{
   \begin{array}{l}
\|f-g\|_{l^{\infty}}=\max\limits_{0\leq n\leq m}\left\{\|f^{n+1}-g^{n+1}\| \right\},\\
\|f-g\|_{l^{2}}=\left(\sum\limits_{n=0}^{m}\Delta t\left\|f^{n+1}-g^{n+1}\right\| ^2\right)^{1/2},\\
\|R-r\|_{ \infty }=\max\limits_{0\leq n\leq m}\{R^{n+1}-r^{n+1}\}. 
\end{array}\right.
\end{flalign*}
 In the following simulation, we choose $ \Omega=(0,1)\times(0,1)$, $\beta=5$, $T=0.1$, $\gamma=1$, $\nu=0.001$, $\epsilon=0.3$, $M=0.001$, $\delta=0$, with the initial condition 
 \begin{equation}
 \begin{split}
& \textbf{u}^0(x,y)=[\sin^2(\pi x)\sin(2\pi y),\, -\sin^2(\pi y)\sin(2\pi x) ],\;p^0(x,y)=0;\\
&\phi^0(x,y)=\cos(\pi x)\cos(\pi y),\; r^0=\sqrt{E_1(\phi^0)+\delta},\; q^0=1.
\end{split}
\end{equation}
 The spatial discretization is based on the MAC scheme on the staggered grid with $N_x=N_y=160$ so that the spatial discretization error is negligible compared to the time discretization error for the time steps used in the simulation.
 
 We measure Cauchy error due to the fact that we do not have possession of exact solution. Specifically, the error between two different time step sizes $\Delta t$ and $\frac{ \Delta t }{2}$ is calculated by $\|e_{\zeta}\|=\|\zeta_{\Delta t}-\zeta_{\Delta t/2}\|$.
We present numerical results for the first- and second-order schemes \eqref{e_model_semi_first_discrete1}-\eqref{e_model_semi_first_discrete5} and \eqref{e_model_semi_second_discrete1}-\eqref{e_model_semi_second_discrete7} 
in Tables \ref{table1_example1}-\ref{table2_example2}. From Tables \ref{table1_example1} and \ref{table2_example1}, one can easily obtain that the numerical results give solid supporting evidence for the expected first-order convergence in time of the fully decoupled MSAV scheme \eqref{e_model_semi_first_discrete1}-\eqref{e_model_semi_first_discrete5} for the Cahn-Hilliard-Navier-Stokes phase field model, which are consistent with error estimates in Theorems \ref{thm: error_estimate_u_phi} and \ref{thm: error_estimate_p}. 

\begin{table}[htbp]
\renewcommand{\arraystretch}{1.1}
\small
\centering
\caption{Errors and convergence rates with the first-order scheme \eqref{e_model_semi_first_discrete1}-\eqref{e_model_semi_first_discrete7} }\label{table1_example1}
\begin{tabular}{p{1cm}p{1.5cm}p{0.7cm}p{1.5cm}p{0.7cm}p{1.5cm}p{0.7cm}} \hline
$\Delta t$    &$\|e_{\phi}\|_{l^{\infty}}$    &Rate &$\|\nabla e_{\phi}\|_{l^{\infty}}$   &Rate   
&$ |e_{r} |_{\infty}$    &Rate   \\ \hline
$ 2^{-3} $    &3.52E-3     & ---      &2.52E-2         &---      &1.87E-3         &---      \\
$ 2^{-4} $    &2.44E-3     &0.53    &1.63E-2         &0.63  &8.15E-4         &1.20   \\
$ 2^{-5} $    &1.43E-3     &0.77    &9.41E-3         &0.80  &3.88E-4         &1.07   \\
$ 2^{-6} $    &7.74E-4     &0.89    &5.06E-3         &0.89  &1.91E-4         &1.02   \\
\hline
\end{tabular}
\end{table}

\begin{table}[htbp]
\renewcommand{\arraystretch}{1.1}
\small
\centering
\caption{Errors and convergence rates with the first-order scheme \eqref{e_model_semi_first_discrete1}-\eqref{e_model_semi_first_discrete7} }\label{table2_example1}
\begin{tabular}{p{1cm}p{1.5cm}p{0.7cm}p{1.5cm}p{0.7cm}p{1.5cm}p{0.7cm}p{1.5cm}p{0.7cm}} \hline
$\Delta t$    &$\| e_{\textbf{u}}\|_{l^{\infty}}$    &Rate  &$\|\nabla \tilde{e}_{\textbf{u}}\|_{l^{2}}$    &Rate &$ \| e_{p} \|_{l^{2}}$   &Rate   &$ |e_{q } |_{\infty}$    &Rate   \\ \hline
$ 2^{-3} $    &4.14E-2     & ---      &1.35E-3         &---      &1.47E-2         &---   
 &1.07E-2         &---   \\
$ 2^{-4} $    &2.18E-2     &0.93    &6.80E-4         &0.99  &6.34E-3         &1.21 
&5.53E-3         &0.96   \\
$ 2^{-5} $    &1.14E-2     &0.94    &3.57E-4         &0.93  &3.03E-3         &1.07  
&2.81E-3         &0.98  \\
$ 2^{-6} $    &5.83E-3     &0.96    &1.85E-4         &0.95  &1.50E-3         &1.01  
&1.41E-3         &0.99  \\
\hline
\end{tabular}
\end{table}

\begin{table}[htbp]
\renewcommand{\arraystretch}{1.1}
\small
\centering
\caption{Errors and convergence rates with the second-order scheme \eqref{e_model_semi_second_discrete1}-\eqref{e_model_semi_second_discrete7} }\label{table1_example2}
\begin{tabular}{p{1cm}p{1.5cm}p{0.7cm}p{1.5cm}p{0.7cm}p{1.5cm}p{0.7cm}} \hline
$\Delta t$    &$\|e_{\phi}\|_{l^{\infty}}$    &Rate &$\|\nabla e_{\phi}\|_{l^{\infty}}$   &Rate   
&$ |e_{r} |_{\infty}$    &Rate   \\ \hline
$ 2^{-3} $    &1.62E-3     & ---      &1.12E-2         &---      &7.85E-4         &---      \\
$ 2^{-4} $    &3.75E-4     &2.11    &2.57E-3         &2.12  &1.67E-4         &2.23   \\
$ 2^{-5} $    &8.97E-5     &2.06    &6.11E-4         &2.07  &3.95E-5         &2.08   \\
$ 2^{-6} $    &2.17E-5     &2.05    &1.48E-4         &2.05  &9.71E-6         &2.03   \\
\hline
\end{tabular}
\end{table}

\begin{table}[htbp]
\renewcommand{\arraystretch}{1.1}
\small
\centering
\caption{Errors and convergence rates with the second-order scheme \eqref{e_model_semi_second_discrete1}-\eqref{e_model_semi_second_discrete7} }\label{table2_example2}
\begin{tabular}{p{1cm}p{1.5cm}p{0.7cm}p{1.5cm}p{0.7cm}p{1.5cm}p{0.7cm}p{1.5cm}p{0.7cm}} \hline
$\Delta t$    &$\| e_{\textbf{u}}\|_{l^{\infty}}$    &Rate  &$\|\nabla \tilde{e}_{\textbf{u}}\|_{l^{2}}$    &Rate &$ \| e_{p} \|_{l^{2}}$   &Rate   &$ |e_{q } |_{\infty}$    &Rate   \\ \hline
$ 2^{-3} $    &7.72E-3     & ---      &1.20E-3         &---      &3.63E-2         &---   
 &2.12E-3         &---   \\
$ 2^{-4} $    &1.50E-3     &2.36    &2.98E-4         &2.01  &1.28E-2         &1.50 
&5.01E-4         &2.08   \\
$ 2^{-5} $    &3.35E-4     &2.17    &7.97E-5         &1.90  &4.56E-3         &1.49  
&1.23E-4         &2.03  \\
$ 2^{-6} $    &8.33E-5     &2.01    &2.21E-5         &1.85  &1.62E-3         &1.50  
&3.03E-5         &2.01  \\
\hline
\end{tabular}
\end{table}

\section{Concluding remarks}
The Cahn-Hilliard-Navier-Stokes  phase field model is a highly  coupled nonlinear system whose energy dissipation relies on delicate cancellations of nonlinear interactions. 
We constructed in this paper efficient time discretization schemes for the Cahn-Hilliard-Navier-Stokes  phase field model by combining the MSAV approach to deal with  the various nonlinear terms and the standard or rotational pressure-correction to deal with the coupling of pressure and velocity. The resulting first- and second-order schemes are fully decoupled, linear, unconditional energy stable and only require solving several elliptic equations with constant coefficients at each time step. So they are very efficient and easy to implement.
 We also carried out a rigorous error analysis for the   first-order scheme and derived optimal error estimates for all relevent function in different norms.  

While we only carried out error analysis for the first-order scheme, it is believed that second-order error estimates can be derived, albeit very tedious, by combing the approach in this paper  with the techniques used to derive second-order error estimates for the rotational pressure-correction scheme in \cite{Gue.S04}. On the other hand, 
we have only considered time discretization in this work. While the stability proofs and  error estimates are based on weak formulations with simple test functions, it is still a big challenge  to extend this approach to fully discrete schemes with  properly formulated spatial discretization. These tasks will be left as subjects of  future endeavor.

\begin{appendix}
\section{Error estimate for the pressure}
Theorem \ref{thm: error_estimate_u_phi} does not lead to any pressure error estimate. As in the error analysis of projection type schemes, the pressure error has to be obtained through the inf-sup condition 
\begin{equation}\label{infsup}
\|p\|_{L^2(\Omega)/\mathbb{R}} \leq \sup_{\textbf{v} \in \textbf{H}^1_0(\Omega)} \frac{
(p,\nabla\cdot \textbf{v}) }{ \|\nabla \textbf{v} \| }=\sup_{\textbf{v} \in \textbf{H}^1_0(\Omega)} \frac{
-(\nabla p, \textbf{v}) }{ \|\nabla \textbf{v} \| },
\end{equation} 
which is obviously true in the space continuous case. Therefore, we need to estimate $(\nabla e_p^{n+1},\textbf{v})$ which requires additional estimates.
 
\begin{theorem} \label{thm: error_estimate_p}
Assuming $ \phi \in W^{3,\infty}(0,T; L^2(\Omega)) \bigcap W^{2,\infty}(0,T; H^2(\Omega))  $, $ \mu \in W^{1,\infty}(0,T; H^1(\Omega)) $ $ \bigcap L^{\infty}(0,T; H^2(\Omega)) $, $\textbf{u}\in W^{3,\infty}(0,T;\textbf{L}^2(\Omega))\bigcap W^{2,\infty}(0,T;\textbf{H}^1(\Omega))\bigcap L^{\infty}(0,T;\textbf{H}^2(\Omega))$, and $p \in $ \\
$ W^{2,\infty}(0,T;H^1(\Omega))$ , then for the first-order  scheme \eqref{e_model_semi_first_discrete1}-\eqref{e_model_semi_first_discrete5}, we have
\begin{equation*}
\aligned
&\Delta t\sum\limits_{n=0}^m\|e_{p}^{n+1}\|^2_{L^2(\Omega)/R}  \leq C(\Delta t)^2,  \ \ \ \forall \ 0\leq m\leq N-1,
\endaligned
\end{equation*}
where $C$ is a positive constant  independent of $\Delta t$.
\end{theorem}

\begin{proof} The proof will be carried out in three steps.

\noindent{\bf Step 1.} 
We first establish an estimate on $ |d_t q^{n+1}|$, which is an essential part of the proof. Multiplying both sides of \eqref{e_q1} with $d_te_q^{n+1}$ leads to
\begin{equation}\label{e_final_p1}
\aligned
& | d_te_q^{n+1}|^2 + \frac{|e_q^{n+1}|^2-|e_q^n|^2}{2T\Delta t}+\frac{|e_q^{n+1}-e_q^n|^2}{2T\Delta t}\\
 =& \exp( \frac{t^{n+1}}{T} ) d_te_q^{n+1} (\textbf{u}^n\cdot\nabla \textbf{u}^n, \tilde{e}_{\textbf{u}}^{n+1} )- \exp( \frac{t^{n+1}}{T} ) d_te_q^{n+1} \left( \textbf{u}^n\cdot\nabla ( 
 \textbf{u}(t^{n+1})-\textbf{u}^n), \textbf{u}(t^{n+1}) \right)\\
 &- \exp( \frac{t^{n+1}}{T} ) d_te_q^{n+1} \left( ( 
 \textbf{u}(t^{n+1})-\textbf{u}^n)\cdot \nabla \textbf{u}(t^{n+1}), \textbf{u}(t^{n+1}) \right)+\textbf{R}_{q}^{n+1}d_te_q^{n+1}.
\endaligned
\end{equation}
From Theorem \ref{thm: error_estimate_u_phi}, we have
\begin{equation}\label{e_final_p2}
\aligned
&\Delta t\sum\limits_{n=0}^m\|\nabla\tilde{e}_{\textbf{u}}^{n+1}\|^2 \leq C(\Delta t)^2,
\endaligned
\end{equation}
which implies that 
\begin{equation}\label{e_final_p3}
\aligned
&\| \textbf{u}^{n+1} \|_{1} \leq C\| \tilde{\textbf{u}}^{n+1} \|_{1} \leq C(\Delta t)^{1/2} \leq K_2.
\endaligned
\end{equation}
Then the first term on the right hand side of \eqref{e_final_p1} can be bounded by 
\begin{equation}\label{e_final_p4}
\aligned
\exp( \frac{t^{n+1}}{T} ) & d_te_q^{n+1} (\textbf{u}^n\cdot\nabla \textbf{u}^n, \tilde{e}_{\textbf{u}}^{n+1} ) \\
\leq &(1+c_1)c_2 \| \textbf{u}^n \|^{1/2}\| \textbf{u}^n \|_1^{1/2} \|\textbf{u}^n \|^{1/2}\| \textbf{u}^n \|_1^{1/2} \| \nabla \tilde{e}_{\textbf{u}}^{n+1} \| |d_te_q^{n+1}| \\
\leq &\frac{1}{6}|d_te_q^{n+1}|^2+C \| \nabla \tilde{e}_{\textbf{u}}^{n+1} \|^2.
\endaligned
\end{equation}
The second and third terms on the right hand side of \eqref{e_final_p1} can be estimated by
\begin{equation}\label{e_final_p5}
\aligned
-\exp( \frac{t^{n+1}}{T} ) &d_te_q^{n+1}  \left( \textbf{u}^n\cdot\nabla ( 
 \textbf{u}(t^{n+1})-\textbf{u}^n), \textbf{u}(t^{n+1}) \right) \\
& - \exp( \frac{t^{n+1}}{T} )  d_te_q^{n+1}  \left( ( 
 \textbf{u}(t^{n+1})-\textbf{u}^n)\cdot \nabla \textbf{u}(t^{n+1}), \textbf{u}(t^{n+1}) \right)\\
\leq &\frac{1}{6}|d_te_q^{n+1}|^2+C \| e_{\textbf{u}}^{n} \|^2+C(\Delta t)^2.
\endaligned
\end{equation}
The last term on the right hand side of \eqref{e_final_p1} can be estimated by
\begin{equation}\label{e_final_p6}
\aligned
\textbf{R}_{q}^{n+1}d_te_q^{n+1} \leq \frac{1}{6} |d_te_q^{n+1}|^2+C \| q\|_{W^{2,\infty}(0,T)} ^2 (\Delta t )^2 .
\endaligned
\end{equation}
Finally combining \eqref{e_final_p1} with \eqref{e_final_p2}-\eqref{e_final_p6}  results in 
\begin{equation}\label{e_final_p7}
\aligned
 | d_te_q^{n+1}|^2 &+ \frac{|e_q^{n+1}|^2-|e_q^n|^2}{2T\Delta t}+\frac{|e_q^{n+1}-e_q^n|^2}{2T\Delta t}\\
 \leq &\frac{1}{2}|d_te_q^{n+1}|^2+C \| \nabla \tilde{e}_{\textbf{u}}^{n+1} \|^2+C \| e_{\textbf{u}}^{n} \|^2+C \| q\|_{W^{2,\infty}(0,T)} ^2 (\Delta t )^2 .
\endaligned
\end{equation}
Multiplying \eqref{e_final_p7} by $2 \Delta t$ and summing up for $n$ from $0$ to $m$, and Recalling Theorem \ref{thm: error_estimate_u_phi}, we can obtain
\begin{equation}\label{e_final_p8}
\aligned
 \Delta t\sum\limits_{n=0}^{m} | d_te_q^{n+1}|^2 
 \leq C \Delta t\sum\limits_{n=0}^{m} \| \nabla \tilde{e}_{\textbf{u}}^{n+1} \|^2 +C \Delta t\sum\limits_{n=0}^{m}\| e_{\textbf{u}}^{n} \|^2 
 + C(\Delta t)^2 \leq C(\Delta t)^2.
\endaligned
\end{equation}

\noindent{\bf Step 2.} 
Next we establish estimates on $ \| d_te_{\phi}^{n+1}\|$, $ \| d_t e_{\mu}^{n+1} \| $ and  $ \| d_te_r^{n+1} \| $. 

Define $d_{tt}G^{n+1}=\frac{ d_tG^{n+1}- d_tG^{n} }{ \Delta t}$.  Taking the difference of two consecutive steps in \eqref{e_phi2} and \eqref{e_phi5} we have
\begin{equation}\label{e_final_p_Step2_1}
\aligned
d_{tt} e_{\phi}^{n+1}- M \Delta d_te_{\mu}^{n+1} =  d_tR_{\phi}^{n+1} +
d_tE_{N}^{n+1}, 
\endaligned
\end{equation}
\begin{equation}\label{e_final_p_Step2_2}
\aligned
 d_t e_{\mu}^{n+1}= & - \Delta d_t e_{\phi}^{n+1} +   \gamma d_t e_{\phi}^{n+1}+ d_t \left( \frac{e_r^{n+1}}{ \sqrt{E_1( \phi^{n})+\delta} }F^{\prime}( \phi^{n}) \right)+ d_t E_{F}^{n+1}.
\endaligned
\end{equation}
Taking the inner product of \eqref{e_final_p_Step2_1} with $ d_t e_{\phi}^{n+1} $ leads to
\begin{equation}\label{e_final_p_Step2_3}
\aligned
 \frac{ \| d_t e_{\phi}^{n+1} \|^2-\| d_t e_{\phi}^n \|^2 }{ 2\Delta t}& + \frac{ \| d_t e_{\phi}^{n+1} - d_t e_{\phi}^n \|^2 }{ 2\Delta t} +M ( \nabla d_te_{\mu}^{n+1}, \nabla d_t e_{\phi}^{n+1} ) \\
& =(  d_tR_{\phi}^{n+1} , d_t e_{\phi}^{n+1} ) + (d_tE_{N}^{n+1},  d_t e_{\phi}^{n+1} ).
\endaligned
\end{equation}
Taking the inner products of \eqref{e_final_p_Step2_2} with $ \frac{M}{2} d_t e_{\mu}^{n+1} $ and $ \frac{M}{2} d_t \Delta e_{\phi}^{n+1} $ gives
\begin{equation}\label{e_final_p_Step2_4}
\aligned
 \frac{M}{2} \| d_t e_{\mu}^{n+1} \|^2=& \frac{M}{2} ( \nabla d_te_{\mu}^{n+1}, \nabla d_t e_{\phi}^{n+1} )
+ \frac{M}{2} \gamma (d_t e_{\phi}^{n+1}, d_t e_{\mu}^{n+1} ) \\
&+ \frac{M}{2} \left( d_t \left( \frac{e_r^{n+1}}{ \sqrt{E_1( \phi^{n})+\delta} }F^{\prime}( \phi^{n}) \right), d_t e_{\mu}^{n+1} \right) + \frac{M}{2} ( d_t E_{F}^{n+1}, d_t e_{\mu}^{n+1} ),
\endaligned
\end{equation}
and 
\begin{equation}\label{e_final_p_Step2_4_ADD}
\aligned
 & \frac{M}{2} \| d_t \Delta e_{\phi}^{n+1} \|^2 + \frac{M}{2} \gamma  \|  d_t \nabla e_{\phi}^{n+1} \|^2  = \frac{M}{2} ( \nabla d_te_{\mu}^{n+1}, \nabla d_t e_{\phi}^{n+1} ) \\
 &\ \ \ \ 
 + \frac{M}{2} \left( d_t \left( \frac{e_r^{n+1}}{ \sqrt{E_1( \phi^{n})+\delta} }F^{\prime}( \phi^{n}) \right),  d_t \Delta e_{\phi}^{n+1} \right) 
+ \frac{M}{2} ( d_t E_{F}^{n+1},  d_t \Delta e_{\phi}^{n+1}  ).
\endaligned
\end{equation}
Adding \eqref{e_final_p_Step2_3}, \eqref{e_final_p_Step2_4} and \eqref{e_final_p_Step2_4_ADD}, we obtain
\begin{equation}\label{e_final_p_Step2_5}
\aligned
& \frac{ \| d_t e_{\phi}^{n+1} \|^2-\| d_t e_{\phi}^n \|^2 }{ 2\Delta t} + \frac{ \| d_t e_{\phi}^{n+1} - d_t e_{\phi}^n \|^2 }{ 2\Delta t} + \frac{M}{2} \| d_t e_{\mu}^{n+1} \|^2 + \frac{M}{2} \| d_t \Delta e_{\phi}^{n+1} \|^2 + \frac{M}{2} \gamma  \|  d_t \nabla e_{\phi}^{n+1} \|^2 \\
=& \frac{M}{2} \gamma (d_t e_{\phi}^{n+1}, d_t e_{\mu}^{n+1} ) + \frac{M}{2} \left( d_t \left( \frac{e_r^{n+1}}{ \sqrt{E_1( \phi^{n})+\delta} }F^{\prime}( \phi^{n}) \right), d_t e_{\mu}^{n+1} + d_t \Delta e_{\phi}^{n+1} \right) \\
&+ \frac{M}{2} ( d_t E_{F}^{n+1}, d_t e_{\mu}^{n+1} +d_t \Delta e_{\phi}^{n+1} ) +(  d_tR_{\phi}^{n+1} , d_t e_{\phi}^{n+1} ) + (d_tE_{N}^{n+1},  d_t e_{\phi}^{n+1} ).
\endaligned
\end{equation}
Using Cauchy-Schwarz inequality, the first term on the right hand side of \eqref{e_final_p_Step2_5} can be estimated by
\begin{equation}\label{e_final_p_Step2_6}
\aligned
& \frac{M}{2} \gamma (d_t e_{\phi}^{n+1}, d_t e_{\mu}^{n+1} ) \leq \frac{M}{12} \| d_t e_{\mu}^{n+1} \|^2 +C \| d_t e_{\phi}^{n+1} \|^2 .
\endaligned
\end{equation}
By using Theorem \ref{thm: error_estimate_u_phi}, we should first give the boundedness for $ \|d_t \phi^{n+1} \|_1 $, $ |d_t r^{n+1}| $ and $ \| \nabla \mu^{n+1} \|$ to continue the following error estimate.
\begin{equation}\label{e_final_p_Step2_7}
\aligned
& \|d_t \phi^{n+1} \|_1 \leq \|d_t \phi(t^{n+1}) \|_1+ (\Delta t)^{-1} \| e_{\phi}^{n+1}-e_{\phi}^n \|_1 \leq C_1, \ \ \forall \ 0\leq n\leq N-1,
\endaligned
\end{equation}
\begin{equation}\label{e_final_p_Step2_7_ADD}
\aligned
& |d_t r^{n+1} | \leq |d_t r(t^{n+1}) |+ (\Delta t)^{-1} | e_{r}^{n+1}-e_{r}^n | \leq C_2, \ \ \forall \ 0\leq n\leq N-1,
\endaligned
\end{equation}
\begin{equation}\label{e_final_p_Step2_7_ADD2}
\aligned
& \| \nabla \mu^{n+1}  \| \leq \| \nabla \mu( t^{n+1} )\| + \| \nabla e_{\mu}^{n+1} \| \leq C+ (\Delta t) ^{1/2} \leq C_3, \ \ \forall \ 0\leq n\leq N-1.
\endaligned
\end{equation}
Recalling Lemma \ref{lem_phi_H2_boundness} and \eqref{e_error1}, the second term on the right hand side of \eqref{e_final_p_Step2_5} can be bounded by
\begin{equation}\label{e_final_p_Step2_8}
\aligned
\frac{M}{2} &\left( d_t \left( \frac{e_r^{n+1}}{ \sqrt{E_1( \phi^{n})+\delta} }F^{\prime}( \phi^{n}) \right), d_t e_{\mu}^{n+1} + d_t \Delta e_{\phi}^{n+1} \right) \\
=&\frac{M}{2} \left( \frac{ F^{\prime}( \phi^{n}) } { \sqrt{E_1( \phi^{n})+\delta} } d_te_r^{n+1} , d_t e_{\mu}^{n+1} +d_t \Delta e_{\phi}^{n+1} \right) \\
& + \frac{M}{2} \left( \frac{ e_r^n } { \sqrt{E_1( \phi^{n})+\delta} }d_t F^{\prime}( \phi^{n})  , d_t e_{\mu}^{n+1} + d_t \Delta e_{\phi}^{n+1} \right) \\
&+ \frac{M}{2} \left( e_r^n F^{\prime}( \phi^{n}) d_t \frac{1} {\sqrt{E_1( \phi^{n})+\delta} }  , d_t e_{\mu}^{n+1} + d_t \Delta e_{\phi}^{n+1} \right) \\
\leq & \frac{M}{12} \|  d_t e_{\mu}^{n+1} \|^2 + \frac{M}{12} \| d_t \Delta e_{\phi}^{n+1} \|^2
+C\| d_t e_r^{n+1}\|^2 +C \| d_t \phi^n \|^2 |e_r^n|^2 .
\endaligned
\end{equation}
Denote $G(\phi^n)= \frac{ F^{\prime}( \phi^{n}) }{\sqrt{E_1( \phi^{n})+\delta}} $ and 
suppose $F(\phi) \in C^3(\mathbb{R})$, we have
\begin{equation}\label{e_final_p_Step2_9}
\aligned
| d_t &\left( G( \phi^{n})- G(\phi(t^{n+1}) ) \right) | \\
=&\frac{1}{\Delta t} | \int_{0}^1 ( \phi^{n}- \phi(t^{n+1}) ) G^{ \prime}(s\phi^n+ (1-s)\phi(t^{n+1}) ) ds \\
&-\int_{0}^1 ( \phi^{n-1}- \phi(t^{n}) ) G^{ \prime}(s\phi^{n-1}+ (1-s)\phi(t^{n}) ) ds | \\
=& | d_t  ( \phi^{n}- \phi(t^{n+1}) )  \int_{0}^1 G^{ \prime}(s\phi^n+ (1-s)\phi(t^{n+1}) ) ds \\
&+ ( \phi^{n-1}- \phi(t^{n}) )  \int_{0}^1 d_t G^{ \prime} (s\phi^n+ (1-s)\phi(t^{n+1}) ) ds | \\
\leq & C |d_te_{\phi} ^n | + C |\phi_{tt}|_{L^{\infty}(0,T)} \Delta t + C(  |e_{\phi} ^{n-1} | + C |\phi_{t}|_{L^{\infty}(0,T)} \Delta t )( | d_t \phi^n |+ |d_t \phi(t^{n+1}) |). 
\endaligned
\end{equation}
Then using above equation and \eqref{e_EF}, the third term on the right hand side of \eqref{e_final_p_Step2_5} can be transformed into
\begin{equation}\label{e_final_p_Step2_10}
\aligned
& \frac{M}{2} ( d_t E_{F}^{n+1}, d_t e_{\mu}^{n+1} +d_t \Delta e_{\phi}^{n+1} ) \\
=& \frac{M}{2} \left( d_t r(t^{n+1}) ( \frac{ F^{\prime}( \phi^{n}) }{\sqrt{E_1( \phi^{n})+\delta}} -
 \frac{ F^{\prime}(\phi(t^{n+1})) }{ \sqrt{E_1(\phi(t^{n+1}) )+\delta} } ) ,  d_t e_{\mu}^{n+1} +d_t \Delta e_{\phi}^{n+1} \right) \\
 &+ \frac{M}{2} r(t^n) \left( d_t ( \frac{ F^{\prime}( \phi^{n}) }{\sqrt{E_1( \phi^{n})+\delta}} -
 \frac{ F^{\prime}(\phi(t^{n+1})) }{ \sqrt{E_1(\phi(t^{n+1}) )+\delta} } ) ,  d_t e_{\mu}^{n+1} +d_t \Delta e_{\phi}^{n+1} \right) \\
 \leq &  \frac{M}{12} \|  d_t e_{\mu}^{n+1} \|^2 + \frac{M}{ 12 } \| d_t \Delta e_{\phi}^{n+1} \| ^2 +C \|d_te_{\phi} ^n \|^2 + C \| e_{\phi}^n \|^2  \\
 & + C \| e_{\phi}^{n-1} \|^2  +  C \|\phi \|_{W^{2,\infty}(0,T; L^2(\Omega))} ^2 (\Delta t)^2 .
\endaligned
\end{equation}
Recalling \eqref{e_phi1}, the fourth term on the right hand side of \eqref{e_final_p_Step2_5} can be bounded by
\begin{equation}\label{e_final_p_Step2_11}
\aligned
(  d_tR_{\phi}^{n+1} , d_t e_{\phi}^{n+1} ) =&-(   d_t \int_{0}^1 s \phi_{tt}( st^{n+1}+ (1-s) t^n ) ds \Delta t, d_t e_{\phi}^{n+1} ) \\
\leq &C \|  d_t e_{\phi}^{n+1} \|^2 + C \| \phi \|_{W^{3,\infty}(0,T; L^2( \Omega ))} ^2 (\Delta t)^2.
\endaligned
\end{equation}
Using \eqref{e_phi13}, the last term on the right hand side of \eqref{e_final_p_Step2_5} can be recast into
\begin{equation}\label{e_final_p_Step2_12}
\aligned
& (d_tE_{N}^{n+1},  d_t e_{\phi}^{n+1} ) \\
=& \left( d_t ( \frac{r(t^{n+1}) }{ \sqrt{E_1(\phi(t^{n+1}))+\delta} } \nabla \cdot ( \textbf{u}(t^{n+1})\phi(t^{n+1})- \textbf{u}(t^{n})\phi(t^{n}) ) ),  d_t e_{\phi}^{n+1} \right) \\
&+ \left( d_t ( ( \frac{ r(t^{n+1}) }{ \sqrt{E_1(\phi(t^{n+1}))+\delta} }-\frac{ r^{n+1} }{ \sqrt{E_1(\phi^n)+\delta} } )  \nabla \cdot ( \textbf{u}(t^{n})\phi(t^{n}) ) ), d_t e_{\phi}^{n+1} \right) \\
&- \left( d_t (\frac{r^{n+1}}{\sqrt{E_1( \phi^{n} )+\delta}} \nabla  \cdot  ( \textbf{u}(t^{n}) e_{\phi}^n + e_{ \textbf{u} }^n \phi^n ) ) ,  d_t e_{\phi}^{n+1}  \right) .
\endaligned
\end{equation}
The first term on the right hand side of \eqref{e_final_p_Step2_12} can be estimated by
\begin{equation}\label{e_final_p_Step2_13}
\aligned
& \left( d_t ( \frac{r(t^{n+1}) }{ \sqrt{E_1(\phi(t^{n+1}))+\delta} } \nabla \cdot ( \textbf{u}(t^{n+1})\phi(t^{n+1})- \textbf{u}(t^{n})\phi(t^{n}) ) ),  d_t e_{\phi}^{n+1} \right) \\
= & d_t \frac{r(t^{n+1}) }{ \sqrt{E_1(\phi(t^{n+1}))+\delta} } \left(  \nabla \cdot ( \textbf{u}(t^{n+1})\phi(t^{n+1})- \textbf{u}(t^{n})\phi(t^{n})  ),  d_t e_{\phi}^{n+1} \right) \\
&+ \frac{r(t^{n}) }{ \sqrt{E_1(\phi(t^{n}))+\delta} } \left(  \nabla \cdot d_t ( \textbf{u}(t^{n+1})\phi(t^{n+1})- \textbf{u}(t^{n})\phi(t^{n})  ),  d_t e_{\phi}^{n+1} \right) \\
\leq & \frac{M\gamma}{8} \| d_t \nabla e_{\phi}^{n+1} \|^2 + C( \| \phi \|_{W^{2,\infty}(0,T;H^1(\Omega) ) } ^2 +   \| \textbf{u} \|_{W^{2,\infty}(0,T;H^1(\Omega) ) } ^2 ) (\Delta t)^2 \\
& + C |r|_{W^{2,\infty}(0,T)}  (\Delta t)^2.
\endaligned
\end{equation}
Using \eqref{e_final_p_Step2_7_ADD}, the second term on the right hand side of \eqref{e_final_p_Step2_12} can be bounded by
\begin{equation}\label{e_final_p_Step2_14}
\aligned
&  \left( d_t ( ( \frac{ r(t^{n+1}) }{ \sqrt{E_1(\phi(t^{n+1}))+\delta} }-\frac{ r^{n+1} }{ \sqrt{E_1(\phi^n)+\delta} } )  \nabla \cdot ( \textbf{u}(t^{n})\phi(t^{n}) ) ), d_t e_{\phi}^{n+1} \right) \\
=& - ( d_te_r^{n+1} \frac{ 1 } {  \sqrt{E_1(\phi(t^{n+1}))+\delta}  } 
+ e_r^{n}d_t \frac{1}{  \sqrt{E_1(\phi(t^{n+1}))+\delta}  } )  \left(  \nabla \cdot ( \textbf{u}(t^{n})\phi(t^{n}) ), d_t e_{\phi}^{n+1} \right) \\
&+ d_t r^{n+1} ( \frac{1}{  \sqrt{E_1(\phi(t^{n+1}))+\delta}  }- \frac{1}{  \sqrt{E_1(\phi^n)+\delta}  }  ) \left(  \nabla \cdot ( \textbf{u}(t^{n})\phi(t^{n}) ), d_t e_{\phi}^{n+1} \right) \\
&+ r^n d_t ( \frac{1}{  \sqrt{E_1(\phi(t^{n+1}))+\delta}  }- \frac{1}{  \sqrt{E_1(\phi^n)+\delta}  }  ) \left(  \nabla \cdot ( \textbf{u}(t^{n})\phi(t^{n}) ), d_t e_{\phi}^{n+1} \right) \\
&+( \frac{ r(t^{n}) }{ \sqrt{E_1(\phi(t^{n}))+\delta} }-\frac{ r^{n} }{ \sqrt{E_1(\phi^{n-1})+\delta} } ) \left( d_t \nabla \cdot ( \textbf{u}(t^{n})\phi(t^{n}) ), d_t e_{\phi}^{n+1} \right) \\
\leq & \frac{M\gamma}{8} \| d_t \nabla e_{\phi}^{n+1} \|^2 + C \| d_t e_{\phi}^{n} \|^2 + C | d_te_r^{n+1} |^2 +  C | e_r^{n} |^2 + C \| e_{\phi}^{n} \|^2 + C \| e_{\phi}^{n-1} \|^2 \\
&+C( \| \phi \|_{W^{1,\infty}(0,T; H^1(\Omega))} ^2 +   \| \textbf{u} \|_{W^{1,\infty}(0,T;H^1(\Omega))} ^2+ |r|_{W^{1,\infty}(0,T)}^2 )(\Delta t)^2 .
\endaligned
\end{equation}
Using \eqref{e_Soblev} and applying integration by parts, the last term on the right hand side of \eqref{e_final_p_Step2_12} can be estimated by
\begin{equation}\label{e_final_p_Step2_15}
\aligned
&- \left( d_t (\frac{r^{n+1}}{\sqrt{E_1( \phi^{n} )+\delta}} \nabla  \cdot  ( \textbf{u}(t^{n}) e_{\phi}^n + e_{ \textbf{u} }^n \phi^n ) ), d_t e_{\phi}^{n+1}  \right) \\
=& d_t \frac{r^{n+1}}{\sqrt{E_1( \phi^{n} )+\delta}} \left( \textbf{u}(t^{n}) e_{\phi}^n + e_{ \textbf{u} }^n \phi^n , d_t  \nabla   e_{\phi}^{n+1}  \right) \\
&+ \frac{r^{n}}{\sqrt{E_1( \phi^{n-1} )+\delta}} \left( d_t( \textbf{u}(t^{n}) e_{\phi}^n + e_{ \textbf{u} }^n \phi^n ) , d_t  \nabla e_{\phi}^{n+1}  \right) \\
\leq  &C (\| \textbf{u}(t^{n}) \|_{L^4} \| e_{\phi}^n \|_{L^4}+ \| e_{ \textbf{u} }^n \|_{L^4} \| \phi^n \|_{L^4} ) \| d_t  \nabla   e_{\phi}^{n+1}  \| \\
&+ C ( \| d_t \textbf{u}(t^{n}) \|_{L^4} \| e_{\phi}^n \|_{L^4}  + \| \textbf{u}(t^{n-1}) \|_{L^4} \| d_t e_{\phi}^n \|_{L^4} ) \| d_t  \nabla   e_{\phi}^{n+1}  \| \\
&+ C  \| d_t e_{ \textbf{u} }^n \| \| \phi^n \|_{L^4}  \| d_t  \nabla   e_{\phi}^{n+1}  \|_{L^4} 
+C \| e_{ \textbf{u} }^{n-1} \|_{L^4} \| d_t\phi^n \|_{L^4} \| d_t  \nabla   e_{\phi}^{n+1}  \| \\
\leq & \frac{M}{ 12 } \| d_t  \Delta e_{\phi}^{n+1} \|^2  +\frac{M\gamma}{8} \| d_t \nabla e_{\phi}^{n} \|^2 + C \| d_t e_{\phi}^{n+1} \|^2 + C \| d_t e_{\phi}^n \|^2  \\
& +C \| \textbf{u} \|_{W^{1,\infty}(0,T; H^1(\Omega))} ^2 ( \|e_{\phi}^n \|^2 +  \| \nabla e_{\phi}^n \|^2 ) + C \| e_{ \textbf{u} }^n \| ^2+ C \| \nabla e_{ \textbf{u} }^n \| ^2 \\
& + C \| e_{ \textbf{u} }^{n-1} \| ^2+ C \| \nabla e_{ \textbf{u} }^{n-1} \| ^2
+ C \| d_t e_{ \textbf{u} }^{n} \|^2 ,
\endaligned
\end{equation}
where the last inequality holds by the fact that
\begin{equation}\label{e_final_p_Step2_16}
\aligned
&  \| d_t \nabla e_{\phi}^{n+1} \|^2 \leq   \frac{ \alpha }{ 2 } \| d_t \Delta e_{\phi}^{n+1} \|^2+ \frac{1}{2\alpha} \| d_t e_{\phi}^{n+1} \|^2, \ \ \forall \ \alpha >0. 
\endaligned
\end{equation}
Combining \eqref{e_final_p_Step2_5} with the above equations \eqref{e_final_p_Step2_6}-\eqref{e_final_p_Step2_15} leads to
\begin{equation}\label{e_final_p_Step2_17}
\aligned
& \frac{ \| d_t e_{\phi}^{n+1} \|^2-\| d_t e_{\phi}^n \|^2 }{ 2\Delta t} + \frac{ \| d_t e_{\phi}^{n+1} - d_t e_{\phi}^n \|^2 }{ 2\Delta t} + \frac{M}{4} \| d_t e_{\mu}^{n+1} \|^2 \\
&+ \frac{M}{ 4 } \| d_t \Delta e_{\phi}^{n+1} \|^2 + \frac{M}{4} \gamma  \|  d_t \nabla e_{\phi}^{n+1} \|^2  \\
\leq & C \| d_t e_{\phi}^{n+1} \|^2 + \frac{M}{ 8 } \gamma \| d_t \nabla e_{\phi}^{n} \|^2 
+ C\| d_t e_r^{n+1}\|^2 +C \| d_t \phi^n \|^2 |e_r^n|^2  \\
&  +C \|d_te_{\phi} ^n \|^2 + C \| e_{\phi}^{n-1} \|^2 
+ C \| \textbf{u} \|_{W^{1,\infty}(0,T; H^1(\Omega))} ^2 ( \|e_{\phi}^n \|^2 +  \| \nabla e_{\phi}^n \|^2 ) \\
& + C \| e_{ \textbf{u} }^n \| ^2+ C \| \nabla e_{ \textbf{u} }^n \| ^2  + C \| e_{ \textbf{u} }^{n-1} \| ^2+ C \| \nabla e_{ \textbf{u} }^{n-1} \| ^2 + C \| d_t e_{ \textbf{u} }^{n} \|^2 \\
&  + C ( \| \phi \|_{W^{3,\infty}(0,T; L^2( \Omega ))} ^2+ \| \phi \|_{W^{2,\infty}(0,T;H^1(\Omega) ) } ^2 ) (\Delta t)^2 \\
& + C \| \textbf{u} \|_{W^{2,\infty}(0,T;H^1(\Omega) ) } ^2  (\Delta t)^2 +C |r|_{W^{2,\infty}(0,T)}^2  (\Delta t)^2 .
\endaligned
\end{equation}
Next we establish an approximation for $ \| d_te_r^{n+1} \| $ to continue the above error estimate. Multiplying both sides of \eqref{e_phi18} with $d_te_r^{n+1}$ leads to
\begin{equation}\label{e_final_p_Step2_18}
\aligned
\| d_te_r^{n+1} \|^2 = & \frac{ d_te_r^{n+1} }{ \sqrt{E_1( \phi^{n})+\delta}} ( F^{\prime}( \phi^{n}), d_t e_{\phi}^{n+1} ) - \frac{ d_te_r^{n+1} }{ \sqrt{E_1( \phi^{n})+\delta}} ( F^{\prime}( \phi^{n}), R_{\phi}^{n+1} ) \\
 & +\frac{ d_te_r^{n+1} }{ \sqrt{E_1( \phi^{n})+\delta}} ( F^{\prime}( \phi^{n})-F^{\prime}(\phi(t^{n+1})) ,  \frac{\partial \phi(t^{n+1}) }{\partial t} ) \\ 
 &+ (\frac{ d_te_r^{n+1} }{ \sqrt{E_1( \phi^{n})+\delta}} - \frac{ d_t e_r^{n+1}  }{ \sqrt{E_1(\phi(t^{n+1}))+\delta} } ) ( F^{\prime}(\phi(t^{n+1})) , \frac{\partial \phi(t^{n+1}) }{\partial t} ) \\
 &  +\frac{ d_te_r^{n+1} }{ \sqrt{E_1( \phi^{n})+\delta}} \left(  (\mu^{n+1}, \textbf{u}^n \cdot \nabla \phi^n) - ( \tilde{\textbf{u}}^{n+1}, \mu^n \nabla \phi^n ) \right) + R_{r}^{n+1}d_te_r^{n+1} .
\endaligned
\end{equation}
Recalling  \eqref{e_error1} , Lemma \ref{lem_phi_H2_boundness} and using Cauchy-Schwarz inequality, the first term on the right hand side of \eqref{e_final_p_Step2_18} can be estimated by
\begin{equation}\label{e_final_p_Step2_19}
\aligned
& \frac{ d_te_r^{n+1} }{ \sqrt{E_1( \phi^{n})+\delta}} ( F^{\prime}( \phi^{n}), d_t e_{\phi}^{n+1} )  \leq \frac {1} {10} | d_te_r^{n+1} |^2 + C \| d_t e_{\phi}^{n+1} \|^2. 
\endaligned
\end{equation}
The second term on the right hand side of \eqref{e_final_p_Step2_18} can be bounded by
\begin{equation}\label{e_final_p_Step2_20}
\aligned
&- \frac{ d_t e_r^{n+1} }{ \sqrt{E_1( \phi^{n})+\delta}} ( F^{\prime}( \phi^{n}), R_{\phi}^{n+1} ) \leq  \frac {1} {10} | d_te_r^{n+1} |^2 +  C  \| \phi \|_{W^{2,\infty}(0,T; L^2(\Omega))}  (\Delta t)^2.
\endaligned
\end{equation}
The third and fourth terms on the right hand side of \eqref{e_final_p_Step2_18} can be transformed into
\begin{equation}\label{e_final_p_Step2_21}
\aligned
\frac{ d_te_r^{n+1} }{ \sqrt{E_1( \phi^{n})+\delta}}& ( F^{\prime}( \phi^{n})-F^{\prime}(\phi(t^{n+1})) , \frac{\partial \phi(t^{n+1}) }{\partial t} ) \\
&+(\frac{ d_t e_r^{n+1} }{ \sqrt{E_1( \phi^{n})+\delta}} - \frac{ d_t e_r^{n+1}  }{ \sqrt{E_1(\phi(t^{n+1}))+\delta} } ) 
( F^{\prime}(\phi(t^{n+1})) , \frac{\partial \phi(t^{n+1}) }{\partial t} ) \\
\leq  & \frac {1} {10} | d_te_r^{n+1} |^2 + C |e_{\phi}^{n}|^2 + C  \| \phi \|_{W^{1,\infty}(0,T; L^2(\Omega))}  ^2 (\Delta t)^2.
\endaligned
\end{equation}
Recalling \eqref{e_error1}, \eqref{e_final_p_Step2_7_ADD2} and Lemma \ref{lem_phi_H2_boundness},  the second to last term on the right hand side of \eqref{e_final_p_Step2_18} can be estimated by
\begin{equation}\label{e_final_p_Step2_22}
\aligned
&\frac{ d_t e_r^{n+1} }{ \sqrt{E_1( \phi^{n})+\delta}} \left(  (\mu^{n+1}, \textbf{u}^n \cdot \nabla \phi^n) - ( \tilde{\textbf{u}}^{n+1}, \mu^n \nabla \phi^n ) \right) \\
\leq & \frac{ d_t e_r^{n+1} }{ \sqrt{E_1( \phi^{n})+\delta}} \left(  ( \mu^{n+1}, \textbf{u}^n \cdot \nabla \phi^n) - (\mu^{n}, \textbf{u}^n \cdot \nabla \phi^n)  \right) \\
&+ \frac{ d_t e_r^{n+1} }{ \sqrt{E_1( \phi^{n})+\delta}} \left( (\textbf{u}^n, \mu^n \nabla \phi^n ) - ( \tilde{\textbf{u}}^{n+1}, \mu^n \nabla \phi^n ) \right) \\
\leq & C | d_t e_r^{n+1} | \| \mu^{n+1}-\mu^n \|_{L^4} \| \textbf{u}^n \| \| \nabla \phi^n \|_{L^4}
+ C | d_t e_r^{n+1} | \| \textbf{u}^n - \tilde{\textbf{u}}^{n+1} \| \| \mu^n \|_{L^4} \| \nabla \phi^n \|_{L^4} \\
\leq & \frac {1} {10} | d_te_r^{n+1} |^2 + C ( \| e_{\mu}^n \|^2+  \| e_{\mu}^{n+1} \|^2 +\| \nabla e_{\mu}^n \|^2+  \| \nabla e_{\mu}^{n+1} \|^2 ) +C\| e_{ \textbf{u} }^n \|^2 +C \| \tilde{e}_{\textbf{u}}^{n+1} \|^2 \\
& + C \| \textbf{u} \|_{W^{1,\infty}(0,T; H^1(\Omega))} ^2 (\Delta t)^2+ C \| \mu \|_{W^{1,\infty}(0,T; H^1(\Omega))} ^2 (\Delta t)^2  .
\endaligned
\end{equation}
Using Cauchy-Schwarz inequality, the last term on the right hand side of \eqref{e_final_p_Step2_18} can be bounded by
\begin{equation}\label{e_final_p_Step2_23}
\aligned
& R_{r}^{n+1}d_te_r^{n+1}  \leq \frac {1} {10} | d_te_r^{n+1} |^2 + C \| r \|_{W^{2,\infty}(0,T)} ^2(\Delta t)^2 .
\endaligned
\end{equation}
Combining \eqref{e_final_p_Step2_18} with \eqref{e_final_p_Step2_19}-\eqref{e_final_p_Step2_23} results in
\begin{equation}\label{e_final_p_Step2_24}
\aligned
& \frac {1} {2} \| d_te_r^{n+1} \|^2 \leq   C \| d_t e_{\phi}^{n+1} \|^2 +C  |e_{\phi}^{n}|^2 
 + C ( \| e_{\mu}^n \|^2+  \| e_{\mu}^{n+1} \|^2 +\| \nabla e_{\mu}^n \|^2+  \| \nabla e_{\mu}^{n+1} \|^2 ) \\
 & \ \ \ \ \ \ +C\| e_{ \textbf{u} }^n \|^2 +C \| \tilde{e}_{\textbf{u}}^{n+1} \|^2
+  C  \| \phi \|_{W^{2,\infty}(0,T; L^2(\Omega))}  (\Delta t)^2 \\
 & \ \ \ \ \ \ + C \| \mu \|_{W^{1,\infty}(0,T; H^1(\Omega))} ^2 (\Delta t)^2 
+C \| r \|_{W^{2,\infty}(0,T)} ^2(\Delta t)^2 \\
 & \ \ \ \ \ \ 
+ C \| \textbf{u} \|_{W^{1,\infty}(0,T; H^1(\Omega))} ^2 (\Delta t)^2. 
\endaligned
\end{equation}
Combining \eqref{e_final_p_Step2_17} with \eqref{e_final_p_Step2_24} gives
\begin{equation}\label{e_final_p_Step2_25}
\aligned
& \frac{ \| d_t e_{\phi}^{n+1} \|^2-\| d_t e_{\phi}^n \|^2 }{ 2\Delta t} + \frac{ \| d_t e_{\phi}^{n+1} - d_t e_{\phi}^n \|^2 }{ 2\Delta t} + \frac{M}{4} \| d_t e_{\mu}^{n+1} \|^2 \\
&+ \frac{M}{ 4 } \| d_t \Delta e_{\phi}^{n+1} \|^2 + \frac{M}{4} \gamma  \|  d_t \nabla e_{\phi}^{n+1} \|^2  \\
\leq & C \| d_t e_{\phi}^{n+1} \|^2 + \frac{M}{ 8 }\gamma \| d_t \nabla e_{\phi}^n \|^2 
 +C \| \tilde{e}_{\textbf{u}}^{n+1} \|^2  +C \| d_t \phi^n \|^2 |e_r^n|^2  \\
&  +C \|d_te_{\phi} ^n \|^2 + C \| e_{\phi}^{n-1} \|^2 
+ C \| \textbf{u} \|_{W^{1,\infty}(0,T; H^1(\Omega))} ^2 ( \|e_{\phi}^n \|^2 +  \| \nabla e_{\phi}^n \|^2 ) \\
& + C \| e_{ \textbf{u} }^n \| ^2+ C \| \nabla e_{ \textbf{u} }^n \| ^2  + C \| e_{ \textbf{u} }^{n-1} \| ^2+ C \| \nabla e_{ \textbf{u} }^{n-1} \| ^2 + C \| d_t e_{ \textbf{u} }^{n} \|^2 \\
& + C ( \| e_{\mu}^n \|^2+  \| e_{\mu}^{n+1} \|^2 +\| \nabla e_{\mu}^n \|^2+  \| \nabla e_{\mu}^{n+1} \|^2 )  \\
&  + C ( \| \phi \|_{W^{3,\infty}(0,T; L^2( \Omega ))} ^2+ \| \phi \|_{W^{2,\infty}(0,T;H^1(\Omega) ) } ^2 ) (\Delta t)^2 \\
& + C \| \textbf{u} \|_{W^{2,\infty}(0,T;H^1(\Omega) ) } ^2  (\Delta t)^2  +C |r|_{W^{2,\infty}(0,T)}^2  (\Delta t)^2 \\
& +  C \| \mu \|_{W^{1,\infty}(0,T; H^1(\Omega))} ^2 (\Delta t)^2.
\endaligned
\end{equation}
Multiplying \eqref{e_final_p_Step2_25} by $2 \Delta t$ and summing up for $n$ from $1$ to $m$, applying the discrete Gronwall lemma \ref{lem: gronwall1} and Recalling Theorem \ref{thm: error_estimate_u_phi}, we can obtain
\begin{equation}\label{e_final_p_Step2_26}
\aligned
& \| d_t e_{\phi}^{m+1} \|^2 + \sum\limits_{n=1}^{m} \| d_t e_{\phi}^{n+1} - d_t e_{\phi}^n \|^2 +  \Delta t\sum\limits_{n=1}^{m} \| d_t e_{\mu}^{n+1} \|^2 \\
& +  \Delta t\sum\limits_{n=1}^{m} \| d_t \Delta e_{\phi}^{n+1} \|^2 +\Delta t \|  d_t \nabla e_{\phi}^{m+1} \|^2  \\
\leq & \| d_t e_{\phi}^{1} \|^2 + \Delta t \|  d_t \nabla e_{\phi}^{1} \|^2 
+ C  \Delta t\sum\limits_{n=1}^{m}  \| d_t e_{ \textbf{u} }^{n} \|^2 + C (\Delta t)^2 .
\endaligned
\end{equation}
It remains to estimate $ \| d_t e_{\phi}^{1} \|^2 $ and $  \|  d_t \nabla e_{\phi}^{1} \|^2 $. 
Substituting \eqref{e_phi5} into \eqref{e_phi2} with $n=0$ leads to
\begin{equation}\label{e_final_p_Step2_27}
\aligned
& e_{\phi}^{1} + \Delta t M \Delta^2 e_{\phi}^1- M \gamma \Delta t \Delta e_{\phi}^1 \\
=& M\Delta t \frac{e_r^{n+1}}{ \sqrt{E_1( \phi^{n})+\delta} } \Delta F^{\prime}( \phi^{n}) +  M \Delta t \Delta  E_{F}^{n+1} + \Delta t R_{\phi}^{n+1} + \Delta t E_{N}^{n+1}.
\endaligned
\end{equation}
Taking the inner product of \eqref{e_final_p_Step2_27} with $ e_{\phi}^{1} $, we have
\begin{equation}\label{e_final_p_Step2_28}
\aligned
& \| e_{\phi}^{1} \|^2 + \Delta t M \| \Delta e_{\phi}^1\|^2 + M \gamma \Delta t \| \nabla e_{\phi}^1 \|^2 \\
=& M\Delta t \frac{e_r^{1}}{ \sqrt{E_1( \phi^{0})+\delta} } ( \Delta F^{\prime}( \phi^{0}), e_{\phi}^{1} ) +  M \Delta t  (\Delta  E_{F}^{1} , e_{\phi}^{1} ) \\
&+ \Delta t ( R_{\phi}^{1}, e_{\phi}^{1} ) + \Delta t ( E_{N}^{1}, e_{\phi}^{1} ) \\
\leq & \frac{ 1 }{ 2 } \| e_{\phi}^{1} \|^2 + C (\Delta t)^2 | e_r^1 |^2 + C (\Delta t)^4. 
\endaligned
\end{equation}
Thus we have 
\begin{equation}\label{e_final_p_Step2_32}
\aligned
& \| e_{\phi}^{1} \|^2 + \Delta t M \| \Delta e_{\phi}^1\|^2 + M \gamma \Delta t \| \nabla e_{\phi}^1 \|^2  \leq  C (\Delta t)^4, 
\endaligned
\end{equation}
which implies the fact that 
\begin{equation}\label{e_final_p_Step2_33}
\aligned
& \| d_t e_{\phi}^{1} \|^2 + \Delta t M \| \Delta d_t e_{\phi}^1\|^2 + M \gamma \Delta t \| \nabla d_t e_{\phi}^1 \|^2  \leq  C (\Delta t)^2.
\endaligned
\end{equation}
Substituting \eqref{e_final_p_Step2_24} and \eqref{e_final_p_Step2_33} into \eqref{e_final_p_Step2_26} results in
\begin{equation}\label{e_final_p_Step2_34}
\aligned
&\| d_t e_{\phi}^{m+1} \|^2 + \sum\limits_{n=1}^{m} \| d_t e_{\phi}^{n+1} - d_t e_{\phi}^n \|^2 +  \Delta t\sum\limits_{n=1}^{m} \| d_t e_{\mu}^{n+1} \|^2 \\
& + \Delta t\sum\limits_{n=1}^{m} \| d_te_r^{n+1} \|^2 +  \Delta t\sum\limits_{n=1}^{m} \| d_t \Delta e_{\phi}^{n+1} \|^2 +\Delta t \|  d_t \nabla e_{\phi}^{m+1} \|^2  \\
& \ \ \ \ \ \ 
\leq   C  \Delta t\sum\limits_{n=1}^{m}  \| d_t e_{ \textbf{u} }^{n} \|^2 + C (\Delta t)^2 .
\endaligned
\end{equation}

\noindent{\bf Step 3.} 
Next we establish an estimate on $ \| d_t e_{\textbf{u}}^{n+1} \|$. 

Adding \eqref{e_u2} and \eqref{e_u4} results in 
\begin{equation}\label{e_final_p_Step3_1}
\aligned
&\frac{ e_{\textbf{u}}^{n+1}-e_{\textbf{u}}^n}{\Delta t}-\nu\Delta\tilde{e}_{\textbf{u}}^{n+1}
= \exp( \frac{t^{n+1}}{T} )q(t^{n+1}) (\textbf{u}(t^{n+1})\cdot \nabla)\textbf{u}(t^{n+1})\\
&-\exp( \frac{t^{n+1}}{T} )q^{n+1} \textbf{u}^{n}\cdot \nabla\textbf{u}^{n}
-\nabla e_p^{n+1}  +\frac{r^{n+1}}{\sqrt{E_1( \phi^{n} )+\delta}}  \mu^n \nabla \phi^n \\
& - \frac{r(t^{n+1}) }{\sqrt{E_1(\phi(t^{n+1}) )+\delta}} \mu(t^{n+1}) \nabla \phi(t^{n+1} ) 
+\textbf{R}_{\textbf{u}}^{n+1}.
\endaligned
\end{equation}
Then taking the  difference of two consecutive steps in \eqref{e_final_p_Step3_1}, we have
\begin{equation}\label{e_final_p_Step3_3}
\aligned
& d_{tt} e_{\textbf{u} }^{n+1}-\nu\Delta d_t\tilde{e}_{\textbf{u}}^{n+1}
= d_t\textbf{R}_{\textbf{u}}^{n+1}-\nabla d_te_p^{n+1} \\ 
& + d_t \left( \exp( \frac{t^{n+1}}{T} )q(t^{n+1}) (\textbf{u}(t^{n+1})\cdot \nabla)\textbf{u}(t^{n+1})- \exp( \frac{t^{n+1}}{T} )q^{n+1} \textbf{u}^{n}\cdot \nabla\textbf{u}^{n} \right) \\
& + d_t \left( \frac{r^{n+1}}{\sqrt{E_1( \phi^{n} )+\delta}}  \mu^n \nabla \phi^n -  \frac{r(t^{n+1}) }{\sqrt{E_1(\phi(t^{n+1}) )+\delta}} \mu(t^{n+1}) \nabla \phi(t^{n+1} ) \right), \ n \geq 1 .
\endaligned
\end{equation}
Taking the inner product of \eqref{e_final_p_Step3_3} with $d_t \tilde{e}_{\textbf{u} }^{n+1}$, we find
\begin{equation}\label{e_final_p_Step3_4}
\aligned
& ( d_{tt} e_{\textbf{u} }^{n+1}, d_t \tilde{e}_{\textbf{u} }^{n+1}) +\nu \|\nabla d_t\tilde{e}_{\textbf{u}}^{n+1}\|^2 = ( d_t\textbf{R}_{\textbf{u}}^{n+1}, d_t \tilde{e}_{\textbf{u} }^{n+1} )
-(\nabla d_te_p^{n+1},d_t \tilde{e}_{\textbf{u} }^{n+1}) \\
& + \left(  d_t \left( \exp( \frac{t^{n+1}}{T} )q(t^{n+1}) (\textbf{u}(t^{n+1})\cdot \nabla)\textbf{u}(t^{n+1})- \exp( \frac{t^{n+1}}{T} )q^{n+1} \textbf{u}^{n}\cdot \nabla\textbf{u}^{n} \right), d_t \tilde{e}_{\textbf{u} }^{n+1}  \right) \\
& + \left( d_t \left( \frac{r^{n+1}}{\sqrt{E_1( \phi^{n} )+\delta}}  \mu^n \nabla \phi^n -  \frac{r(t^{n+1}) }{\sqrt{E_1(\phi(t^{n+1}) )+\delta}} \mu(t^{n+1}) \nabla \phi(t^{n+1} ) \right), 
d_t \tilde{e}_{\textbf{u} }^{n+1} \right) .
\endaligned
\end{equation}
The first term on the left hand side of \eqref{e_final_p_Step3_4} can be recast into
\begin{equation}\label{e_final_p_Step3_5}
\aligned
&(d_{tt} e_{\textbf{u} }^{n+1},d_t \tilde{e}_{\textbf{u} }^{n+1})= 
\frac{ \|d_te_{\textbf{u} }^{n+1}\|^2- \|d_te_{\textbf{u} }^{n}\|^2}{2\Delta t}
+\frac{ \| d_te_{\textbf{u} }^{n+1}-d_te_{\textbf{u} }^n\|^2}{2\Delta t}.
\endaligned
\end{equation}
Next we bound the first four terms on the right hand side as follows by using the similar procedure as in \cite{Li2020new}.
\begin{equation}\label{e_final_p_Step3_6}
\aligned
&(d_t\textbf{R}_{\textbf{u}}^{n+1},d_t \tilde{e}_{\textbf{u} }^{n+1}) \leq \frac{\nu}{6} \|\nabla d_t\tilde{e}_{\textbf{u}}^{n+1}\|^2+C \| \textbf{u} \|^2_{W^{3,\infty}(0,T; L^2(\Omega))} 
( \Delta t )^2.
\endaligned
\end{equation}
For the second term on the right hand side of \eqref{e_final_p_Step3_4}, we have
\begin{equation}\label{e_final_p_Step3_7}
\aligned
&-(\nabla d_te_p^{n+1},d_t \tilde{e}_{\textbf{u} }^{n+1}) =-(\nabla d_te_p^{n},d_t \tilde{e}_{\textbf{u} }^{n+1})  - (\nabla ( d_te_p^{n+1}-d_te_p^{n} ),d_t \tilde{e}_{\textbf{u} }^{n+1}) .
\endaligned
\end{equation}
Since we can derive from \eqref{e_u4} that
\begin{equation}\label{e_final_p_Step3_8}
\aligned
&d_t\tilde{e}_{\textbf{u}}^{n+1}=d_te_{\textbf{u}}^{n+1} +\nabla (p^{n+1}-2p^n+p^{n-1}).
\endaligned
\end{equation}
The first term on the right hand of \eqref{e_final_p_Step3_7} can be bounded by
\begin{equation}\label{e_final_p_Step3_9}
\aligned
-(\nabla d_te_p^{n},d_t \tilde{e}_{\textbf{u} }^{n+1}) =&-(\nabla d_te_p^{n}, \nabla (p^{n+1}-2p^n+p^{n-1}) ) \\
=&-\Delta t(\nabla d_te_p^{n}, \nabla ( d_te_p^{n+1}-d_te_p^n ) )-  \left( \nabla d_te_p^{n}, \nabla (p(t^{n+1})-2p(t^n)+p(t^{n-1}) ) \right) \\
\leq &-\frac {\Delta t} 2 ( \| \nabla d_te_p^{n+1} \|^2- \| \nabla d_te_p^n \|^2- \|\nabla d_te_p^{n+1}-\nabla d_te_p^n \|^2) \\
&+ (\Delta t)^2 \|\nabla d_te_p^{n} \|^2+C \| p\|^2_{W^{2,\infty}(0,T; H^1(\Omega ))} (\Delta t)^2 .
\endaligned
\end{equation}
The second term on the right hand of \eqref{e_final_p_Step3_7} can be transformed into
\begin{equation}\label{e_final_p_Step3_10}
\aligned
-(\nabla ( d_te_p^{n+1}&-d_te_p^{n} ),d_t \tilde{e}_{\textbf{u} }^{n+1}) =- \left( \nabla  ( d_te_p^{n+1}-d_te_p^{n} ) , \nabla (p^{n+1}-2p^n+p^{n-1}) \right) \\
=&-  \left(\nabla ( d_te_p^{n+1}-d_te_p^{n} ),  \nabla (p(t^{n+1})-2p(t^n)+p(t^{n-1}) ) \right)  \\
&- \Delta t \left( \nabla ( d_te_p^{n+1}-d_te_p^{n} ), \nabla ( d_te_p^{n+1}-d_te_p^n ) \right) \\
\leq & -\frac {\Delta t} 2 \|\nabla d_te_p^{n+1}-\nabla d_te_p^n \|^2 +C \| p\|^2_{W^{2,\infty}(0,T; H^1(\Omega ))} (\Delta t)^2.
\endaligned
\end{equation}
Recalling Theorem \ref{thm: error_estimate_u_phi}, the third term on the right hand side of \eqref{e_final_p_Step3_4} can be bounded by
\begin{equation}\label{e_final_p_Step3_11}
\aligned
&\left(  d_t \left( \exp( \frac{t^{n+1}}{T} )q(t^{n+1}) (\textbf{u}(t^{n+1})\cdot \nabla)\textbf{u}(t^{n+1})- \exp( \frac{t^{n+1}}{T} )q^{n+1} \textbf{u}^{n}\cdot \nabla\textbf{u}^{n} \right), d_t \tilde{e}_{\textbf{u} }^{n+1}  \right) \\
\leq & \frac{\nu}{6} \|\nabla d_t\tilde{e}_{\textbf{u}}^{n+1}\|^2+C\| d_te_{\textbf{u}}^n\|^2+C \| e_{\textbf{u}}^n \|^2+C \| e_{\textbf{u}}^{n-1} \|^2 \\
& + C \| e_{\textbf{u}}^n \|^2_1 + C|d_te_q^{n+1}|^2+C|e_q^{n+1}|^2+C(\Delta t)^2 ,
\endaligned
\end{equation}
where we used the fact that 
$$\|\nabla d_t\tilde{e}_{\textbf{u}}^{n}\|^2 \leq  (\Delta t)^{-2} \|\nabla \tilde{e}_{\textbf{u}}^{n}\|^2 \leq C (\Delta t)^{-1}, \ \forall \ 1\leq n\leq N.$$
Next we  concentrate on the last term on the right hand side of \eqref{e_final_p_Step3_4}. Since
\begin{equation}\label{e_final_p_Step3_12}
\aligned
& d_t \left( \frac{r^{n+1}}{\sqrt{E_1( \phi^{n} )+\delta}}  \mu^n \nabla \phi^n -  \frac{r(t^{n+1}) }{\sqrt{E_1(\phi(t^{n+1}) )+\delta}} \mu(t^{n+1}) \nabla \phi(t^{n+1} ) \right)
 \\
=& d_t \left( \frac{r^{n+1}}{\sqrt{E_1( \phi^{n} )+\delta}}  (\mu^n \nabla \phi^n - \mu(t^{n+1}) \nabla \phi(t^{n+1} ) ) \right) \\
& + d_t \left( ( \frac{r^{n+1}}{\sqrt{E_1( \phi^{n} )+\delta}}- \frac{r(t^{n+1}) }{\sqrt{E_1(\phi(t^{n+1}) )+\delta}} )  \mu(t^{n+1}) \nabla \phi(t^{n+1} ) \right) .
\endaligned
\end{equation}
Using \eqref{e_final_p_Step2_7}, \eqref{e_final_p_Step2_7_ADD} and \eqref{e_final_p_Step2_14} , the last term on the right hand side of \eqref{e_final_p_Step3_4} can be bounded by
\begin{equation}\label{e_final_p_Step3_13}
\aligned
&  \left( d_t \left( \frac{r^{n+1}}{\sqrt{E_1( \phi^{n} )+\delta}}  \mu^n \nabla \phi^n -  \frac{r(t^{n+1}) }{\sqrt{E_1(\phi(t^{n+1}) )+\delta}} \mu(t^{n+1}) \nabla \phi(t^{n+1} ) \right), 
d_t \tilde{e}_{\textbf{u} }^{n+1} \right) \\
= &   \left(  d_t \left( \frac{r^{n+1}}{\sqrt{E_1( \phi^{n} )+\delta}}  (\mu^n \nabla \phi^n - \mu(t^{n+1}) \nabla \phi(t^{n+1} ) ) \right) , d_t \tilde{e}_{\textbf{u} }^{n+1} \right) \\
&+    \left(  d_t \left( ( \frac{r^{n+1}}{\sqrt{E_1( \phi^{n} )+\delta}}- \frac{r(t^{n+1}) }{\sqrt{E_1(\phi(t^{n+1}) )+\delta}} )  \mu(t^{n+1}) \nabla \phi(t^{n+1} ) \right), d_t \tilde{e}_{\textbf{u} }^{n+1} \right) \\
\leq &  \frac{\nu}{6} \|\nabla d_t\tilde{e}_{\textbf{u}}^{n+1}\|^2+ C \| d_t e_{\mu}^n \|^2 + C \| d_t \Delta e_{\phi}^n \|^2 + C \| d_t \nabla e_{\phi}^n \|^2   \\
& +C | d_t e_r^{n+1} |^2 + C | e_r^n |^2 + C \| e_{\phi} ^n \|^2+C \| e_{\phi} ^{n-1} \|^2 + C (\Delta t)^2 .
\endaligned
\end{equation}
Combining \eqref{e_final_p_Step3_4} with \eqref{e_final_p_Step3_5}-\eqref{e_final_p_Step3_13} gives
\begin{equation}\label{e_final_p_Step3_14}
\aligned
&\frac{ \|d_te_{\textbf{u} }^{n+1}\|^2- \|d_te_{\textbf{u} }^{n}\|^2}{2\Delta t}
+\frac{ \| d_te_{\textbf{u} }^{n+1}-d_te_{\textbf{u} }^n\|^2}{2\Delta t}+\frac \nu 2 \|\nabla d_t\tilde{e}_{\textbf{u}}^{n+1}\|^2 \\
& + \frac {\Delta t} 2 ( \| \nabla d_te_p^{n+1} \|^2- \| \nabla d_te_p^n \|^2 ) \\
\leq & (\Delta t)^2 \|\nabla d_te_p^{n} \|^2 +C\| d_te_{\textbf{u}}^n\|^2+C \| e_{\textbf{u}}^{n-1} \|^2  + C \| e_{\textbf{u}}^n \|^2_1 + C \| d_t \nabla e_{\phi}^n \|^2 \\
&+ C |d_te_q^{n+1}|^2+C|e_q^{n+1}|^2 + C \| d_t e_{\mu}^n \|^2 + C \| d_t \Delta e_{\phi}^n \|^2  \\
& +C | d_t e_r^{n+1} |^2 + C | e_r^n |^2 + C\| e_{\phi} ^n \|^2+ C\| e_{\phi} ^{n-1} \|^2 +C ( \Delta t )^2 .
\endaligned
\end{equation}
Multiplying \eqref{e_final_p_Step3_14} by $2 \Delta t$, summing up for $n$ from $1$ to $m$, and using \eqref{e_final_p8} and \eqref{e_final_p_Step2_34}, we have
\begin{equation}\label{e_final_p_Step3_15}
\aligned
 \|d_te_{\textbf{u} }^{m+1}\|^2& + \nu \Delta t \sum\limits_{n=1}^{m}  \|\nabla d_t\tilde{e}_{\textbf{u}}^{n+1}\|^2 + (\Delta t)^2 \| \nabla d_te_p^{m+1} \|^2 \\
\leq & \|d_te_{\textbf{u} }^{1}\|^2 + (\Delta t)^2 \| \nabla d_te_p^{1} \|^2 + (\Delta t)^3 \sum\limits_{n=1}^{m} \|\nabla d_te_p^{n} \|^2  + C  \Delta t \sum\limits_{n=1}^{m}  \| d_te_{\textbf{u}}^n\|^2  \\
&  +  C  \Delta t \sum\limits_{n=1}^{m}  |d_te_q^{n+1}|^2   + C  \Delta t \sum\limits_{n=1}^{m} \| d_t \nabla e_{\phi}^n \|^2 +  C  \Delta t \sum\limits_{n=1}^{m} \| d_t e_{\mu}^n \|^2 \\
& +C  \Delta t \sum\limits_{n=1}^{m} \| d_t \Delta e_{\phi}^n \|^2 + C  \Delta t \sum\limits_{n=1}^{m}  |d_te_r^{n+1}|^2  + C ( \Delta t )^2 \\
\leq  & \|d_te_{\textbf{u} }^{1}\|^2 + (\Delta t)^2 \| \nabla d_te_p^{1} \|^2 + (\Delta t)^3 \sum\limits_{n=1}^{m} \|\nabla d_te_p^{n} \|^2 \\
& + C  \Delta t\sum\limits_{n=1}^{m}  \| d_t e_{ \textbf{u} }^{n} \|^2 + C (\Delta t)^2 .
\endaligned
\end{equation}
Next we  estimate the first two terms on the right hand side of \eqref{e_final_p_Step3_15}. Recalling \eqref{e_u2}, we can obtain
\begin{equation}\label{e_final_p_Step3_16}
\aligned
 \tilde{e}_{\textbf{u}}^{1}-\nu\Delta t \tilde{e}_{\textbf{u}}^1 =&\Delta t \exp( \frac{t^{1}}{T} )q(t^{1}) (\textbf{u}(t^{1})\cdot \nabla)\textbf{u}(t^{1}) \\
& \ \ - \Delta t \exp( \frac{t^{1}}{T} )q^{1} \textbf{u}^{0}\cdot \nabla\textbf{u}^{0}
- \Delta t \nabla (p^0-p(t^{1})) +\Delta t \frac{r^{1}}{\sqrt{E_1( \phi^{0} )+\delta}}  \mu^0 \nabla \phi^0 \\
& \ \ - \Delta t \frac{r(t^{1}) }{\sqrt{E_1(\phi(t^{1}) )+\delta}} \mu(t^{1}) \nabla \phi(t^{1} ) 
+\Delta t \textbf{R}_{\textbf{u}}^{1}.
\endaligned
\end{equation}
Taking the inner product of \eqref{e_final_p_Step3_16} with $\tilde{e}_{\textbf{u}}^1$ leads to
\begin{equation}\label{e_final_p_Step3_17}
\aligned
 \| \tilde{e}_{\textbf{u}}^{1} \|^2&+\nu\Delta t \| \nabla \tilde{e}_{\textbf{u}}^1 \|^2 \\
 =& \Delta t\left( \exp( \frac{t^{1}}{T} )q(t^{1})  (\textbf{u}(t^{1})\cdot \nabla)\textbf{u}(t^{1})- \exp( \frac{t^{1}}{T} )q^{1}  \textbf{u}^{0}\cdot \nabla\textbf{u}^{0}, \tilde{e}_{\textbf{u}}^1 \right) \\
 & + \Delta t \left( \frac{r^{1}}{\sqrt{E_1( \phi^{0} )+\delta}}  \mu^0 \nabla \phi^0 - \frac{r(t^{1}) }{\sqrt{E_1(\phi(t^{1}) )+\delta}} \mu(t^{1}) \nabla \phi(t^{1} ) , \tilde{e}_{\textbf{u}}^1 \right) \\
 &-\Delta t (\nabla (p(t^0)-p(t^{1})), \tilde{e}_{\textbf{u}}^1)
 +\Delta t (\textbf{R}_{\textbf{u}}^{1}, \tilde{e}_{\textbf{u}}^1) \\
 \leq &\frac{1}{2}\| \tilde{e}_{\textbf{u}}^{1} \|^2+C(\Delta t)^4.
\endaligned
\end{equation}
Hence we can obtain 
$$\|d_te_{\textbf{u} }^{1}\|^2 \leq  \|d_t\tilde{e}_{\textbf{u} }^{1}\|^2 =(\Delta t)^{-2}\| \tilde{e}_{\textbf{u}}^{1} \|^2\leq C(\Delta t)^2. $$
Using \eqref{e_u4} with $n=1$ results in 
\begin{equation}\label{e_final_p_Step3_18}
\aligned
& (\Delta t)^2\|\nabla d_te_p^{1} \|^2 \leq (\Delta t)^{-2}( \| e_{\textbf{u}}^{1} \|^2+ \| \tilde{e}_{\textbf{u}}^{1} \|^2) + (\Delta t)^2\| \nabla d_tp(t^{1}) \|^2 \leq C (\Delta t)^2.
\endaligned
\end{equation}
Substituting the above estimates into  \eqref{e_final_p_Step3_15} and applying the discrete Gronwall lemma \ref{lem: gronwall1}, we finally obtain  
\begin{equation}\label{e_final_p_Step3_19}
\aligned
 &\|d_te_{\textbf{u} }^{m+1}\|^2 + \nu \Delta t \sum\limits_{n=1}^{m}  \|\nabla d_t\tilde{e}_{\textbf{u}}^{n+1}\|^2 +(\Delta t)^2 \|\nabla d_te_p^{m+1} \|^2 \leq 
C(\Delta t)^2.
\endaligned
\end{equation}

We are now in position to prove the pressure estimate.  Taking the inner product of \eqref{e_final_p_Step3_1} with $\textbf{v}\in \textbf{H}^1_0(\Omega)$, we have
\begin{equation}\label{e_final_p_Step3_20}
\aligned
(\nabla e_p^{n+1},\textbf{v})=&-(\frac{e_{\textbf{u}}^{n+1}-e_{\textbf{u}}^n}{\Delta t},\textbf{v})+\nu(\Delta\tilde{e}_{\textbf{u}}^{n+1},\textbf{v})+(\textbf{R}_{\textbf{u}}^{n+1},\textbf{v})\\
&+ \left( \exp( \frac{t^{n+1}}{T} )q(t^{n+1}) (\textbf{u}(t^{n+1})\cdot \nabla)\textbf{u}(t^{n+1}) -\exp( \frac{t^{n+1}}{T} )q^{n+1} \textbf{u}^{n}\cdot \nabla\textbf{u}^{n} ,  \textbf{v}\right) \\
& +  \left( \frac{r^{n+1}}{\sqrt{E_1( \phi^{n} )+\delta}}  \mu^n \nabla \phi^n - \frac{r(t^{n+1}) }{\sqrt{E_1(\phi(t^{n+1}) )+\delta}} \mu(t^{n+1}) \nabla \phi(t^{n+1} )  ,  \textbf{v}\right) .
\endaligned
\end{equation}
For the second to last term on the right hand side of \eqref{e_final_p_Step3_20}, we have  
\begin{equation}\label{e_final_p_Step3_21}
\aligned
 &\left( \frac{q(t^{n+1})}{ \exp(  -\frac{ t^{n+1} } {T} ) } \textbf{u}(t^{n+1})\cdot \nabla \textbf{u}(t^{n+1})-\frac{q^{n+1}}{ \exp(  -\frac{ t^{n+1} } {T} ) } \textbf{u}^{n}\cdot \nabla\textbf{u}^{n}, \textbf{v} \right) \\
=&\frac{q(t^{n+1}) }{ \exp(  -\frac{t^{n+1}}{T} ) }\left( ( \textbf{u}(t^{n+1})-\textbf{u}^{n} )\cdot \nabla \textbf{u}(t^{n+1}),  \textbf{v} \right)\\
&+\frac{q(t^{n+1}) }{ \exp(  -\frac{t^{n+1}}{T} ) }\left(  \textbf{u}^{n}\cdot \nabla (\textbf{u}(t^{n+1})-\textbf{u}^{n} ),  \textbf{v} \right)-\frac{e_q^{n+1}}{ \exp(  -\frac{t^{n+1}}{T} ) }\left((\textbf{u}^n \cdot \nabla)\textbf{u}^n,  \textbf{v} \right) \\
\leq & C(\|e_{ \textbf{u} }^{n}\| +\| \nabla \tilde{e}_{\textbf{u}}^n\| + |e_q^{n+1}| + \| \textbf{u} \|_{W^{1,\infty}(0,T; H^1(\Omega))} \Delta t )  \|\nabla \textbf{v} \| .
\endaligned
\end{equation}
By using the similar procedure in \eqref{e_u11}, the last term on the right hand side of \eqref{e_final_p_Step3_20} can be transformed into
\begin{equation}\label{e_final_p_Step3_22}
\aligned
&  \left( \frac{r^{n+1}}{\sqrt{E_1( \phi^{n} )+\delta}}  \mu^n \nabla \phi^n - \frac{r(t^{n+1}) }{\sqrt{E_1(\phi(t^{n+1}) )+\delta}} \mu(t^{n+1}) \nabla \phi(t^{n+1} )  ,  \textbf{v}\right) \\
=& \left(  \frac{r^{n+1}}{\sqrt{E_1( \phi^{n} )+\delta}}  (\mu^n \nabla \phi^n - \mu(t^{n+1}) \nabla \phi(t^{n+1} ) ) ,   \textbf{v} \right) \\
& + \left( ( \frac{r^{n+1}}{\sqrt{E_1( \phi^{n} )+\delta}}- \frac{r(t^{n+1}) }{\sqrt{E_1(\phi(t^{n+1}) )+\delta}} )  \mu(t^{n+1}) \nabla \phi(t^{n+1} ) ,  \textbf{v} \right) \\
\leq & C ( \| e_{\mu}^n \| + \| \nabla e_{\mu}^n \|+ \| \nabla e_{\phi}^n \| +  \| e_{\phi}^n \| + |e_r^{n+1}| + |r|_{W^{1,\infty}(0,T)} )  \|\nabla \textbf{v} \| \\
& + C (  \| \mu \|_{W^{1,\infty}(0,T; H^1(\Omega)) } \Delta t + \| \phi \|_{W^{1,\infty}(0,T; H^1 (\Omega)) } \Delta t )  \|\nabla \textbf{v} \| .
\endaligned
\end{equation}
Thus by using the above estimates and the fact that 
\begin{equation*}
\| e_p^{n+1} \|_{L^2(\Omega)/\mathbb{R}} \leq \sup_{\textbf{v} \in \textbf{H}^1_0(\Omega)} \frac{
(\nabla e_p^{n+1},\textbf{v}) }{ \|\nabla \textbf{v} \| }, 
\end{equation*}
we finally  derive that
\begin{equation*}\label{e_final_p_Step3_24}
\aligned
&\Delta t \sum\limits_{n=0}^m \|e_p^{n+1}\|^2_{L^2(\Omega)/\mathbb{R}} \leq 
C\Delta t \sum\limits_{n=0}^m \left( \| d_te_{\textbf{u}}^{n+1}\|^2+ \| \nabla \tilde{e}_{\textbf{u}}^{n+1} \|^2 +\|e_{\textbf{u}}^n\|^2 \right. \\
&\ \ \ \ \ 
\left.  + |e_q^{n+1}|^2 + \| e_{\mu}^n \|^2 + \| \nabla e_{\mu}^n \|^2 + \| \nabla e_{\phi}^n \|^2 +  \| e_{\phi}^n \|^2 + |e_r^{n+1}|^2 \right) \\
&\ \ \ \ \  + C ( \| \textbf{u} \|_{W^{2,\infty}(0,T; L^2(\Omega))}^2+  \| \mu \|_{W^{1,\infty}(0,T; H^1(\Omega)) } ^2 ) (\Delta t)^2 \\
&\ \ \ \ \  + C \| \phi \|_{W^{1,\infty}(0,T; H^1 (\Omega)) } ^2  (\Delta t)^2.
\endaligned
\end{equation*}
The proof is complete.
\end{proof}
\end{appendix}
\bibliographystyle{siamplain}
\bibliography{CHNS_Decoulped_SAV}

\end{document}